\documentclass[times,sort&compress,3p]{elsarticle}
\journal{Journal of Multivariate Analysis}
\usepackage[labelfont=bf]{caption}

\usepackage{amsmath,amsfonts,amssymb,amsthm,booktabs,color,epsfig,graphicx,hyperref,url}

\usepackage{url}
\usepackage{ifthen}
\usepackage{graphicx,colordvi,psfrag}
\usepackage{amsmath,amsfonts,amssymb,amsthm}
\usepackage{calc,pstricks, pgf, xcolor}
\usepackage{bbding}
\usepackage{bbm}
\usepackage{dsfont}
\usepackage{centernot}
\usepackage{bm}
\usepackage{subcaption}
\usepackage{enumitem}
\theoremstyle{plain}% Theorem-like structures provided by amsthm.sty
\newtheorem{theorem}{Theorem}

\newtheorem{lemma}{Lemma}
\newtheorem{corollary}{Corollary}

\theoremstyle{definition}
\newtheorem{definition}{Definition}
\newtheorem{remark}{Remark}

\newcommand{\x}{{\bf x}}

\newcommand{\be}{\begin{equation*}}
\newcommand{\ee}{\end{equation*}}

\newcommand{\E}{{\mathbb{E}}}

\newcommand{\y}{{\bf y}}
\newcommand{\Y}{{\bf Y}}

\newcommand{\K}{{\bf K}}

\newcommand{\w}{{\bf w}}

\newcommand{\R}{\mathbb{R}}
\newcommand{\Sp}{\mathbb{S}^{p-1}}

\newcommand{\Bc}{\ensuremath{\mathcal{B}}}
\newcommand{\Ec}{\ensuremath{\mathcal{E}}}

\newcommand{\Nc}{\ensuremath{\mathcal{N}}}

\newcommand{\Ic}{\ensuremath{\mathcal{I}}}
\newcommand{\tr}{\textrm{Tr}}
\newcommand{\Tr}{\tr}
\newcommand{\norm}[1]{\left|\left| #1 \right|\right|}

\DeclareMathOperator*{\argmin}{\arg\!\min}
\DeclareMathOperator*{\argmax}{\arg\!\max}

\newcommand{\T}{\top}

\newcommand{\wTRE}{w^{\mbox{\tiny{TRE}}}}
\newcommand{\wMRE}{w^{\mbox{\tiny{MRE}}}}
\newcommand{\wTyl}{w^{\mbox{\tiny{Tyl}}}}
\newcommand{\wMar}{w^{\mbox{\tiny{Mar}}}}
\newcommand{\SigmaTyl}{\hat{\Sigma}_{\mbox{\tiny{Tyl}}}}
\newcommand{\SigmaMar}{\hat{\Sigma}_{\mbox{\tiny{Mar}}}}
\newcommand{\SigmaMRE}{\hat{\Sigma}_{\mbox{\tiny{MRE}}}}
\newcommand{\SigmaTRE}{\hat{\Sigma}_{\mbox{\tiny{TRE}}}}

\newcommand{\abs}[1]{\left \vert #1 \right \vert}

\newcommand{\Id}{\mathrm{I}_{p \times p}}
\newcommand{\Idn}[1]{\mathrm{I}_{#1 \times #1}}

\newcommand{\nb}{\bar{n}}

\newcommand{\trueCov}{\ensuremath{\Sigma_p}}
\newcommand{\trCov}{\tau_p}
\newcommand{\trLB}{\tau}
\newcommand{\trInvUB}{\underline{\tau}}

\newcommand{\AssumIndep}{\textbf{[SG-IND]}}
\newcommand{\AssumCCPSB}{\textbf{[CCP-SBP]}}
\newcommand{\AssumLC}{\textbf{[LC]}}

\newcommand{\oneVec}{{\bf 1}}
\newcommand{\indic}[1]{ \mathbbm{1}_{\left\{ #1 \right\} }}

\newcommand{\Cheeger}{\Psi_p}

\newcommand{\mh}{\widehat{m}}

\newcommand{\wh}{\widehat{w}}
\newcommand{\bwh}{\widehat{\w}}

\newcommand{\TODO}[1]{ 
	\ifx\NOTES\undefined\else
	{\tt \color{red} [TODO:#1] } 
	\fi
}

\renewcommand{\epsilon}{\varepsilon}

\newcommand{\Revision}[1]{{#1}}

\let\today\relax
\makeatletter
\def\ps@pprintTitle{%
	\let\@oddhead\@empty
	\let\@evenhead\@empty
	\def\@oddfoot{\footnotesize\itshape
		{} \hfill\today}%
	\let\@evenfoot\@oddfoot
}
\makeatother

\begin{document}

\begin{frontmatter}

\title{Tyler's and Maronna's M-estimators: Non-Asymptotic Concentration Results}

\author[1]{Elad Romanov\corref{mycorrespondingauthor}}
\author[2]{Gil Kur}
\author[3]{Boaz Nadler}

% \address[1]{School of Computer Science and Engineering, The Hebrew University of 
% Jerusalem}
\address[1]{Department of Statistics, Stanford University}
\address[2]{Department of Electrical Engineering and Computer Science, Massachusetts Institute of Technology}
\address[3]{Faculty of Mathematics and Computer Science, Weizmann Institute of Science}

\cortext[mycorrespondingauthor]{Corresponding author. Email address: \url{elad.romanov@gmail.com}}

\begin{abstract}
Tyler's and Maronna's M-estimators, as well as their regularized variants, 
	are popular robust methods to estimate the scatter or covariance matrix of a multivariate distribution.
	In this work, we study the non-asymptotic behavior of these estimators, for data sampled from
	a distribution that satisfies one of the following properties: 1) independent sub-Gaussian entries, up to a linear transformation; 2) log-concave distributions; 3) distributions satisfying a convex concentration property. 
	Our main contribution is the derivation of tight non-asymptotic concentration bounds of these M-estimators around a suitably scaled version of the data sample covariance matrix.
	Prior to our work, non-asymptotic bounds were
derived only for Elliptical and Gaussian distributions.
Our proof uses a variety of tools from non asymptotic random matrix theory and high dimensional geometry.
	Finally, we illustrate the utility of our results on two examples of practical interest: sparse covariance and sparse precision matrix estimation.  
\end{abstract}

%\begin{keyword} %alphabetical order
%Robust covariance estimation \sep
%Tyler's M-estimator \sep
%Maronna's M-estimator \sep
%\MSC[2020] Primary 62H12 \sep
%Secondary 15B52
%\end{keyword}

\end{frontmatter}

\section{Introduction}
        \label{sect:intro}
         
% TODO: Shorten writing; review recent literature on robust covariance estimation (Minsker, F. Pascal, Diakonikolas); sharpen message about novelty - these distributions admit non-independent entries;          

 Let $\x_1,\ldots,\x_n$ be $n$ i.i.d. samples from a $p$-dimensional random variable $X$. 
%  as well as to have $X\ne \bm{0}$ with probability $1$. 
 The $p\times p$ covariance matrix $\Sigma$ of $X$ is a central quantity of interest in multiple applications \cite{anderson,muirhead2009aspects}. In the classical regime with $n\gg p$,  if the random variable $X$ is not heavily tailed and there are no outliers, the empirical covariance matrix yields a relatively accurate estimator for $\Sigma$. 

To deal with heavy tails and potential outliers, 
%over the past decades 
several robust estimators were proposed and studied theoretically. Two popular procedures, applicable when $p<n$, include Maronna's and Tyler's M-estimators \cite{maronna1976robust,tyler1987distribution}.
Regularized variants, applicable also when $p>n$,
were also proposed and studied \cite{abramovich,ollila2014regularized,chen,pascal}. These estimators have 
found use in multiple applications, ranging from signal processing and radar detection to finance, see for example
\cite{couillet2014large,pascal,ollila2014regularized}.
We remark that in addition to the above, many other robust covariance estimators have been proposed and analyzed, see for example
% For many other estimators and theory on robust covariance estimation, see for example
\cite{hubert2008high,wiesel2015structured,dumbgen2016new, maronna2019robust,catoni2016pac,chen2018robust,ke2019user,minsker2022robust,mendelson2020robust,diakonikolas2019recent}
and references therein. 

In this work we study the properties of 
Tyler's and Maronna's M-estimators under 
several families of multivariate distributions. Our analysis is non-asymptotic and generalizes previous results, which were either asymptotic or limited to elliptical distributions. Before presenting our results, we first briefly describe these estimators and related prior work. For simplicity, 
we describe the estimators assuming 
$X$ has zero mean, and discuss how to relax this assumption later on. 

\paragraph{Maronna's M-estimator (ME)} 
One of the first
proposals for a robust covariance estimator, introduced by Maronna  \cite{maronna1976robust}, is defined as follows.  
Let  $u:[0,\infty) \to (0,\infty)$ be a function that is strictly positive, non-increasing, continuous, and such that the accompanying function $\phi(x)=x u(x)$ is non-decreasing and bounded. 
Then, for $n\ge p$,  Maronna's M-estimator (if exists) is a solution to the non-linear matrix equation
% \cite{maronna1976robust}
\begin{equation}
\label{eq:def_maronna}
\SigmaMar = \frac{1}{n}\sum_{i=1}^n \wMar_i \x_i \x_i^{\top}\,,
\quad
\textrm{where }
\quad
\wMar_i  = u \left( \frac{1}{p}\x_i^{\top} \SigmaMar^{-1}\x_i \right) \,.
\end{equation}
% such that the weights $\wMar_i$ satisfy
% \begin{equation}\label{eq:def_maronna_weights}
% \wMar_i  = u \left( \frac{1}{p}\x_i^{\top} \SigmaMar^{-1}\x_i \right) \,.
% \end{equation} 
% Denote $\lim_{n\to\infty}\phi(x)=\phi_\infty\in \R_+\cup\left\{\infty \right\}$. 
Maronna \cite{maronna1976robust} proved that under certain deterministic conditions on the samples $\x_i$, 
% Eqs. (\ref{eq:def_maronna})-(\ref{eq:def_maronna_weights}) have 
Eq. (\ref{eq:def_maronna}) has
a unique solution. 
Couillet {\it et al.} \cite{couillet2014robust} considered a high dimensional asymptotic framework, wherein \(n,p\to\infty\) 
with their ratio converging to a constant.
% with \(p/n\to c<1\). 
Assuming that $X$ has i.i.d. entries with sufficiently many finite moments, and that $\phi_\infty:=\sup_{x}\phi(x)>1$,
% Assuming that: 1) the samples are i.i.d. (with sufficiently many finite moments); and 2) $\phi_\infty:=\sup_{x}\phi(x)>1$, 
they proved that
% Eqs.(\ref{eq:def_maronna})-(\ref{eq:def_maronna_weights}) 
Eq. (\ref{eq:def_maronna})
has a unique solution with probability tending to one. 

 \paragraph{Maronna's regularized M-estimator (MRE)} For \(p>n\), Ollila and Tyler \cite{ollila2014regularized} proposed the following generalization of Maronna's M-estimator:
 For regularization parameter $\alpha>0$,
 \begin{equation}
 \label{eq:def_maronna_reg}
 \SigmaMRE = \frac{1}{1+\alpha}\frac{1}{n}\sum_{i=1}^{n} \wMRE_i \x_i \x_i^{\top} + \frac{\alpha}{1+\alpha} \Id \,,
 \quad
 \textrm{where }
 \quad 
 \wMRE_i = u \left( \frac{1}{p}\x_i^{\top} \SigmaMRE^{-1} \x_i \right) \,.
 \end{equation}
%  where  $\alpha>0$ is a regularization parameter, and
%  \begin{equation}
%  \label{eq:def_maronna_reg_weights}
% \wMRE_i = u \left( \frac{1}{p}\x_i^{\top} \SigmaMRE^{-1} \x_i \right) \,.
%  \end{equation}
%The function $u(\cdot)$ is assumed to satisfy the same conditions as in the un-regularized case.
%  As in the un-regularized case, the function $u(\cdot)$ is assumed to be positive, continuous, non-increasing with $\phi(x)=xu(x)$ non-decreasing. 
 As proven in \cite[Theorem 1]{ollila2014regularized}, for any $\x_1,\ldots,\x_n$
 and $\alpha>0$, 
 Eq. (\ref{eq:def_maronna_reg})
%  Eqs.  (\ref{eq:def_maronna_reg})-(\ref{eq:def_maronna_reg_weights}) have 
 has a unique solution.
%  , for any $\alpha >0$ and arbitrary samples $\x_1,\ldots,\x_n$.  

\paragraph{Tyler's M-estimator (TE)}
Introduced by Tyler
in \cite{tyler1987distribution}, TE is defined as the solution (if exists) of 
\begin{equation}\label{eq:def_tyler}
\SigmaTyl = \frac{1}{n}\sum_{i=1}^{n}\wTyl_i \x_i\x_i^\top, 
% \quad \Tr(\SigmaTyl) = p,
\quad
\textrm{where}
\quad
\wTyl_i = \left(\frac1p \x^\top_i\SigmaTyl^{-1}\x_i\right)^{-1},
\quad
\textrm{and subject to }
\quad
\Tr(\SigmaTyl) = p\,.
\end{equation}
% where
% \begin{equation}
% \label{eq:def_tyler_weights}
% \wTyl_i = \left(\frac1p \x^\top_i\SigmaTyl^{-1}\x_i\right)^{-1}.
% \end{equation}
While Tyler's estimator may seem like a special case of Maronna's with $u(x)=1/x$, this is not so since $u(\cdot)$ is singular at $x=0$.
For $n>p$, 
Kent and Tyler \cite[Theorems 1 and 2]{kent1988maximum} 
proved 
existence and uniqueness of 
% Tyler's M-estimator 
TE
under 
the condition that any linear subspace of \(\mathbb{R}^p\) of dimension $1\leq d\leq p-1$  contains less than $nd/p$ samples. 
For i.i.d. samples from a
random vector $X$ with a proper density in $\mathbb{R}^p$, this condition is satisfied with probability one. 

%This condition is satisfied, in particular, 
%when $\x_1,\ldots,\x_n$ are in \emph{general position}, namely, every subspace of dimension $d$ contains at most $d$ points.
% , then the condition of Kent and Tyler is satisfied. 
%If the law of $X$ gives zero probability to any subspace of positive co-dimension, then $\x_1,\ldots,\x_n$ are in general position with probability $1$. This is the case, for example, if $X$ has a density.

% Another important example is when $X$ has an elliptical distribution, 
% \begin{equation}
%         \label{eq:Elliptical}
%         X =  z\cdot \trueCov^{1/2} \cdot \xi \,,
% \end{equation}
% where $\xi \sim \mathrm{Uniform}(\Sp)$, with $\Sp$ denoting the $p$-dimensional unit sphere, $S_p$ is a positive definite matrix, and $z$ a strictly positive random variable which is {independent} of $\xi$ (but otherwise arbitrary). 
% % In that case,
% % TE exists with probability one; 
% For elliptically distributed $X$, TE is the maximum likelihood
% estimator for the shape matrix \cite{frahm2010generalization}. Moreover, its weights have the same distribution as for samples drawn from  $X\sim N(0,\Id)$ (see for example \cite{goes2017robust}).

  \paragraph{Tyler's regularized M-estimator (TRE)} Similarly to 
%   to Maronna's M-estimator, 
  ME, 
  Tyler's M-estimator 
%   also 
  does not exist for \(p>n.\) In recent years, several regularized variants were proposed \cite{abramovich,chen,pascal,sun2014,ollila2014regularized}. 
%   In this paper, 
  We focus on the estimator proposed in \cite{pascal}.
Given a regularization parameter $\alpha>0$, it is defined by
%   defined as the solution (if exists) of the following non-linear system where $\alpha>0$
%   is a regularization parameter, 
  \begin{equation}
 \label{eq:def_tyler_reg}
 \SigmaTRE = \frac{1}{1+\alpha} \frac{1}{n} \sum_{i=1}^{n} \wTRE_i \x_i \x_i^{\top} + \frac{\alpha}{1+\alpha} \Id \,,
 \quad
 \textrm{and }
 \quad
 \wTRE_i = \left(\frac1p \x_i^{\top} \SigmaTRE^{-1} \x_i\right)^{-1}\,.
 \end{equation}
%  where  
 %and $\alpha>0$ is a regularization parameter.
%  , and  
%  \begin{equation}
%  \label{eq:def_typer_reg_weights}
% \wTRE_i = \left(\frac1p \x_i^{\top} \SigmaTRE^{-1} \x_i\right)^{-1} \,.
%  \end{equation}
%Regarding existence and uniqueness,by
By
\cite[Theorem 2]{ollila2014regularized}, when $\alpha>p-1$, 
% Eqs. (\ref{eq:def_tyler_reg})-(\ref{eq:def_typer_reg_weights}) 
Eq. (\ref{eq:def_tyler_reg})
always admits a unique solution. When $\alpha\le p-1$, \cite[Theorem 3]{ollila2014regularized}
gave a deterministic sufficient and {almost} necessary condition for existence and uniqueness; for  $\x_1,\ldots,\x_n$ in {general position}, the condition holds for $\alpha > \max\left\{ 0,\frac{p}{n}-1 \right\}$. When $X$ has a density, this condition appeared earlier in \cite{pascal}.

\paragraph{Prior work}
%The field of robust 
%M-estimation has essentially started with the influential work of Huber, see the book \cite{huber2011robust}. 
Maronna's and Tyler's M-estimators and their variants, have been studied extensively, 
with a particular focus under elliptical distributions; see  \cite{wiesel2012geodesic,soloveychik2015performance,abramovich,ollila2014regularized,chen,pascal,wiesel2015structured,couillet2014robust,couillet2015random,couillet2014large,ollila2020shrinking,zhang2016marvcenko,goes2017robust,kent1988maximum,morales2015large,auguin2018large,ollila2012complex,couillet2016second}. 
% \paragraph{Theoretical Properties of Tyler's and Maronna's M-estimators.}
%\paragraph{Theoretical properties of M-estimators of scatter.}
The present paper extends 
several works that studied these estimators 
in a high-dimensional regime, where the number of samples $n$ and the dimension $p$ are both large and comparable. 
% Directly related to the present paper 
% is a collection of works studying 
% Maronna's and Tyler's M-estimators in the high-dimensional regime, as $n,p\to\infty$ with $p/n\to \gamma$. 
% several works which studied the above estimators under an asymptotic high dimensional setting, where $p,n\to\infty$ with   $p/n\to\gamma$. 
% %  Maronna's M-estimator was studied in \cite{couillet2i014robust}, under the assumption 
 {\it Couillet et al.} \cite{couillet2014robust} studied the asymptotic behavior of ME in the joint limit $p,n\to\infty$ with their ratio 
 tending to a constant. 
 Assuming that 
%  the random vector 
 $X$ has independent entries with zero mean, unit variance and sufficiently many finite moments, they proved that after a suitable scaling, ME converges asymptotically in spectral norm to the sample covariance matrix. 
In \cite{couillet2015random}, this analysis was extended to $X$ having an elliptical distribution with a general scatter matrix.  A similar asymptotic analysis for MRE
appeared in \cite{auguin2018large}.
Two variants of TRE were studied in \cite{couillet2014large}, assuming $X$ has an elliptical distribution. A key result of their analysis is that asymptotically, these M-estimators behave similarly to regularized sample covariance matrices with Gaussian measurements. %his analysis was applied to study the optimal shrinkage parameter $\alpha$ in Eq. (\ref{eq:def_tyler_reg}).  
% A\ finite sample bound on the Frobenius norm between the inverse covariance and the inverse of Tyler's M-estimator, under an elliptical distribution, was derived in 
% \cite{soloveychik2015performance}. 
% Closely related to our work is
{\it Zhang et al.}
 \cite{zhang2016marvcenko}, 
 studied TE assuming $X$ is Gaussian distributed with identity covariance. 
 They
 showed
%  who showed 
 that as $n,p\to\infty$, 
 the limiting
 the spectral distribution 
 of a properly scaled TE 
%  converges to the 
 is
 a
 Mar\v{c}enko-Pastur 
%  distribution. 
 law.
 Moreover,
 similar to our own results in the present paper,
%  Similarly to the approach taken in the present paper,
 they 
%  moreover 
proved a non-asymptotic deviation bound for the weights (\ref{eq:def_tyler}), showing that they are concentrated around some particular value\footnote{
 The proof of \cite[Lemma 3.3]{zhang2016marvcenko}  contains an error,
%  (specifically, the proof of \cite[Lemma 3.3]{zhang2016marvcenko}),
%  subtle (but consequential) error 
 which we remedy in the present paper; see Lemma~\ref{lem:Tyl-inv-bound} and the ensuing discussion.
 }.
% Next, 
Relying on their results, \cite{goes2017robust} extended 
the analysis
% the analysis of \cite{zhang2016marvcenko} 
to cover TRE, assuming $X$ is elliptically distributed.
% distributions, and have shown that the weights are tightly concentrated around some other value (that depends on the scatter matrix and regularization parameter). 
% showing a similar concentration result for the weights. 
We remark that the proofs in 
\cite{zhang2016marvcenko,goes2017robust} rely on properties specific to the Gaussian and elliptical distributions, and do not extend easily to other distributions.

Another recent work is \cite{louart2020robust}, which derived nonasymptotic concentration results for the Stieltjes transform of the spectral distribution of certain \Revision{regularized} M-estimators.
\Revision{
In the context of this work, they derived results only for the regularized variants of Maronna's M-estimator.
}
Their results apply under rather broad distributional assumptions, requiring only a concentration of measure property; see also their related papers \cite{louart2018concentration,louart2020equations}. 
% Their approach is similar in spirit with the distributional assumptions we adopt in the present paper (see Section~\ref{sec:prelims}), which. 
% Regarding the M-estimators considered above, their results apply only to the {\it regularized} variant of Maronna's estimator. 
% Treating the unregularized estimators, as well as certain subtleties that arise specifically for TRE, will call for a more sophisticated analysis.
\Revision{In constrast, our analysis of Tyler's M-estimators and the unregularized Maronna's M-estimator requires an additional \emph{anti-concentration} property (the small ball property) for the random vector $X$. It is an interesting question whether it is possible to derive similar results without the SBP assumption.
% That is, unlike the setup considered in \cite{louart2020robust}, assuming only a concentration of measure property does not suffice. 
}

\paragraph{Our contributions}
As mentioned above, 
most of the literature on Maronna's and Tyler's M-estimators has focused on asymptotic results,  establishing convergence as $n,p\to\infty$
without specifying rates.
Other works,
that derived non-asymptotic finite-$n$ bounds,  mostly considered Gaussian or elliptical distributions. 
% In almost all the aforementioned works, 
% the theoretical analysis was for distributions with independent entries, up to a linear transformation, an assumption which may 
% not hold in practice.
% This paper extends and generalizes 
% these works in several directions. 
% Specifically, 
This paper extends and
generalizes these works in several directions. 
We present a non-asymptotic analysis of both Tyler's and Marrona's M-estimators, and their regularized variants, 
under three broad families of multivariate data distributions: 1) independent sub-Gaussian entries, up to a linear transformation; 2) log-concave distributions; 3) distributions satisfying a convex concentration property (CCP).   
%  1) The class of \emph{log-concave} distributions, which has garnered, both historically and contemporarily, substantial interest from the mathematical statistics literature; 2) Distributions that satisfy a \emph{convex concentration property} (CCP), namely, distributions whose convex functionals concentrate at a sub-Gaussian rate. 
% Among the family of CCP distributions, a notable sub-family is the distributions that satisfy a  log-Sobolev inequality (LSI), whose smooth (not necessary convex) functionals concentrate at a sub-Gaussian rate. 
% First, we present a non-asymptotic theoretical analysis of both Tyler's and Marrona's M-estimators, and their regularized variants, for isotropic log Sobolev and log-concave distributions. Namely, we only assume that up to a linear transformation, the entries are uncorrelated. Also, we extend the analysis of \cite{zhang2016marvcenko}  for distributions with i.i.d. sub-Gaussian entries.
\Revision{Our main results are given in terms of non-asymptotic concentration bounds for the weights of the M-estimators \eqref{eq:def_maronna}-\eqref{eq:def_tyler_reg}
around some particular deterministic value. They imply that 
for these three families of distributions,
Maronna's and Tyler's M-estimators behave similarly to a rescaled sample covariance matrix. In Section~\ref{sec:app-covariance-est} we illustrate the utility of these results for two concrete examples of practical interest: sparse covariance and sparse precision matrix estimation. 
}
% Lastly, we illustrate the utility of our results on two concrete examples of practical interest: sparse covariance and sparse precision matrix estimation. 

% As noted by \cite{couillet2015random}, this expectation may be backed by recent developments in neural network -based generative modeling \cite{goodfellow2014generative,goodfellow2016deep} and its success in e.g. photorealistic image generation. Since a neural network $\mathrm{NeuralNet} : \R^d \to \R^p$ is a piece-wise linear, hence Lipschitz, mapping, it preserves Lipschitz concentration properties. Thus, if $Y$ satisfies an LSI (e.g., $Y\sim \mathcal{N}(0,I)$) then $X=\mathrm{NeuralNet}(Y)$ satisfies the Lipschitz concentration property (and in particular the CCP).

% {\color{blue} This is written awkwardly. Think how to improve this.}

%%%%%%%%%%%%%%%%%%%%%%%%%%%%%%%%%%%%%%%%%%%%%%%%%%%%%%%%%%%%%%%%%
%%%%%%%%%%%%%%%%%%%%%%%%%%%%%%%%%%%%%%%%%%%%%%%%%%%%%%%%%%%%%%%%%
\section{Main results}
\label{sect:main-results}

% Our goal is to 
Let $\x_1,\ldots,\x_n \in \R^p$ be $n$ samples of the form
\begin{equation}
\label{eq:X-Y}
        \x_i = \trueCov^{1/2} \cdot \y_i \,,
\end{equation}
where $\trueCov$ is a 
{strictly positive} $p$-by-$p$ matrix, 
and $\y_1,\ldots,\y_n \in \R^p$ are $n$ { i.i.d.} realizations of a zero mean isotropic random vector $Y$, namely
\begin{equation} \label{eq:Y-moments}
        \E[Y] = \mathbf{0}\,, \quad \E[Y Y^\T] = \Id \,.
\end{equation}
% With precise definition appearing in 
% Section~\ref{sec:prelims}, 
We assume that $Y=(Y_1,\ldots,Y_p)^\top$ is a continuous 
random vector, and
satisfies \emph{one} of the following distributional assumptions, with the precise definitions deferred to Section~\ref{sec:prelims}:
\begin{enumerate}[nosep]
    \item \AssumIndep: 
    The coordinates of $Y$ are independent, sub-Gaussian (Definition~\ref{def:sub-gaussian-linear})
    and have a bounded density. The constant $K>0$ denotes a uniform bound on sub-Gaussian constants, and $C_0>0$ a bound on the densities. 
    % All coordinates $Y_j$ are independent, sub-Gaussian random variables with constant $K>0$ (Definition~\ref{def:sub-gaussian-linear}).
    % Furthermore, their univariate densities are all upper bounded by a constant $C_0$.

    % has independent entries, each following univariate a sub-Gaussian distribution (Definition~\ref{def:sub-gaussian-linear}). The sub-Gaussian constants are all bounded by some constant $K$. Moreover, the densities are all bounded by some other constant $C$.
    \item \AssumCCPSB: 
    $Y$ satisfies the convex concentration property (CCP) with constant $K>0$ and also the small-ball property (SBP) with constant $C_0>0$ (Definition~\ref{def:small-ball}).
    
    % $Y$ satisfies the convex concentration property (CCP) with some constant  $K$ (Definition~\ref{def:CCP-property}) and the small-ball property (SBP) with another constant  $C$ (Definition~\ref{def:small-ball}). 
    % The constants $C,K$ are dimension-independent.
    \item \AssumLC: $Y$ has a log-concave distribution (Definition~\ref{def:logconcave}).
\end{enumerate}

\begin{remark}
    While the assumption that $X$ has zero mean may not hold in practice, it has been used extensively in previous studies of Tyler's and Maronna's M-estimators, cf. \cite{couillet2014robust,zhang2016marvcenko,goes2017robust,auguin2018large}. In Section~\ref{sec:discussion} we discuss how this restriction may be removed for some of our results.

\end{remark}

\Revision{
\begin{remark}
     We emphasize that the three families of distributions considered above are all distinct in the sense that neither one is contained in another. In particular, it is known that an i.i.d. sub-Gaussian vector does not necessarily satisfy the CCP, see for example \cite{gozlan2018characterization,huang2021dimension}. Furthermore, below (after the statement of Theorem~\ref{thm:main_tyler}) we mention examples of two distributions that satisfy one of \AssumCCPSB,~\AssumLC~ but not the other.
\end{remark}
}

We consider Maronna's and Tyler's M-estimators $\SigmaMar$ and $\SigmaTyl$, as well as their regularized variants $\SigmaMRE$ and $\SigmaTRE$, all computed from $\x_1,\ldots,\x_n$.   
 The latter two estimators depend on the regularization parameter $\alpha>0$, which we omit to simplify notation. 
%
%This work studies 
% focuses on studying 
We study
the nonasymptotic properties of 
these estimators 
%Maronna's and Tyler's M-estimators 
in the {high-dimensional regime}, where the number of samples $n$ and the dimension $p$ are both large and comparable. Their ratio is denoted by
% $\gamma=p/n\in(0,\infty)$.
\begin{equation}
    \gamma = 
    \frac{p}{n} \in (0,\infty) \,.
\end{equation}
Similarly to \cite{zhang2016marvcenko,goes2017robust}, while our results are nonasymptotic in $n$ (in the form of finite-$n$ deviation bounds), they involve constants that may depend on $\gamma$, often in a complicated manner that we do not keep track of explicitly. 

\subsection{Concentration for the weights of Maronna's and Tyler's Estimators}

Our first two results show concentration for the weights of Maronna's and Tyler's M-estimators
for
$\gamma=p/n < 1$, hence $n-p=\Omega(p)$.
Below, $\Cheeger$ denotes (up to a universal constant) the \emph{Cheeger constant} of the family of $p$-dimensional log-concave distributions; it is known that $\Cheeger\ge 1/\mathrm{polylog}(p)$ and conjectured that $\Cheeger=\Theta(1)$; see Section~\ref{sec:prelims}, Eq. (\ref{eq:cheeger-bound-main-paper}).
 
We start with Maronna's estimator. For $\epsilon>0$, let $\Ec_{ME}(\epsilon)$ be the event that (\ref{eq:def_maronna}) has a unique solution
whose weights satisfy 
%Maronna's estimator exists uniquely,
% (in other words, Eqs. (\ref{eq:def_maronna})-(\ref{eq:def_maronna_weights}) admit a unique solution), 
%and that furthermore its weights satisfy 
\begin{equation}\label{eq:main-maronna-weights}
    \max_{1\le i \le n} \left|\wMar_i-1/\phi^{-1}(1)\right| \le \epsilon \,.
\end{equation}
\begin{theorem}[The weights of Maronna's  estimator]
\label{thm:main_maronna}
        Assume $\gamma<1$, and that the functions $u(x)$, $\phi(x)=xu(x)$ 
        of Maronna's M-estimator further satisfy:
        \begin{enumerate}[label=(\roman*)]
            \item $\phi_\infty := \lim_{x\to\infty}\phi(x)>1$.
            \item There is a unique $d_0$ such that $\phi(d_0)=1$, namely, $\phi$ is strictly increasing at $\phi^{-1}(1)$. Moreover, the inverse map $\phi^{-1}$ is locally Lipshitz at $1$.
            \item $u(\cdot)$ is locally Lipschitz at $\phi^{-1}(1)$.
        \end{enumerate}
                Then, there are constants $c,C,\epsilon_0>0$ that depend on the distribution of $Y$ and on $\gamma$, such that for any  $\epsilon<\epsilon_0$,
                the following holds.
               
                \begin{itemize}
                     \item
            %          Under the assumption that $Y$ satisfies \AssumLC, it holds that 
            %  
            Assume that $Y$ satisfies~\AssumLC. Then      $
                    \Pr(\Ec_{ME}(\epsilon)^c) \le Cn^2 e^{-c(\Cheeger\sqrt{n})\epsilon} 
                    $.
                    \item 
     % Under the assumption that $Y$ satisfies
     Assume that $Y$ satisfies either~\AssumIndep~or \AssumCCPSB. Then 
                    % it holds that 
                    $
                    \Pr(\Ec_{ME}(\epsilon)^c) \le Cn^2 e^{-cn\epsilon^2} 
                    $.
                \end{itemize}  
\end{theorem}

\begin{remark}
    When saying that a constant $C$ \emph{  depends on the distribution of $Y$}, we mean that it may depend on the parameters $K,C_0>0$ above, whenever $Y$ is distributed according to  \AssumIndep~ or \AssumCCPSB. \Revision{
    Specifically, in the bound above in Theorem~\ref{thm:main_maronna}, the constants $c,C$ can be taken to be monotonic in the parameters $K,C_0$. The reason is that as one increases $K$ and decreases $C_0$, the class of distributions considered (\AssumIndep~or~\AssumLC) becomes larger.
    However, we do not keep track of this dependence explicitly.
    % While it is clear that our probability bounds degrade as $K$ increases and $C_0$ decreases (since the assumed constraint becomes looser, see Section~\ref{sec:prelims}), we do not keep track of this dependence explicitly.
    }
\end{remark}
% Note that the statement  $C$ \emph{  depends on the distribution of $Y$} implies that it may depend on the parameters $K,C_0>0$ above (when $Y$ is distributed  \AssumIndep~ or \AssumCCPSB). \newline

% \begin{remark}
It is interesting to compare Theorem~\ref{thm:main_maronna} to the results of \cite{couillet2014robust}.
%on Maronna's M-estimator. 
% They consider a setting where
Assuming that 
the random vector $Y$ has i.i.d. entries with sufficiently many finite moments,
%\cite{couillet2014robust}
they proved that asymptotically, as $n,p\to\infty$ 
with their ratio tending to a constant, almost surely $\max_{1\le i \le n}\left|\wMar_i-1/\phi^{-1}(1)\right|\to 0$, see \cite[Eq. (6)]{couillet2014robust}.
Our analysis extends theirs in two aspects: (i) It holds also for random vectors $Y$ that do not have independent entries, but instead satisfy certain multivariate concentration properties;  (ii) Our results are non-asymptotic, in the form of concentration inequalities for the weights.
% \end{remark}

Our next theorem regards Tyler's M-estimator. 
To this end, denote by $\trCov = \frac{1}{p}\tr\trueCov$ the normalized trace of the population covariance matrix. For any  $\epsilon>0$, let  $\Ec_{TE}(\epsilon)$ be the event that Tyler's estimator (\ref{eq:def_tyler}) exists uniquely, 
with weights that satisfy 
\begin{equation}\label{eq:main-tyler-weights}
    \max_{1\le i \le n} \left|\trCov \cdot \wTyl_i-1\right| \le \epsilon .
\end{equation}

\begin{theorem}[The weights of Tyler's estimator]
        \label{thm:main_tyler}
        There are constants $c,C,\epsilon_0>0$, that depend on the distribution of $Y$ and on $\gamma<1$, such that for any $\epsilon<\epsilon_0$ the following hold.
       %Depending on the distribution of $Y$,  
       \begin{itemize}
            \item 
            % Under the assumption that $Y$ satisfies \AssumLC, it holds that
            Assume that $Y$ satisfies ~\AssumLC. Then 
            $
            \Pr\left( \Ec_{TE}(\epsilon)^c  \right) \le Cn^2 e^{-c \Cheeger\min\left\{ \sqrt{n}\epsilon,\,  n^{1/4}\right\} } 
            $.
            \item 
            % Under the assumption that $Y$ satisfies
            %         \AssumIndep~ or \AssumCCPSB, 
            %         it holds that 
            Assume that $Y$ satisfies either~\AssumIndep~or \AssumCCPSB. Then
            $
            \Pr\left( \Ec_{TE}(\epsilon)^c \right) \le Cn^2 e^{-c \min\left\{ n\epsilon^2,\, n^{1/2}\right\}} 
            $.
        \end{itemize}
\end{theorem}
% Concentration of the weights of Tyler's estimator was studied in \cite{zhang2016marvcenko,goes2017robust}.
% , under a slightly different definition for the weights.  
% \begin{remark}
%
%
% We remark that Zhang {\it et. al.} \cite{zhang2016marvcenko} provided (an incorrect) proof of Theorem~\ref{thm:main_tyler} in the specific case where $Y \sim \mathcal{N}(\bm{0},\Id)$.  They 
% stated a non-asymptotic bound for the 
% weights\footnote{Note that their definition of the weights differs from ours by a constant factor, and therefore does not depend on $\trCov$.}: 
% % satisfy a uniform non-asymptotic concentration bound: 
% %  \[
% $
%  \Pr\left(\max_{1\leq i \leq n}\left|w_i -1\right| \geq \epsilon \right) \leq Cne^{-cn\epsilon^2}
%  $.
% %  for all $\epsilon>0$ small enough. 
% This 
% % result
% bound
% should be compared to 
% % our result under assumption
% Theorem~\ref{thm:main_tyler} under Assumption~
% \AssumIndep. 
%
%
% \end{remark}
% ; their proof relies explicitly on properties of the Gaussian distribution. 
% \cite{goes2017robust} extended the analysis to elliptical distributions with an arbitrary covariance matrix. We remark that the analysis of \cite{zhang2016marvcenko} contains a  nontrivial mathematical error, that we  correct in the present paper;
% ; this is the origin of the correction terms (the expressions that do not depend on $\epsilon$) in the bounds above, that do not appear in \cite{zhang2016marvcenko}. 
%see Section~\ref{sec:proof-Tyler} for details. \newline
\Revision{We remark that Zhang {\it et. al.} \cite{zhang2016marvcenko} provided a proof of Theorem~\ref{thm:main_tyler} in the specific case where $Y \sim \mathcal{N}(\bm{0},\Id)$.  They 
stated a non-asymptotic bound for the 
weights\footnote{Note that their definition of the weights differs from ours by a constant factor, and therefore does not depend on $\trCov$.}: 
$
 \Pr\left(\max_{1\leq i \leq n}\left|w_i -1\right| \geq \epsilon \right) \leq Cne^{-cn\epsilon^2}
 $.
This bound should be compared to Theorem~\ref{thm:main_tyler} under Assumption~\AssumIndep. Importantly, their proof contains a non-trivial gap, which we correct in the present paper (see Section~\ref{sec:proof-Tyler}). }
\newline

Next,
we illustrate  Theorems~\ref{thm:main_maronna} and \ref{thm:main_tyler} by the following numerical experiment. We generated i.i.d. samples according to 
either of the following 
two distributions for $Y$:  
\begin{enumerate}[noitemsep,nolistsep]
    \item {\it Laplace:} All coordinates $Y_i$ are i.i.d. Laplace, with density  $\mathrm{Lap}(y)= {\frac{1}{\sqrt{2}}}\exp(-\sqrt{2}|y|)$. Hence, $Y$ is isotropic, log-concave, but not sub-Gaussian. 
    \item {\it Permuted smoothed}: $Y = \frac{1}{\sqrt{1+\sigma^2}} A + \frac{\sigma}{\sqrt{1+\sigma^2}} Z$, where $A \in \{\pm 1\}^p$ is uniform in the set $\left\{a\in \{\pm 1\}^p \,:\,\sum_{i=1}^p a_i=0\right\}$, $Z\sim \mathcal{N}(\bm{0},\Id)$ and $\sigma=0.01$ is the smoothing level. The entries of $A$ are clearly dependent; nonetheless, by a classical result of Maurey \cite{maurey1979construction}, it can be shown to satisfy the CCP. Consequently, $Y$ satisfies the CCP and SBP.
\end{enumerate}
\vspace{1em}

We compute Tyler's and Maronna's M-estimators from $n=2p$ samples, for the latter using $u(x)=2/(1+x)$. Figure~\ref{fig:1} shows the deviation of the corresponding weights from their limiting value %, which, in both cases, is 
$w^*=1$.  
We present on a log-log scale both the $\ell_\infty$ deviation $\max_{1\le i \le n}|\hat{w}_i-w^*|$, and the root mean squared error (RMSE) ${\small \sqrt{\frac1n \sum_{i=1}^n (\hat{w}_i-w^*)^2}}$, as a function of the dimension $p$. The slope of either line is approximately $\frac12$, consistent with  Theorems~\ref{thm:main_maronna} and \ref{thm:main_tyler}. 

\begin{figure}
    \centering
    \includegraphics[width=0.45\textwidth]{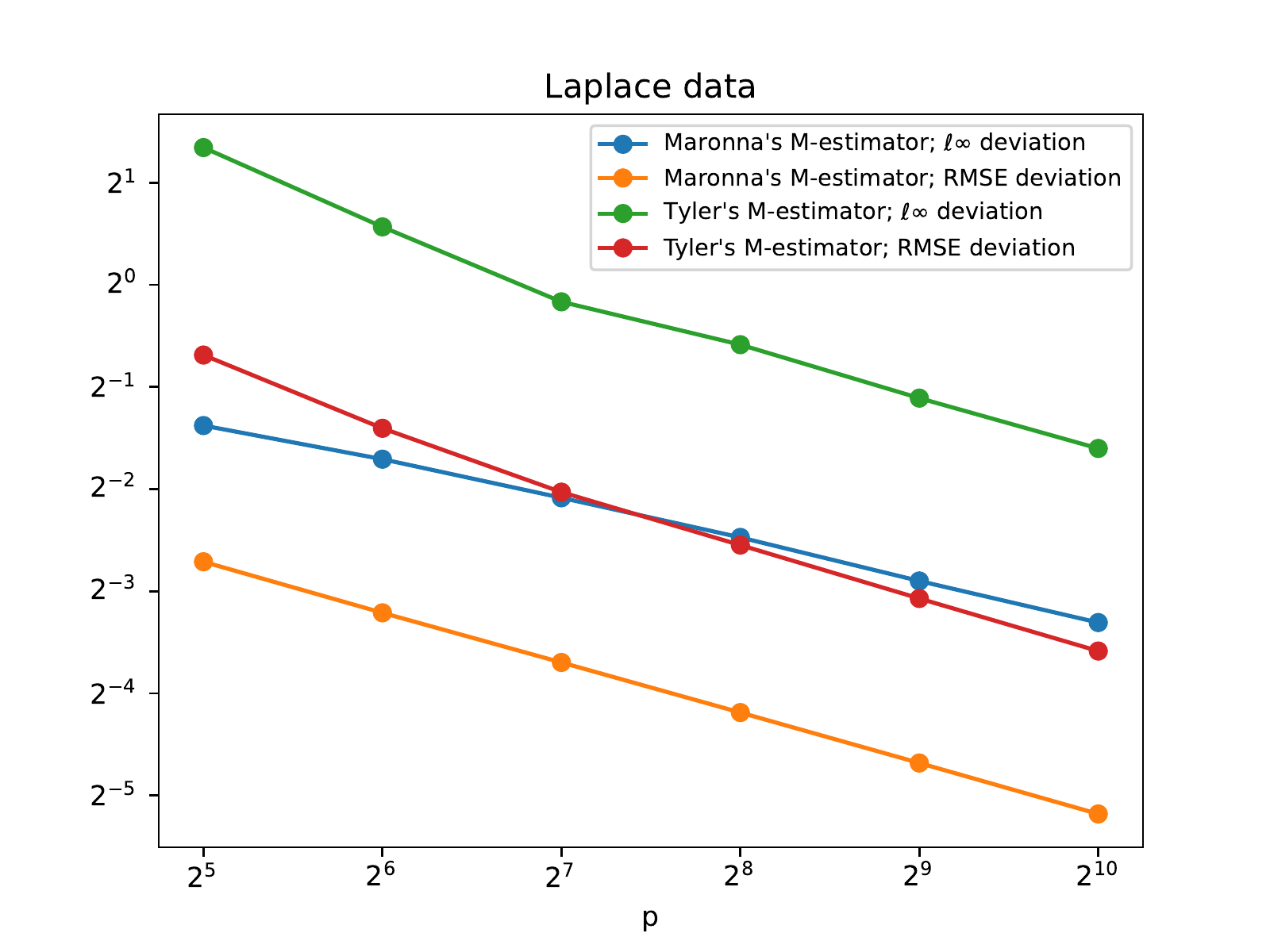}
    \vspace{5pt}
    \includegraphics[width=0.45\textwidth]{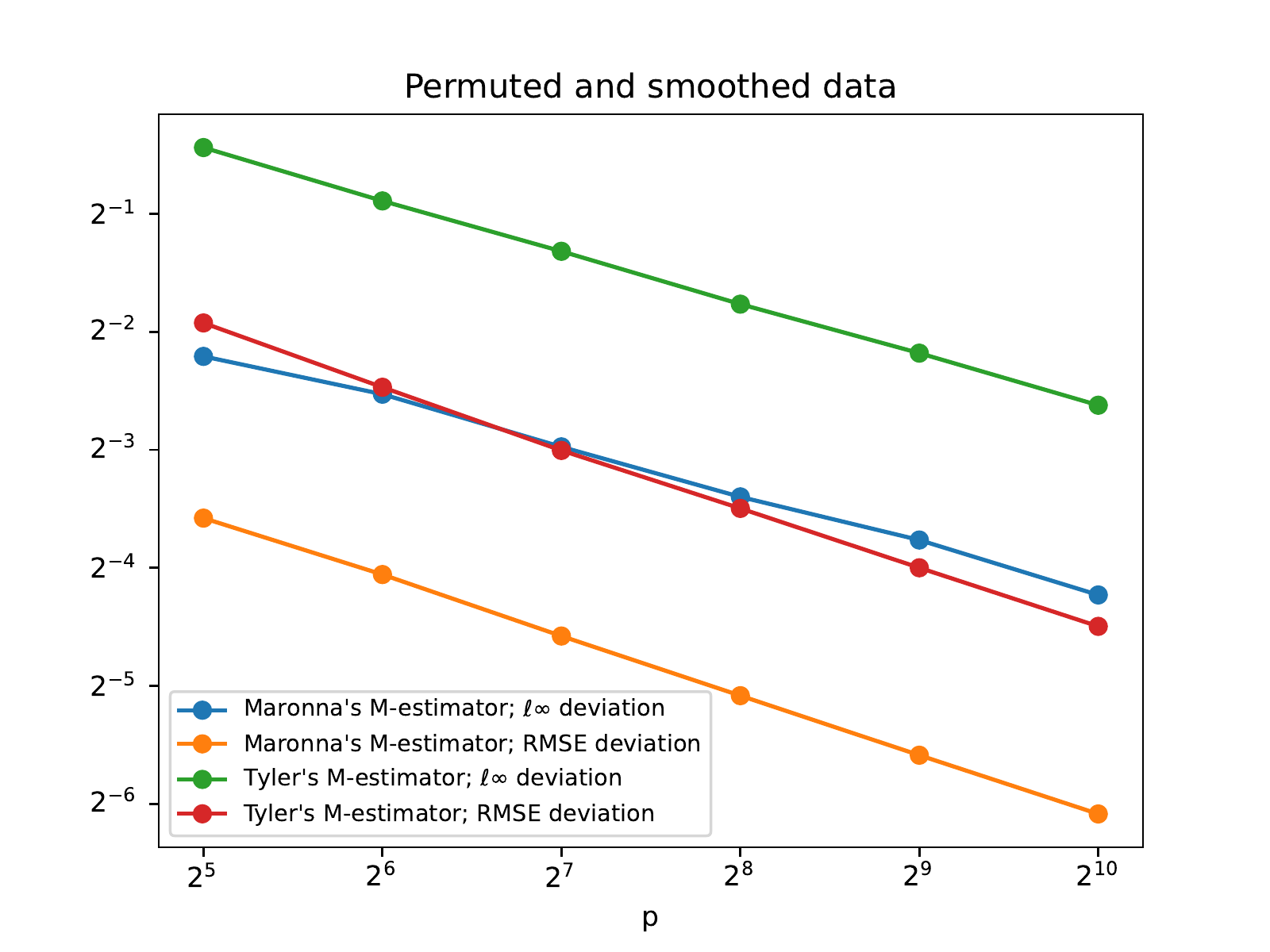}
    \caption{The empirical deviation of the weights, plotted in log-log scale. Each point on the graph corresponds to the average of $50$ random repetitions.}
    \label{fig:1}
\end{figure}

\begin{remark}
    Observe that in Theorem~\ref{thm:main_maronna}, the weights of Maronna's estimator do not depend on
    the population covariance $\trueCov$. 
    This is because the weights are preserved under an arbitrary linear full-rank transformation of the data. Let $\x_1,\ldots,\x_n$ and consider the transformed measurements $\tilde{\x}_i=A\x_i$, where $A\in \R^{p\times p}$ is full-rank. Denote the corresponding Marrona's estimators (\ref{eq:def_maronna}) by
    % , per (\ref{eq:def_maronna}),
    $\hat{\Sigma} = \frac1n \sum_{i=1}^n w_i \x_i\x_i^\T$ and 
    $\tilde{\Sigma}=\frac1n \sum_{i=1}^n \tilde{w}_i \tilde{\x}_i\tilde{\x}_i^\T$.
    % \[
    % \hat{\Sigma} = \frac1n \sum_{i=1}^n w_i \x_i\x_i^\T,\quad \tilde{\Sigma}=\frac1n \sum_{i=1}^n \tilde{w}_i \tilde{\x}_i\tilde{\x}_i^\T     \,. 
    % \]
    One may verify that $\tilde{\Sigma}=A\hat{\Sigma}A^\T$, hence $w_i=\tilde{w}_i$; setting $A=\trueCov^{-1/2}$, we deduce that the weights $\wMar_i$ do not depend on the covariance matrix of $X$. 
    
    In contrast, consistent with Theorem~\ref{thm:main_tyler}, the weights of Tyler's estimator \emph{do} depend on $\trueCov$. The reason is that while the linearly transformed estimator $A\hat{\Sigma}A^\T$ does solve the {\it unconstrained} Eq. (\ref{eq:def_tyler}), corresponding to the transformed measurements (similarly to Maronna's estimator),  one needs to rescale the weights due to the constraint ${\Tr(\tilde{\Sigma})=p}$.
\end{remark}

\subsection{Regularized Estimators}

We next consider the regularized variants of Maronna's and Tyler's M-estimators. 
As mentioned in Section~\ref{sect:intro}, MRE exists uniquely for all $\gamma,\alpha>0$, whereas for any $\gamma>0$, TRE is only guaranteed to exist uniquely when $\alpha=\alpha(\gamma)>0$ is sufficiently large.
% Recall that 
% these estimators are well-defined
% for all values of $\gamma$,
% with TRE provided the regularization $\alpha$ is sufficiently large for TRE. 

The regularization term $\frac{\alpha}{1+\alpha}\Id$ in Eqs. (\ref{eq:def_maronna_reg}) and (\ref{eq:def_tyler_reg}) shrinks the solution towards the identity matrix. %in particular, $\lambda_{\min}(\SigmaMRE),\lambda_{\min}(\SigmaTRE)\ge \frac{\alpha}{1+\alpha}$.
% (when $n<p$, in fact $\lambda_{i}(\SigmaMRE)= \frac{\alpha}{1+\alpha}$ for all $n+1\le i \le p$). 
As a result, the weights of MRE and TRE, and our deviation bounds 
for them, 
depend on the underlying population  matrix $\trueCov$.
Accordingly, 
throughout this section we operate under the following additional constraints on $\trueCov$, so to ensure that it has the same scale as the identity matrix. Specifically,
% on the population covariance.
for constants $s_{\max}\ge \trLB > 0$, we assume:
\begin{itemize}
    \item Bounded operator norm:
    \begin{equation}\label{eq:cond-smax}
        \|\trueCov\| \le s_{\max} \,.
    \end{equation}
    \item Lower bound on total energy:
    \begin{equation}\label{eq:cond-trLB}
        \trCov = \frac{1}{p}\tr\trueCov \ge \trLB \,.
    \end{equation}
\end{itemize}
% Conditions (\ref{eq:cond-smax})-(\ref{eq:cond-trLB}) imply that $\trueCov$ has the same scale as the identity matrix.

Towards stating our results, 
given a function $\phi(x)$, we 
define the following two functions ${Q,F:\R_{+} \to \R_+}$:
\begin{equation}
    Q(d)=\frac1p \E \tr \trueCov\left( \phi(d)\sum_{i=1}^{n-1}\x_i\x_i^\T + \alpha d \Id\right)^{-1} \,,\quad
    F(d) = (1+\alpha)\frac{Q(d)}{1+\gamma \phi(d) Q(d)}\,.
\end{equation} 
% where $\phi(x)=xu(x)$ is non-decreasing and $u(x)$ is non-increasing. 
As we shall see below, the functions $Q$ and $F$ play a decisive role in determining the weights of MRE and TRE. A key quantity is the solution $d^*>0$ to the following ``master equation'':
\begin{equation}\label{eq:main-master-equation}
    F(d^*)=1 \,.
\end{equation}
Later, we show that 
if $u(\cdot)$ is non-increasing and
$\phi(x)=xu(x)$ is non-decreasing, 
then \eqref{eq:main-master-equation}
admits a unique solution. 
%indeed there exists a unique solution to the above equation. 

The theorem below regards Maronna's regularized M-estimator. For $\epsilon>0$, let $\Ec_{MRE}(\epsilon)$ be the event that MRE exists uniquely with weights that satisfy 
        \begin{equation}
            \max_{1\le i \le n} |\wMRE_i-u(d^*)| \le \epsilon\,.
        \end{equation}

\begin{theorem}[The weights of MRE.]
    \label{thm:main_maronna_reg}
    Let $\alpha>0$ and $u$ be a Lipschitz continuous on any compact sub-interval of $[0,\infty)$. Then,
    Eq. (\ref{eq:main-master-equation}) has a unique solution which satisfies $d^*\in [\underline{d},\overline{d}]$, where 
    the constants $0<\underline{d}\le \overline{d}$ depend on the distribution of $Y$, $\gamma$, $\alpha$, $s_{\max}$ and $\trLB$. 

    Furthermore, there are constants $c,C,\epsilon_0>0$, that depend on the distribution of $Y$, $\gamma$, $\alpha$, $s_{\max}$ and $\trLB$, such that for all $\epsilon< \epsilon_0$,
        \begin{itemize}
            \item Assume that $Y$ satisfies~\AssumLC. Then $\Pr(\Ec_{MRE}(\epsilon)^c) \le Cn^2e^{-c(\Cheeger\sqrt{n})\epsilon}$.  
            \item Assume that $Y$ satisfies either~\AssumIndep~or \AssumCCPSB. Then ${\Pr(\Ec_{MRE}(\epsilon)^c) \le Cn^2e^{-cn\epsilon^2}}$.  
        \end{itemize}
    \end{theorem}
%\begin{remark}
We note that the weights of MRE were previously studied in \cite{auguin2018large}, 
assuming $Y$ has i.i.d. entries.
% with sufficiently many finite moments. 
They proved that asymptotically, as $n,p\to\infty$ at a fixed ratio $p/n=\gamma$,
% where they proved that asymptotically 
% %as $n,p\to\infty$ with 
% with $p/n\to c$, 
% then 
one has
$\max_{1\le i \le n}\|\wMRE_i-u(d^*)\| \to 0$ w.p. $1$. 

Lastly, we consider Tyler's regularized M-estimator (TRE). TRE is superficially a special case of MRE, corresponding to $u(x)=1/x$ and $\phi(x)=1$; a crucial difference, however, is that $u(\cdot)$ is singular at $x=0$.
% : $u(x)\to \infty$ as $x\to 0$. 
Accordingly, to carry out our analysis, we 
% impose 
require
an additional constraint on $\trueCov$: there exists a constant $\trInvUB>0$ such that
\begin{equation}
    \frac{1}{p}\tr\trueCov^{-1} \le \trInvUB \,.
\end{equation}
% We remark that $\trInvUB$ can be interpreted as a measure of the non-degeneracy (invertibility) of $\trueCov$; of course, $p^{-1}\tr\trueCov^{-1}\le 1/\lambda_{\min}(\trueCov)$.
% however note that $1/\lambda_{\min}(\trueCov)$ may in fact be larger than $p^{-1}\tr\trueCov^{-1}$ by a factor of up to $p$. Consequently, a bound on $p^{-1}\tr\trueCov^{-1}$ can be, in some cases, a substantially less restrictive requirement than a constant lower bound on $\lambda_{\min}(\trueCov)$. 

For $\epsilon>0$, let $\Ec_{TRE}(\epsilon)$ be the event that TRE exists uniquely, and that furthermore its weights satisfy 
        \begin{equation}
            \max_{1\le i \le n} |\wTRE_i-1/d^*| \le \epsilon\,.
        \end{equation}

\begin{theorem}
    [The weights of TRE.]
    \label{thm:main_tyler_reg}
    Let $u(x)=1/x$, $\phi(x)=1$ and ${\alpha>\max\{0,\gamma-1\}}$. 
    Then,
    Eq. (\ref{eq:main-master-equation}) has a unique solution which satisfies $d^*\in [\underline{d},\overline{d}]$, where 
    the constants $0<\underline{d}\le \overline{d}$ depend on the distribution of $Y$, $\gamma$, $\alpha$, $s_{\max}$, $\trLB$ and $\trInvUB$.

    % \begin{itemize}
        % \item
        % There is a unique solution $d^*>0$ to (\ref{eq:main-master-equation}).
        % (Solution of the master equation.) 
        % There is a unique number $d^*>0$ such that 
        % \begin{equation}\label{eq:main-master-TRE}
        %     F(d^*)=1 \,.
        % \end{equation}
        % Moreover, there are constants $0<\underline{d}\le \overline{d}$, that depend on the distribution of $Y$, $\gamma$, $\alpha$, $s_{\max}$, $\trLB$ and $\trInvUB$, such that $d^*\in [\underline{d},\overline{d}]$.

        % \item 
        % (Concentration for the weights.) With probability $1$, there exists a unique solution to the fixed point equations (\ref{eq:def_tyler_reg})-(\ref{eq:def_typer_reg_weights}). 
        % TRE exists uniquely w.p. $1$. 
        Furthermore, there are constant $c,C,\epsilon_0>0$, that depend on the distribution of $Y$, $\gamma$, $\alpha$, $s_{\max}$, $\trLB$ and $\trInvUB$, such that for all $\epsilon< \epsilon_0$,
        \begin{itemize}
            \item Assume that $Y$ satisfies~\AssumLC. Then $\Pr(\Ec_{TRE}(\epsilon)^c) \le Cn^2e^{-c(\Cheeger\sqrt{n})\epsilon}$.  
            \item Assume that $Y$ satisfies either~\AssumIndep~or \AssumCCPSB. Then ${\Pr(\Ec_{TRE}(\epsilon)^c) \le Cn^2e^{-cn\epsilon^2}}$.  
        \end{itemize}

    % \end{itemize}
\end{theorem}
% \begin{remark}
For the special case where $X$ follows a Gaussian or an elliptical distribution, similar non-asymptotic concentration bounds for the weights of TRE were provided in \cite{goes2017robust}.
% They , relying on specific properties of the Gaussian distribution; to derive results under more general distributions, we leverage tools from random matrix theory.
% A important step in the proof of \cite{goes2017robust} entails the analysis of the master equation (\ref{eq:main-master-TRE}), and to do so, the authors of \cite{goes2017robust} rely explicitly on particular properties of the Gaussian distribution. To treat the more general setting considered in the present paper, we need to leverage more sophisticated tools from random matrix theory. 
%
Lastly, we remark that as noted in \cite{goes2017robust}, the dependence on $\trInvUB$ can be removed when $\alpha$ is sufficiently large, i.e  $\alpha \ge Cs_{\max}$ where $C$ is a suitable constant (see Remark~\ref{remark:goes-alpha} in 
% Appendix, Section~
\ref{sec:proof-lem:TRE-1} for further details). 
% We make some remarks regarding this point during the proof of Theorem~\ref{thm:main_tyler_reg}, 

% \end{remark}
        
        \paragraph{Paper outline} 
        The rest of the manuscript is organized as follows.
        In Section~\ref{sec:prelims} we provide definitions and technical background related to the distributional assumptions from Section~\ref{sect:main-results} above. In Section~\ref{sec:app-covariance-est} we describe some applications of our results to robust covariance and precision matrix estimation.
        Section~\ref{sec:proofs} is devoted to the proofs of Theorems \ref{thm:main_maronna}-\ref{thm:main_tyler_reg}, with some technical details deferred to the Appendix. Finally, we offer some concluding remarks in Section~\ref{sec:discussion}. 
        
\section{Preliminaries and Technical Background}
\label{sec:prelims}

We provide a brief background and definitions for the distributions considered in Section~\ref{sect:main-results}. 
% For more details, we refer 
% % the interested reader 
% to Appendix, Section~\ref{sec:appendix-more-prelims}.
%We remind the reader that  $Y\in\R^p$ 
Recall that $Y\in\R^p$
is an isotropic random vector. Throughout this text, $\|\cdot\|$ denotes the Euclidean norm in the suitable dimension $p$.

\begin{definition}
    \label{def:sub-gaussian-linear}
            $Y$ is a {\bf sub-Gaussian} vector with constant $K>0$ if for any 
            % directions 
            $\|u\|=1$,
            \[
            \Pr\left( \left|u^\T Y -\E[u^\T Y]\right| \ge t \right) \le 2\exp \left[ -(t/2K)^2\right]  \,.                  
            \] 
    \end{definition}
% In other words: $Y$ is sub-Gaussian if its one-dimensional projections have a (uniform) Gaussian tail.
Sub-Gaussianity implies that linear functions of $Y$ concentrate.
For our analysis we shall also need concentration of some sufficiently well-behaved non-linear functions.
% for parts of our analysis, however, we shall needs means of controlling more complicated functionals
% of $Y$. 
Following
\cite{adamczak2011sharp,meckes2012concentration}, we consider the following property:
\begin{definition}
    \label{def:CCP-property}
        $Y$ satisfies the  {\bf convex concentration property  (CCP)} with constant  $K>0$ if for every $1$-Lipschitz {convex} $f:\R^p\to\R$, the random variable $f(Y)$ is sub-Gaussian with constant $K$:
        \begin{equation}\label{eq:ccp}
            \Pr\left( \left|f(Y) -\E[f(Y)]\right| \ge t \right) \le 2\exp\left[ -(t/2K)^2\right] \,.
        \end{equation}
    \end{definition}
The CCP was introduced by  
 Talagrand \cite{talagrand1995concentration}, who proved that if $Y$ has independent, uniformly bounded entries then it satisfies the CCP. Subsequent works established weaker conditions under which the CCP holds, see  \cite{adamczak2015note,van2014probability} and references therein.
Another family of distributions satisfying the CCP, which has received attention in machine learning and statistics (e.g. \cite{block2020generative,bakry2013analysis,raginsky2017non}), are the distributions satisfying a Log-Sobolev Inequality.\footnote{Such distributions in fact satisfy a stronger {\it Lipschitz concentration property}: the function $f(\cdot)$ in Definition~\ref{def:CCP-property} does not need to be convex.}
% If $Y$ satisfies an LSI (with appropriate constant), then (\ref{eq:ccp}) holds for any $L$-Lipschitz $f(\cdot)$, not necessarily convex. We provide some background and details on the LSI in the Appendix, Section~\ref{sec:appendix-more-prelims-LSI}.
% for more details, see Appendix, Section~\ref{sec:appendix-more-prelims-LSI}.

% We consider the following:
For parts of our analysis, we shall also need the following:
\begin{definition}
\label{def:small-ball}
    $Y$ satisfies the {\bf small ball property (SBP)} with constant $C_0$  if for any $\|u\|=1$ and $a\in \R$, 
    \begin{equation}\label{eq:sbp}
        \Pr \left( \left|u^\T Y - a\right| \le t \right) \le C_0 t \quad \textrm{ for all }t\ge 0 \,.
    \end{equation}
\end{definition}
The SBP is an \emph{anti-concentration} property: it states that the law of $u^\T Y$ cannot put 
% too much mass very close to its mean $\E[u^\T Y]=0$. 
large mass around any particular value $a\in\R$.
For our purposes, the SBP will be {especially important for bounding} the smallest eigenvalue of the sample covariance matrix, which is a key step in 
{several of our proofs.}
% analyzing Maronna's and Tyler's unregularized estimators. 
Note that if the density of $u^\T Y$ is bounded, then (\ref{eq:sbp}) holds for some appropriate $C_0$. The following remarkable result, due to Rudelson and Vershynin, states that if $Y$ has independent entries, each with bounded density, then it satisfies the SBP \cite[Theorem 1.2]{rudelson2015small}:
\begin{lemma}
    \label{lem:sbp-indep}
    Let $Y=(Y_1,\ldots,Y_p)$ have independent entries, 
    with univariate densities all uniformly bounded by $C$.
    Then $Y$ satisfies the SBP with constant $2\sqrt{2}C$.
\end{lemma}

\paragraph*{Log-concave distributions}
We next consider the family of log-concave distributions on $\R^p$. 
This rich family has found multiple applications, for example, in
statistics \cite{bagnoli2005log}, pure mathematics \cite{brazitikos2014geometry,stanley1989log}, computer science \cite{balcan2013active,lovasz2007geometry} and economics \cite{an1997log}. 
% It includes
Particular members in this family include
the Gaussian, exponential, uniform over convex bodies, logistic, Gamma, Laplace, Weibull, Chi and Chi-Squared, Beta distributions and more.

\begin{definition}\label{def:logconcave}
        $Y$ is \textbf{log-concave} if it has a density $p_{Y}(y) = e^{-V(y)}$ 
        % for $p_{Y}(\cdot)$ is of the form,
        % \[
        % p_{Y}(y) = e^{-V(y)}\,,
        % \]
        % where 
        for $V:\R^p \to \R\cup \{\infty\}$ convex.
\end{definition}

\Revision{It is known that log-concave random vectors are sub-exponential with a universal constant, see e.g. \cite{brazitikos2014geometry}.}
% {\color{olive} Elad: We don't ever use this Lemma. Delete altogether?}
% \begin{lemma}
% \label{lem:log-concave-linear}
% There is a universal $c>0$ such that for every isotropic log-concave $Y$ and $\|u\|=1$,
% \[
% \Pr\left( \abs{u^\top X -\E[u^\top X]} \ge t \right) \le 2\exp\left[ -c\min\{t,t^2\} \right] \,.     
% \] 
% \end{lemma}   
% As mentioned above when introducing the CCP, to prove our results we shall need means of controlling more complicated functionals of $Y$ beside linear ones. 
The following 
% concentration inequality 
is implied by a recent breakthrough result of 
Klartag and Lehec \cite[Theorem 1.1]{klartag2022bourgain}: 
% Chen \cite{chen2021almost}:
\begin{lemma}\label{lem:log-concave-lip-concentration}
    There exists a universal $c>0$ and some
    % are universal $c_0,c>0$, and 
    $\Cheeger$ satisfying
    \begin{equation}
        \label{eq:cheeger-bound-main-paper}
        % p^{-c_0 \left(\frac{\log\log p}{\log p}\right)^{1/2}} 
         (\log p)^{-5}
        \le \Cheeger \le 1  
    \end{equation}
    such that for every
     isotropic log-concave $Y\in \R^p$ and $1$-Lipschitz function $f(\cdot)$,
    \begin{equation}
        \Pr\left( \left| f(Y)-\E[f(Y)]\right| \ge t \right) \le \exp\left[-c\Cheeger t\right] \,.
    \end{equation}
\end{lemma}
The quantity $\Cheeger$ is (up to a universal constant) the \emph{Cheeger constant} corresponding to the family of log-concave distributions on $\R^p$. 
It is conjectured \cite{kannan1995isoperimetric} that in fact $\Cheeger=\Theta(1)$ (the KLS conjecture).
% While (\ref{eq:cheeger-bound-main-paper}) implies that $\Cheeger\ge p^{-o(1)}$, it is conjectured that $\Cheeger=\Omega(1)$ (the KLS conjecture). 
For background on the KLS conjecture and its consequences, see \cite{lee2018stochastic,chen2021almost}.
% For more details, see 
% %Appendix Section
% ~\ref{sec:appendix-more-prelims-lc}.
Lastly, 
% we state the following implication of 
\cite[Lemma 5.5]{lovasz2007geometry} implies the following:
\begin{lemma}\label{lem:logconcave-smallball}
    There is a universal $0<C_0\le 2$ so that every isotropic log-concave $Y$ satisfies the SBP with constant $C_0$.
\end{lemma}

\section{Applications to Robust Sparse Covariance and Inverse Covariance Estimation}
\label{sec:app-covariance-est}

%Before presenting the proofs of our main results, we describe their application to two problems of interest in robust statistics: estimation 
%of a sparse high dimensional shape matrix, and estimation of the inverse of a covariance matrix, also assuming it is sparse. 

As mentioned in the introduction, 
robust estimation of the covariance and inverse covariance matrices, given possibly heavy tailed data are important tasks in statistics. In high dimensional settings, these matrices are often assumed to be sparse, allowing their estimation from a limited number of samples. 
A common model for multivariate heavy tailed data, under which various estimators were derived and analyzed is to assume that the samples are elliptically distributed \cite{kelker1970distribution,cambanis1981theory,frahm2004generalized,fang2018symmetric}.
Specifically, 
%The elliptical distribution is a popularly-used model for heavy-tailed data \cite{kelker1970distribution,cambanis1981theory,frahm2004generalized,fang2018symmetric}.
a random vector $\tilde{X}$ follows an elliptical distribution with mean $\bm{\mu}$ and shape matrix $\trueCov \in S^p_{++}$ if it has the form
\begin{equation}
        \label{eq:Elliptical}
        \tilde{X} =  \bm{\mu} + z\trueCov^{1/2}Y \,,
\end{equation}
where $Y \sim \mathrm{Unif}(\Sp)$, 
% $\Sp$ being the $p$-dimensional unit sphere, 
and $z$ is a strictly positive random variable which is {independent} of $Y$ (but otherwise arbitrary). To make the model (\ref{eq:Elliptical}) identifiable, the scaling $p^{-1}\tr\trueCov = 1$ is assumed.  

The authors of \cite{goes2017robust} considered the problem of shape matrix estimation under a sparsity constraint. They showed
that given (possibly heavy tailed)
elliptical samples, a sparse
$\trueCov$ may nonetheless be estimated at the same rate as if one had sub-Gaussian samples with covariance $\trueCov$. Their estimator is remarkably simple: compute Tyler's M-estimator  corresponding to the $n$ samples, and then threshold its entries as proposed by \cite{bickel2008covariance}. 

In this section, building upon the theorems of Section \ref{sect:main-results}, we show that the above approach yields accurate estimates for
heavy tailed distributions 
%Our main results imply that the findings of \cite{goes2017robust} hold more generally 
beyond the elliptical model (\ref{eq:Elliptical}). 
Specifically, the vector $Y$ in (\ref{eq:Elliptical}) may be replaced by any isotropic random vector satisfying the assumptions of Section~\ref{sect:main-results}. Furthermore, we show a similar result for estimating the inverse shape matrix $\trueCov^{-1}$ assuming it is sparse.
For simplicity, we assume that $\tilde X$ is zero mean with
$\bm{\mu} = \bm{0}$ whereas in Section
\ref{sec:discussion} we discuss how this restriction may be overcome. 

%%%%%%%%%%%%%%%%%%%%%%%%%%%%%%%%%
\subsection{Sparse shape matrix estimation}
\label{sec:app-sparse-cov}

Observe that Tyler's M-estimator, Eq. (\ref{eq:def_tyler}),  is invariant to an arbitrary scaling of the samples. Consequently, Tyler's estimator computed from the elliptically-distributed samples $\tilde{\x}_i=z_i\trueCov^{1/2}\y_i$, call it $\SigmaTyl$, is \emph{exactly the same} as the estimator computed from the rescaled samples $\x_i=\trueCov^{1/2}\y_i$. We emphasize that the rescaled vectors $\x_1,\ldots,\x_n$ are \emph{not}  available to the estimator.
Denote $S=n^{-1}\sum_{i=1}^n \x_i\x_i^\T$, and $\SigmaTyl = n^{-1}\sum_{i=1}^n \wTyl_i \x_i\x_i^\T$. For a matrix $M$, denote the norms,
\[
\|M\|_{\max} = \max_{i,j}|M_{ij}|,\quad \|M\|_{1}=\sum_{i,j}|M_{ij}| \,.    
\]
We start with the following important lemma, which asserts that $\SigmaTyl$ is close to $\trueCov$ entrywise:
\begin{lemma}\label{lem:cov-estimation-application}
    Assume the setting of Theorem~\ref{thm:main_tyler}, recalling the normalization $\trCov=1$. There is $C>0$, that depends on the distribution of $Y$ and on $\gamma$, such that w.p. $1-o(1)$, the following holds. 
    \begin{enumerate}
        \item 
        % Under the assumption that $Y$ satisfies \AssumLC, 
        %     it holds that 
        Assume that $Y$ satisfies~\AssumLC. Then
        $\|\SigmaTyl-\trueCov\|_{\max} \le C\|\trueCov\|\frac{\log p}{\Cheeger\sqrt{n}}$.
        % \[
        % \|S-\trueCov\|_{\max} \le Cs_{\max} \frac{\log p}{\Psi_n\sqrt{n}},\quad \|S-c_0\SigmaMar\| \le   Cs_{\max} \frac{\log p}{\Psi_n\sqrt{n}} \,.
        % \]
        \item 
        % Under the assumption that $Y$ satisfies \AssumIndep~ or \AssumCCPSB, it holds that
        Assume that $Y$ satisfies either~\AssumIndep~or~\AssumCCPSB. Then
        $\|\SigmaTyl-\trueCov\|_{\max} \le C\|\trueCov\|\sqrt{\frac{\log p}{n}}$.
        % \[
        %     \|S-\trueCov\|_{\max} \le Cs_{\max} \sqrt{\frac{\log p}{n}},\quad \|S-c_0\SigmaMar\| \le   Cs_{\max} \sqrt{\frac{\log p}{n}} \,.
        %     \]
    \end{enumerate}
\end{lemma}
\begin{proof}
    %(Sketch.)
    By the triangle inequality, $\|\SigmaTyl-\trueCov\|_{\max} \le \|\SigmaTyl-S\|_{\max} + \|S-\trueCov\|_{\max}$. We bound the second term, by  Lemma~\ref{lem:sc-entrywise} in the Appendix, which bounds the deviations of the entries of $S$ about their expectation $\trueCov$. For the first term, 
    \[
    \Revision{ \|\SigmaTyl-S\|_{\max} \le   } 
    \|\SigmaTyl-S\| = \left\| \frac1n\sum_{i=1}^n (\wTyl_i-1)\x_i\x_i^\T  \right\| \le \max_{1\le i \le n} |\wTyl_i-1|\cdot \|S\| \,.  
    \]
    By  Lemma~\ref{lem:main-smax} in the Appendix, $\|S\| \lesssim \|\trueCov\|$ w.h.p., whereas $\max_{1\le i \le n} |\wTyl_i-1|$ may be bounded by Theorem~\ref{thm:main_tyler}.
\end{proof}
Following Bickel and Levina \cite{bickel2008covariance}, consider the class of approximately sparse covariance matrices, $q\in [0,1)$,
\begin{equation}
    \label{eq:sparse-cov-family}
    \mathcal{U}_p(q,s_q(p)) = \left\{ \Sigma\in \mathcal{S}_{++}^{p}\;:\;\,\sum_{j=1}^p |\Sigma_{ij}|^q \le s_q(p),\,1\le i\le p \right\}     \,.
\end{equation}
% The authors of \cite{bickel2008covariance} proposed an estimator for the class $\mathcal{U}_p(q,s_q(p))$ that works by thresholding the entries of the sample covariance $S$. In \cite{goes2017robust}, the authors proposed to replace the sample covariance by a robust covariance estimator (specifically, Tyler's M-estimator). They show, assuming $Y$ follows an elliptical distribution, that the resulting procedure has an optimal error rate in $n,p$. We now genralize this result.
Let $\mathcal{T}_t$ be the entry-wise hard-thresholding operator $\mathcal{T}_t(M)_{ij} = M_{ij}\indic{|M_{ij}|\ge t}$. 
The authors of \cite{bickel2008covariance} showed that to accurately estimate a matrix $\trueCov\in \mathcal{U}_p(q,s_q(p))$ with respect to operator norm, it suffices to construct a matrix $A$ which is close to $\trueCov$ entrywise, and then threshold it. We cite the following form of their result, as stated in \cite[Lemma 6]{goes2017robust}:

\begin{lemma}\label{lem:bickel}
    Let $\trueCov\in \mathcal{U}_p(q,s_q(p))$ and $A$ such that $\|A-\trueCov\|_{\max}\le \varepsilon$. There is a threshold $t=c_1\varepsilon$ so that for some $c_2=c_2(q)$, $\|\mathcal{T}_t(A)-\trueCov\|\le c_2s_q(p)\varepsilon^{1-q}$.
\end{lemma}
Combining Lemmas~\ref{lem:cov-estimation-application} and \ref{lem:bickel} yields the following generalization of \cite[Theorem 1]{goes2017robust}, \Revision{which proved a similar result under the more restrictive assumption that $Y$ has an elliptical distribution:}
\begin{corollary}
    There are $c,C>0$, that may depend on the distribution of $Y$ and on $\gamma$, such that the following holds. For an appropriately chosen threshold $t$, 
    %(per Lemmas~\ref{lem:bickel} and \ref{lem:cov-estimation-application}), 
    if $\trueCov\in \mathcal{U}_p(q,s_q(p))$, then w.p. $1-o(1)$,
    \begin{enumerate}
        \item 
        Assume that $Y$ satisfies~\AssumLC. Then
        % Under the assumption that $Y$ satisfies \AssumLC, 
        %     it holds that  
        $\|\trueCov-\mathcal{T}_t(\SigmaTyl)\| \le Cs_q(p)\|\trueCov\|^{1-q}\left( \frac{(\log p)^2}{\Cheeger^2 n} \right)^{(1-q)/2}$. 
        \item 
        % Under the assumption that $Y$ satisfies \AssumIndep~or \AssumCCPSB, then
        Assume that $Y$ satisfies either~\AssumIndep~or~\AssumCCPSB. Then
        $\|\trueCov-\mathcal{T}_t(\SigmaTyl)\| \le Cs_q(p)\|\trueCov\|^{1-q}\left( {\frac{\log p}{n}} \right)^{(1-q)/2}$.
    \end{enumerate}
\end{corollary}
Under \AssumIndep~and \AssumCCPSB, the attained rate is minimax optimal in $n,p$ \cite{cai2012optimal}.

\subsection{Sparse inverse shape matrix estimation}
\label{sec:app-sparse-inv}

We now consider the problem of estimating $\trueCov^{-1}$, assuming that it is sparse: $\trueCov^{-1}\in \mathcal{U}_p(q,s_q(p))$. 
Cai {\it et. al.}  proposed the CLIME estimator \cite{cai2011constrained}, which solves a linear program of the form:
\begin{equation}\label{eq:CLIME}
    \min_{\Omega} \|\Omega\|_1 \quad\textrm{subject to}\quad \|\hat{S}\Omega-\Id\|_{\max}\le \lambda \,,
\end{equation}
where $\lambda$ is a tuning parameter and $\hat{S}$ is a proxy for $\trueCov$ (\cite{cai2011constrained} propose to use the data sample covariance matrix). Having solved (\ref{eq:CLIME}), a symmetrization step is applied to get the final estimator. 
\cite[Theorem 6]{cai2011constrained} states that if $\hat{S}$ is close entrywise to $\trueCov$, then the estimator $\hat{\Omega}=\hat{\Omega}(\hat{S})$ is close in operator norm to $\trueCov^{-1}$:
\begin{lemma}\label{lem:cai}
    Suppose that $\trueCov^{-1}\in \mathcal{U}_p(q,s_q(p))$ and $\hat{S}$ satisfies $\|\trueCov-\hat{S}\|_{\max}\le \varepsilon$. For any $\lambda\ge \|\trueCov^{-1}\|_1\varepsilon$, the estimator obtained by solving (\ref{eq:CLIME}) and applying symmetrization satisfies $\|\hat{\Omega}-\trueCov^{-1}\| \le Cs_q(p)\|\trueCov^{-1}\|_1^{1-q}\lambda^{1-q}$. 
\end{lemma}
Setting $\hat{S}=\SigmaTyl$ in (\ref{eq:CLIME}) and using Lemma~\ref{lem:cov-estimation-application}, yields:
\begin{corollary}\label{cor:precision}
    There are $c,C>0$, that may depend on the distribution of $Y$ and on $\gamma$, such that the following holds. For an appropriately chosen $\lambda$, 
    %(per Lemmas~\ref{lem:cai} and \ref{lem:cov-estimation-application}), 
    if $\trueCov^{-1}\in \mathcal{U}_p(q,s_q(p))$, then w.p. $1-o(1)$,
    \begin{enumerate}
        \item 
        % Under the assumption that $Y$ satisfies \AssumLC, 
        %     it holds that
        Assume that $Y$ satisfies~\AssumLC. Then $\|\trueCov^{-1}-\hat{\Omega}(\SigmaTyl)\| \le Cs_q(p)\|\trueCov^{-1}\|_1^{2-q}\|\trueCov\|^{1-q}\left( \frac{(\log p)^2}{\Cheeger^2 n} \right)^{(1-q)/2}$. 
        \item 
        % If $Y$ satisfies \AssumIndep~or \AssumCCPSB, then it holds that
        Assume that $Y$ satisfies either~\AssumIndep~or~\AssumCCPSB. Then
        $\|\trueCov^{-1}-\hat{\Omega}(\SigmaTyl)\| \le Cs_q(p)\|\trueCov^{-1}\|_1^{2-q}\|\trueCov\|^{1-q}\left( {\frac{\log p}{n}} \right)^{(1-q)/2}$.
    \end{enumerate}
\end{corollary}
Lastly, we remark that CLIME is known to have a sub-optimal rate for i.i.d. sub-Gaussian data. In \cite{cai2016estimating}, the minimax rate is computed and a rate-optimal adaptive estimator ACLIME is proposed. 
While beyond the scope of the present paper, we conjecture that 
 %and we suspect that
% While we suspect that with some effort, 
a similar result, as Corollary~\ref{cor:precision}, may be derived for ACLIME as well. 

\section{Proofs}
\label{sec:proofs}

Recall that the input data consists of $n$ samples $\x_i=\trueCov^{1/2}\y_i$ with $\y_i$ isotropic. Denote by $S$ (resp. $T$) the sample covariance corresponding to the $\x_i$-s (resp. $\y_i$-s),
\[
S = \frac{1}{n}\sum_{i=1}^{n} \x_i \x_i^\top = \trueCov^{1/2} T \trueCov^{1/2} \,,\quad T = \frac1n \sum_{i=1}^n \y_i \y_i^\top .
\]
% By definition, $S$ and $T$ related to one another by 
% Clearly, 
% $S = \trueCov^{1/2} T \trueCov^{1/2}$. 
For $1\le j \le n$, denote by $S_{-j}$ (similarly $T_{-j}$) the sample covariance of $(n-1)$ samples excluding $\x_j$,
\[
S_{-j} = S - \frac1n \x_i\x_i^\T = \frac1n \sum_{i=1,i\ne j}^n \x_i \x_i^\top \,.
\]
     
The following lemma, proven in \ref{sec:proof-lem:main-technical}, is key to the analysis
of Maronna's and Tyler's estimators. 
  
\begin{lemma}\label{lem:main-technical}
    Assume  $\gamma<1$. There are 
        $c,C,\epsilon_0>0$, that depend on the distribution of $Y$ and on $\gamma$, so that for all $\epsilon<\epsilon_0$: 
    % Suppose that $\gamma=p/n<1$. Denote by $\Ec_{\mathrm{Lem.}\ref*{lem:main-technical}}(\epsilon)$ the event that $\max_{1\le i \le n}|p^{-1}\x_i^\T S^{-1}\x_i - 1|\le \epsilon$. There are 
    %     $c,C,\epsilon_0>0$, that depend on the distribution of $Y$ and on $\gamma$, such that for all $\epsilon<\epsilon_0$:         
        
        \begin{enumerate}
                
                \item 
                % If $Y$ is a log-concave random vector, then
                Assume \AssumLC. Then
                $\Pr(\max_{1\le i\le n}|p^{-1}\x_i^\T S^{-1}\x_i - 1|\ge \epsilon) \le Cn^2e^{-c(\Cheeger\sqrt{n}) \epsilon}$.
                % $\Pr(\Ec_{\mathrm{Lem.}\ref*{lem:main-technical}}(\epsilon)^c) \le Cn^2e^{-c(\Cheeger\sqrt{n}) \epsilon}$.
                % \begin{equation}
                % \Pr\left(\max_{1\le i \le n}\left|\frac{1}{p}\x^\top_iS^{-1}\x_i - 1\right| \geq \epsilon\right) \leq  Cn^2e^{-c(\Cheeger\sqrt{n}) \epsilon} \,.
                % \end{equation}

                \item 
                Assume \AssumIndep~ or \AssumCCPSB. Then 
                $\Pr(\max_{1\le i\le n}|p^{-1}\x_i^\T S^{-1}\x_i - 1|\ge \epsilon) \le C n^2 e^{-cn\epsilon^2}$.
                % $\Pr(\Ec_{\mathrm{Lem.}\ref*{lem:main-technical}}(\epsilon)^c) \le C n^2 e^{-cn\epsilon^2}$.
                % \begin{equation}
                % \label{eq:main-lemma-ccp}
                % \Pr\left(\max_{1\le i \le n}\left|\frac{1}{p}\x^\top_iS^{-1}\x_i - 1\right| \geq \epsilon\right) \leq  C n^2 e^{-cn\epsilon^2} \,.
                % \end{equation}
        \end{enumerate}
\end{lemma}
Note that $\x_i^\T S^{-1}\x_i = \y_i^\T \trueCov^{1/2}(\trueCov^{1/2}T\trueCov^{1/2})^{-1}\trueCov^{1/2}\y_i=\y_i^\T T^{-1}\y_i$, so the quadratic form $p^{-1}\x_i^\T S^{-1}\x_i$ does not depend on $\trueCov$. 
Assuming that $\x_i$-s are Gaussian, a similar result 
was derived 
in \cite{zhang2016marvcenko}.
Their proof relies on the orthogonal invariance of the isotropic Gaussian distribution, and does not generalize to other distributions.
Lemma~\ref{lem:main-technical} implies
that 
$\max_{1\le i\le n}|\x_i^\top S^{-1} \x_i-1|\to 1$ w.p. $1$ in the asymptotic limit $n,p\to\infty$ with $\frac{p}{n}\to \gamma$.
Assuming $Y$ has independent entries with finite fourth moment, this asymptotic result was proven in \cite{couillet2014robust}.

\subsection{Proof of Theorem \ref{thm:main_maronna} (Maronna's M-Estimator)}
\label{sec:proof-thm:main_maronna}

% Throughout this section, we always operate under the assumptions of Theorem~\ref{thm:main_maronna}.
% with the concentration result for $\frac1p \x_i^\top S^{-1}\x_i$ of Lemma~\ref{lem:main-technical}. 

\paragraph{Existence} 
We first prove that a solution to the fixed point equation (\ref{eq:def_maronna}) indeed exists. 
This proof follows that of  \cite{couillet2014robust}. We briefly describe it, for the sake of completeness.  
Denote  $\mathbf{d}=(d_1,\ldots,d_n)$. 
Consider the function $h:\R_+^n \to \R_+^n$,
% whose $j$-th coordinate is given by
\begin{equation}
        \label{eq:h_j}
h_j(\mathbf{d})=\frac1p \x_j^\top \left(\frac1n \sum_{i=1}^n u(d_i)\x_i\x_i^\top \right)^{-1}\x_j ,\quad 1\le j \le n \,.
\end{equation}
Since by assumption $X$ has a density and $n>p$,
% Under our distributional assumptions on the $n$ samples $\x_i$, the function 
$h$ is well-defined w.p. 1.
% with probability one. 
Moreover, 
a vector $\bm{d}$ yields a Maronna's estimator with weights $w_j=u(d_j)$ if and only if $d_j=h_{j}(\mathbf{d})$ for all $j$. Hence, to establish the existence of Maronna's estimator, we need to show that $h$ 
has a fixed point. 
To this end, as proven in  \cite{couillet2014robust}, the function  $h(\cdot)$ satisfies the following three properties. 1) Positivity: $h(\mathbf{d})> \mathbf 0$ for any vector $\mathbf{d}\geq \mathbf0$, namely with $d_j\geq 0$ 
for all $j$; 2) Monotonicity: if $\mathbf{d}\geq \mathbf{d}'\geq \mathbf0$, then $h(\mathbf{d})\ge h(\mathbf{d}')$; 3) Scalability: for any $\alpha>1$, $\alpha h(\mathbf{d}) \ge h(\alpha \mathbf{d})$. 
% \begin{enumerate}
%         \item Positivity: $h(\mathbf{d})> \mathbf 0$.
%         \item Monotonicity: if $\mathbf{d}\geq \mathbf{d}'\geq \mathbf0$, then $h(\mathbf{d})\ge h(\mathbf{d}')$. 
%         \item Scalability: for any $\alpha>1$, $\alpha h(\mathbf{d}) \ge h(\alpha \mathbf{d})$.  
% \end{enumerate}
 A function $h(\cdot)$ 
%  that satisfies these properties 
satisfying these properties
 is (almost) a standard interference function, in the sense of \cite{yates1995framework}. By \cite[Theorem 1]{yates1995framework}, 
%  states that 
 if there exists some $\mathbf{d}$ with $\mathbf{d}\geq h(\mathbf{d})$, then $h(\cdot)$ has a fixed point.
%  applied to (\ref{eq:h_j}), this implies that Maronna's estimator exists. The next lemma show that w.h.p., this is indeed the case:
 The following lemma shows that w.h.p., such a vector ${\bf d}$ exists, which in turn implies the existence of Maronna's estimator.
%   under our distributional assumptions on $Y$. 

\begin{lemma}
        \label{lem:maronna-existence}
        % Let $\Ec_{\mathrm{Lem.}\ref*{lem:maronna-existence}}$ be the event that there exists some $\bm{d}$ with $\bm{d}\ge h(\bm{d})$. 
        There are $c,C>0$, that depend on the distribution of $Y$ and on $\gamma$, such that
        \begin{enumerate}
                \item Assume \AssumLC. 
                % Then $\Pr(\Ec_{\mathrm{Lem.}\ref*{lem:maronna-existence}}^c) \le Cn^2e^{-c(\Cheeger\sqrt{n})} $.
 $\Pr\left(\exists \bm{d}>\bm{0}\;:\;\bm{d}\ge h(\bm{d})\right) \ge 1-Cn^2e^{-c(\Cheeger\sqrt{n})} $.
                % w.p.  $\ge 1-Cn^2e^{-c(\Cheeger\sqrt{n})}$, there exists some $\mathbf{d}$ with $\mathbf{d}\geq h(\mathbf{d})$.
                \item Assume \AssumIndep~ or \AssumCCPSB. Then      
            $\Pr\left(\exists \bm{d}>\bm{0}\;:\;\bm{d}\ge h(\bm{d})\right) \ge 1-Cn^2e^{-cn} $.
                % $\Pr(\Ec_{\mathrm{Lem.}\ref*{lem:maronna-existence}}^c) \le Cn^2e^{-cn} $.
                % % w.p. $\ge 1-Cn^2e^{-cn}$, there exists some $\mathbf{d}$ with $\mathbf{d}\geq h(\mathbf{d})$.
        \end{enumerate}
\end{lemma}
\begin{proof}
        By assumption (i) of Theorem \ref{thm:main_maronna}, $\phi_{\infty}>1$.         Hence, we may take some $d_0>0$ such that $\phi(d_0)>1$. 
        Setting $\bm{d}_0=d_0\bm{1}=(d_0,\ldots,d_0)$ in (\ref{eq:h_j}) and using $u(d)=\phi(d)/d$, 
        one may verify that $h_j(\bm{d}_0)=\frac{d_0}{\phi(d_0)}\frac1p\x_j^\top S^{-1}\x_j$. By Lemma~\ref{lem:main-technical}, w.h.p., $\max_{1\le j\le n}\frac1p\x_j^\top S^{-1}\x_j\le \phi(d_0)$, and so $h(\bm{d}_0)\le \bm{d}_0$.
        % \[
        % h_j(\mathbf{d}) = \frac{d}{\phi(d)}\cdot\frac1p \x_j^\top  \left(\frac1n \sum_{i=1}^n \x_i\x_i^\top\right)^{-1}\x_j =\frac{d}{\phi(d)}\frac1p\x_j^\top S^{-1}\x_j\,.
        % \]
        % Now, by Lemma~\ref{lem:main-technical}, with high probability the following holds for all $1\leq j\leq n$, 
        % \[
        % \frac1p \x_j^\top S^{-1} \x_j < 1 + \min\left(\epsilon_0, \phi(d)-1\right)\leq \phi(d) \,.
        % \]
        % In other words, with high probability,  ${\bf d}\geq h({\bf d})$. 
        \end{proof}

\paragraph{Uniqueness and concentration}
% Suppose that 
Let $\mathbf{d}=(d_1,\ldots,d_n)$ be a vector satisfying $h(\mathbf{d})=\mathbf{d}$, which yields a valid Maronna's estimator. Next, we prove that only one such $\bm{d}$ exists. To this end,
% so $w_i=u(d_i)$ are the weights of the corresponding Maronna's estimator(s). 
denote $j_{\min} = \argmin_{1\le j \le n} d_j$, $j_{\max} = \argmax_{1\le j \le n} d_j$.
% \[
% j_{\min} = \argmin_{j=1,\ldots,n} d_j, \quad\,\quad j_{\max} = \argmax_{j=1,\ldots,n} d_j \,.
% \]
Since $h(\cdot)$ is non-decreasing, $d_j=h_j(\bm{d})\le h_j(d_{j_{\max}}\bm{1}) = (u(d_{j_{\max}}))^{-1}p^{-1}\x_j^\T S^{-1}\x_j$. Setting $j=j_{\max}$ and multiplying by $u(d_{j_{\max}})$ gives $\phi(d_{j_{\max}}) \le p^{-1}\x_{j_{\max}}^\T S^{-1}\x_{j_{\max}}$. 
% A similar calculation, starting from
Similarly, $h(\bm{d})\ge h(d_{j_{\min}}\bm{1})$, and thus $\phi(d_{j_{\min}}) \ge p^{-1}\x_{j_{\min}}^\T S^{-1}\x_{j_{\min}}$. Finally, since $\phi$ is non-decreasing, we deduce that for all $j$, $p^{-1}\x_{j_{\min}}^\T S^{-1}\x_{j_{\min}}\le \phi(d_j) \le p^{-1}\x_{j_{\max}}^\T S^{-1}\x_{j_{\max}}$. 
% In particular, this implies
Consequently,
\begin{equation}\label{eq:maronna-1}
    \max_{1\le j \le n}|\phi(d_j)-1| \le \max_{1\le j \le n}|p^{-1}\x_j^\T S^{-1}\x_j-1| \,.
\end{equation}
The next lemma shows that w.h.p., $h(\cdot)$ cannot have any fixed points whose entries are far from the constant $\phi^{-1}(1)$:
% \begin{align*}
%         d_j 
%         &= \frac1p \x_j^\top \left(\frac1n \sum_{i=1}^n u(d_i)\x_i\x_i^\top\right)^{-1} \x_j \le \frac1p \x_j^\top \left(\frac1n \sum_{i=1}^n u(d_{j_{\max}})\x_i\x_i^\top\right)^{-1} \x_j 
%         = \frac{1}{u(d_{j_{\max}})} \cdot \frac1p \x_j^\top S^{-1} \x_j \,.
% \end{align*}
% In particular, taking $j=j_{\max}$ and multiplying by $u(d_{j_{\max}})$, gives $\phi(d_{j_{\max}}) \le \frac1p \x_{j_{\max}}^\top S^{-1} \x_{j_{\max}}$. 
% % \[
% % \phi(d_{j_{\max}}) \le \frac1p \x_{j_{\max}}^\top S^{-1} \x_{j_{\max}} \,.
% % \]
% A similar calculation, this time replacing all the $d_i$-s by $d_{j_{\min}}$, gives the lower bound $\phi(d_{j_{\min}}) \ge \frac1p \x_{j_{\min}}^\top S^{-1} \x_{j_{\min}}$. 
% % \[
% % \phi(d_{j_{\min}}) \ge \frac1p \x_{j_{\min}}^\top S^{-1} \x_{j_{\min}} \,.
% % \]
% Since $\phi(\cdot)$ is non-decreasing, we deduce that for all $j=1,\ldots,n$, 
% \begin{equation}
% \label{eq:maronna-bound}
% \frac1p \x_{j_{\min}}^\top S^{-1} \x_{j_{\min}} \le \phi(d_{j}) \le \frac1p \x_{j_{\max}}^\top S^{-1} \x_{j_{\max}}.
% \end{equation}

\begin{lemma}
        \label{lem:maronna-conc}
%        Denote by $\Ec(\epsilon)=\{\exists{\bf d} \mbox{ such that } \}$
% Denote by $\Ec(\epsilon)$ the event that there is a fixed point $\mathbf{d}=h(\mathbf{d})$ with $\norm{\mathbf{d}-\phi^{-1}(1)\oneVec}_\infty > \epsilon$. 
For $\epsilon>0$, let
$\mathcal{F}_{\mathrm{bad}}(\epsilon)$ the set of ``bad'' fixed points of $h(\cdot)$, namely, such that  
% $\Ec_{\mathrm{Lem.}\ref*{lem:maronna-conc}}(\epsilon)$ be the event that there exist some (``bad'') fixed point $\bm{d}=h(\bm{d})$ with 
${\|\bm{d}-\phi^{-1}(1)\bm{1}\|_{\infty}\ge \epsilon}$. 
There $c,C,\epsilon_0>0$,  that depend on the distribution of $Y$ and on $\gamma$, such that for all  $\epsilon<\epsilon_0$,
        \begin{enumerate}
                \item Assume \AssumLC. Then
                % $
                % \Pr\left(\Ec_{\mathrm{Lem.}\ref*{lem:maronna-conc}}(\epsilon)\right) \le Cn^2e^{-c(\Cheeger\sqrt{n})\epsilon}.
                % $
                $
                \Pr\left(\mathcal{F}_{\mathrm{bad}}(\epsilon)\ne \emptyset\right) \le Cn^2e^{-c(\Cheeger\sqrt{n})\epsilon}.
                $
                \item Assume \AssumIndep~ or \AssumCCPSB. Then
                % $
                % \Pr\left(\Ec_{\mathrm{Lem.}\ref*{lem:maronna-conc}}(\epsilon)\right) \le Cn^2e^{-cn\epsilon^2}.
                % $
                $
                \Pr\left(\mathcal{F}_{\mathrm{bad}}(\epsilon)\ne \emptyset \right) \le Cn^2e^{-cn\epsilon^2}.
                $
        \end{enumerate}
\end{lemma}
\begin{proof}
Follows immediately by Lemma~\ref{lem:main-technical}, Eq. (\ref{eq:maronna-1}) and Assumption (ii) of Theorem~\ref{thm:main_maronna}, which states that $\phi$ is invertible in a neighborhood of $\phi^{-1}(1)$ and that the inverse $\phi^{-1}$ is locally Lipshitz at $1$.
% Let $\epsilon>0$ be sufficiently small. By Lemma~\ref{lem:main-technical} and Eq. (\ref{eq:maronna-bound}), with high probability,
% \[
% 1-\epsilon \le \phi(d_j) \le 1+\epsilon \quad \forall j=1,\ldots,n.
% \]   
% By assumption, $\phi(\cdot)$ is strictly increasing at $\phi^{-1}(1)$, with $\phi^{-1}(\cdot)$ defined and Lipschitz in a neighborhood of $\phi^{-1}(1)$. Hence, for sufficiently small $\epsilon$, 
% $|d_j-\phi^{-1}(1)|\le O(\epsilon)$ for all $j$.    
\end{proof}

% We now prove that the solution to Eqs. (\ref{eq:def_maronna}) and (\ref{eq:def_maronna_weights}) is unique with high probability.
Lastly, we prove that Maronna's estimator exists uniquely w.h.p:
\begin{lemma}
        \label{lem:maronna-unique}
        % Let $\Ec_{\mathrm{Lem.}\ref*{lem:maronna-unique}}$ be the event that Maronna's estimator exists uniquely. 
        There are $c,C>0$, that depend on the distribution of $Y$ and on $\gamma$, such that 
        \begin{enumerate}
                \item Assume~\AssumLC. Then
               $\Pr\left(\textrm{ME exists uniquely}\right) \ge 1- Cn^2e^{-c(\Cheeger\sqrt{n})}$.
            %   $\Pr(\Ec_{\mathrm{Lem.}\ref*{lem:maronna-unique}}^c) \le Cn^2e^{-c(\Cheeger\sqrt{n})}$.
                % Maronna's estimator exists and is unique w.p. $\ge 1-Cn^2e^{-c(\Cheeger\sqrt{n})}$.
                \item Assume~\AssumIndep~ or \AssumCCPSB. Then
                $\Pr\left(\textrm{ME exists uniquely}\right) \ge 1- Cn^2e^{-cn}$.
                % $\Pr(\Ec_{\mathrm{Lem.}\ref*{lem:maronna-unique}}^c) \le Cn^2e^{-cn}$.
                % Maronna's estimator exists and is unique w.p. $\ge 1-Cn^2e^{-cn}$.
        \end{enumerate}
\end{lemma}  
\begin{proof}
        Existence, w.h.p., is guaranteed by Lemma~\ref{lem:maronna-existence}. Assume by contradiction that $h$ has two different fixed points, $\mathbf{d} \ne \mathbf{d}'$. Take $\alpha = \max_{j} (d_j/d_j')$, and let $l$ be a coordinate where the maximum is attained. Assume w.l.o.g. that $\alpha>1$ 
(otherwise replace  $\mathbf{d}$ with $\mathbf{d}'$), and note that by its definition, $\mathbf{d}\leq \alpha\mathbf{d}'$.       

%Then $d_l = \alpha d_l'$, and for all $k\ne l$, $d_k \le \alpha d_k'$.

        Let $\eta>0$ be so that $\phi(\cdot)$ is strictly increasing in $[\phi^{-1}(1)-\eta,\phi^{-1}(1)+\eta]$. 
By Lemma~\ref{lem:maronna-conc}, w.h.p. both $\mathbf{d},\mathbf{d}' \in [\phi^{-1}(1)-\eta/2,\phi^{-1}(1)+\eta/2]^n$. In addition, since $\alpha>1$, then
% . Note that under this event, since $\alpha>1$,  
$\phi(\alpha d_j') > \phi(d_j')$ for all $j$. This, in turn, implies that
\[
u(\alpha d_j')=\frac{\phi(\alpha d_j')}{\alpha d_j'} >\frac{\phi(d_j')}{\alpha d_j'} = \frac1\alpha u(d_j')\,.
\]
Plugging this into (\ref{eq:h_j}) gives $h(\alpha \bm{d}')<\alpha h(\bm{d}')$.  
% Combining the above with the definition of $h$, Eq. (\ref{eq:h_j}),
%   implies that      $h_l(\alpha \mathbf{d}') < \alpha h_l(\mathbf{d}')$. Since $\mathbf{d}\leq \alpha\mathbf{d}'$, the fact that $h(\cdot)$ is increasing implies $ h(\mathbf{d})\leq h(\alpha\mathbf{d}')$. 
  Since $\mathbf{d}\leq \alpha\mathbf{d}'$ and $h$ is non-decreasing, we deduce $h(\bm{d}) < \alpha h(\bm{d}')$. 
  But, since $\mathbf{d},\mathbf{d}'$ are fixed points of $h(\cdot)$, this yields a contradiction:
        $
        d_l = h_l({\bf d}) < \alpha h_l(\mathbf{d}') = \alpha d_l' = d_l \,.
        $
%       This means that $\mathbf{d}\ne \mathbf{d}'$ and 
%       $\mathbf{d},\mathbf{d}' \in [\phi^{-1}(1)-\eta/2,\phi^{-1}(1)+\eta/2]^n$
%        cannot occur simultaneously, with the latter 
%        occurring with high probability.
\end{proof}

\paragraph{Proof of Theorem~\ref{thm:main_maronna}} Existence and uniqueness follow from Lemma~\ref{lem:maronna-unique}.
By Lemma \ref{lem:maronna-conc}, the coordinates of the fixed point ${\bf d}$ concentrate around $\phi^{-1}(1)$. Hence, the weights of Maronna's estimator, $\wMar_i=u(d_i)$, concentrate around $u(\phi^{-1}(1))=\phi(\phi^{-1}(1))/\phi^{-1}(1)=1/\phi^{-1}(1)$.
\qed

\subsection{Proof of Theorem \ref{thm:main_tyler} (Tyler's M-Estimator)}
\label{sec:proof-Tyler}

Our proof combines the strategy of Zhang {\it et. al.} \cite{zhang2016marvcenko} with 
% the main concentration 
Lemma~\ref{lem:main-technical}.
Since $X$ has a density, by \cite[Theorems 1 and 2]{kent1988maximum} Tyler's estimator exists uniquely w.p. $1$. By \cite[Lemma 2.1]{zhang2016marvcenko}, its weights 
% of Tyler's estimator are given by
are
\begin{equation}\label{eq:what-to-wTyl}
    \wTyl_i = \frac{p\cdot \wh_i}{\Tr \left( \frac1n \sum_{i=1}^n \wh_i \x_i\x_i^\T  \right)}\,,
\end{equation}
where $\bwh = (\wh_1,\ldots,\wh_n)$ is the unique minimizer of 
\begin{equation}\label{eq:Tyl-what-objective}
    F(\w) = -\sum_{i=1}^n \log w_i + \frac{n}{p}\log \det \left(\sum_{i=1}^n w_i \x_i\x_i^\T \right),\quad \textrm{subject to } \w>\bm{0}\textrm{ and }\sum_{i=1}^n \w_i=n \,.
\end{equation}
As in \cite{zhang2016marvcenko}, the proof proceeds in two steps: 
(I) Show that $\wh_1,\ldots,\wh_n$ all concentrate around $1$; (II) Using (\ref{eq:what-to-wTyl}), deduce concentration for the weights $\wTyl_1,\ldots,\wTyl_n$. 

We start with the weights $\bwh$. 
The argument of \cite[Section 3.2]{zhang2016marvcenko} starts with the following observation:

\begin{lemma}\label{lem:Tyl-obs}
    Let $\beta>0$ be arbitrary. Then $\bwh$ is the unique stationary point of the following function $G_\beta:\R^n\to \R$:
    \begin{equation}\label{eq:Tyl-def-G}
        G_\beta(\w) = F(\w) + \frac{\beta}{2} \left( \sum_{i=1}^n w_i - n \right)^2 = -\sum_{i=1}^n \log w_i + \frac{n}{p}\log \det \left(\sum_{i=1}^n w_i \x_i\x_i^\T \right) + \frac{\beta}{2} \left( \sum_{i=1}^n w_i - n \right)^2 \,.
    \end{equation}
\end{lemma}
%See \cite[Section 3.2]{zhang2016marvcenko}. Thus, we need to study the fixed points of $G_\beta(\cdot)$, equivalently, the zeros of $g_\beta(\cdot)=\nabla G_\beta(\cdot)$.
Next, note that $\w=\oneVec=(1,\ldots,1)$ is ``almost'' a zero of $g_\beta(\cdot)=\nabla G_\beta(\cdot)$ (for any $\beta$), in the sense that w.h.p. $g_\beta(\oneVec)$ is small. Indeed, a straightforward calculation, Appendix Eq. (\ref{eq:Tyl-grad-G}), gives
\begin{equation}
    \label{eq:Tyler-value-at-one}
    \left(g_\beta(\oneVec)\right)_\ell = -1 + \frac{1}{p}\x_\ell^\T \left( \frac{1}{n}\sum_{i=1}^n\x_i\x_i^\T \right)^{-1}\x_\ell = -1 + \frac{1}{p}\x_\ell^\T S^{-1} \x_\ell \,\quad \textrm{for all } 1\le \ell \le n\,.
\end{equation}
By Lemma~\ref{lem:main-technical},
the right-hand-side of (\ref{eq:Tyler-value-at-one}) concentrates tightly around $0$. That is, the following deviation bound holds:
\begin{lemma}\label{lem:Tyl-almost-zero}
        There are $c,C,\epsilon_0>0$, that depend on the distribution of $Y$ and on  $\gamma$, so that for all $\epsilon<\epsilon_0$ and all $\beta$
        \begin{enumerate}
                
                \item 
                Assume \AssumLC. Then $\Pr\left(\|g_{\beta}(\oneVec)\|_\infty \geq \epsilon\right) \leq  Cn^2e^{-c(\Cheeger\sqrt{n}) \epsilon}$.

                \item 
                Assume either \AssumIndep~ or \AssumCCPSB. Then $\Pr\left(\|g_{\beta}(\oneVec)\|_\infty \geq \epsilon\right)  \leq  C n^2 e^{-cn\epsilon^2}$.
        \end{enumerate}
\end{lemma}

Next, we carry out a perturbation argument: we show that $\|g_\beta(\bm{1})\|_\infty$ being small implies that the unique root $g_\beta(\bwh)=\bm{0}$ is close to $\bm{1}$.
To this end, we use the following result,
 \cite[Lemma 3.1]{zhang2016marvcenko}. Below, for a matrix $M \in \R^{n\times n}$, we denote by $\|M\|_{\infty,\infty}=\max_{1\le j \le n} \sum_{\ell=1}^n |M_{j\ell}|$ its $\ell_\infty$-to-$\ell_\infty$ operator norm.
\begin{lemma}\label{lem:Tyl-perturb}
    Let $h:\R^p \to \R^p$ be differentiable, $\w_0\in \R^p$. Suppose that for
    some $L,R>0$ and $0<\epsilon \le \min\{R,L^{-1}\}$: 
\begin{enumerate}[label=(\Roman*)]
    \item ($\w_0$ is ``almost'' a zero). $\|h(\w_0)\|_\infty \le \frac12\epsilon$.
    \item (Non-degeneracy at $\w_0$). $\nabla h(\w_0) = \Id$, namely $\left(\nabla h\right)_{j,\ell}=\frac{\partial h_j}{\partial w_\ell} = \delta_{j,\ell}$. 
    \item (Smoothness around $\w_0$). $\|\nabla h(\w)-\nabla h(\w_0)\|_{\infty,\infty} \le L\|\w-\w_0\|_\infty$ for all $\w$ in ${\|\w-\w_0\|_\infty\le R}$.
\end{enumerate}
Then $h(\cdot)$ has a zero $\bwh$ close to $\w_0$ such that $\|\bwh-\w_0\|_\infty \le \epsilon$.  
\end{lemma}
% For completeness, we provide a short proof of Lemma~\ref{lem:Tyl-perturb} in 
% % Appendix, Section~
% \ref{sec:proof:lem:Tyl-perturb}.

To prove that $g_\beta$ has a zero near $\w_0=\bm{1}$ via Lemma~\ref{lem:Tyl-perturb}, we consider
% We would like to use Lemma~\ref{lem:Tyl-perturb} to 
% show that $g_\beta$ has a zero near $\oneVec$.
% In light of condition (II) above, consider
$h_\beta(\w)=\left(\nabla g_\beta(\oneVec)\right)^{-1} g_\beta(\w)$. Clearly, $g_\beta$ and $h_\beta$ have the same zeros. Also, by its definition, $h_\beta$ satisfies condition (II) above. 
We next show that w.h.p. $h_\beta$ satisfies condition (I). Since $\|h_\beta(\bm{1})\|_\infty \le \left\|\left(\nabla g_\beta(\oneVec)\right)^{-1}\right\|_{\infty,\infty}\|g_\beta(\oneVec)\|_\infty$, 
it suffices to bound the matrix norm $\left\|\left(\nabla g_\beta(\oneVec)\right)^{-1}\right\|_{\infty,\infty}$:
\begin{lemma}\label{lem:Tyl-inv-bound}
        There are $c,C,B,\beta_0>0$, that depend the distribution of $Y$ and on $\gamma$, so that setting $\beta=\beta_0/n$,
        \begin{enumerate}
             \item Assume ~\AssumLC. Then $\Pr\left( \|\left(\nabla g_\beta(\oneVec)\right)^{-1}\|_{\infty,\infty} \ge B \right) \le Cn^2 e^{-c\Cheeger n^{1/4}}$.
            % \[
            % \Pr\left( \|\left(\nabla g_\beta(\oneVec)\right)^{-1}\|_{\infty,\infty} \ge B \right) \le Cn^2 e^{-c\Cheeger^{2/3}n^{1/6}} \,.
            % \]
            \item Assume~\AssumIndep~or~\AssumCCPSB. Then $\Pr\left( \|\left(\nabla g_\beta(\oneVec)\right)^{-1}\|_{\infty,\infty} \ge B \right) \le Cn^2 e^{-cn^{1/2}}$.
            % \[
            % \Pr\left( \|\left(\nabla g_\beta(\oneVec)\right)^{-1}\|_{\infty,\infty} \ge B \right) \le Cn^2 e^{-cn^{1/4}} \,.
            % \]
        \end{enumerate}
\end{lemma}
We prove Lemma~\ref{lem:Tyl-inv-bound} in 
% Appendix, Section~
\ref{sec:proof:lem:Tyl-inv-bound}.
Our proof follows 
% the strategy of 
Zhang {\it et. al.} \cite[Lemma 3.3]{zhang2016marvcenko}. Their argument, however, contains a mathematical error that we correct. 

Lastly, we address condition (III) of Lemma~\ref{lem:Tyl-perturb}. We prove 
the following in
% Appendix Section~
\ref{sec:proof:lem:Tyl-smooth}:
\begin{lemma}\label{lem:Tyl-smooth}
        Let $L_g$ be the $\ell_\infty$ Lipschitz constant of $\nabla g_\beta$ on an $\ell_\infty$ ball of radius $\frac12$ around $\oneVec$:
        \[
        L_g = \max_{\w\,:\,\|\w-\oneVec\|_\infty\le \frac12} \frac{\|\nabla g_\beta(\w)-\nabla g_\beta(\oneVec)\|_{\infty,\infty}}{\|\w-\oneVec\|_\infty}\,.
        \]

        There are $c,C,L>0$, that depend on the distribution of $Y$ and on $\gamma$, so that
        \begin{enumerate}
            \item Assume \AssumLC. Then $\Pr(L_g\ge L) \le Cn^2e^{-c\Cheeger\sqrt{n}}$.
            \item Assume \AssumIndep~ or \AssumCCPSB. Then $\Pr(L_g\ge L) \le Cn^2e^{-cn}$.
        \end{enumerate}
\end{lemma}

Equipped with the preceding lemmas, we are ready to prove our concentration result for $\bwh$:
\begin{lemma}\label{lem:Tyl-conc-what}
Let $\bwh=(\wh_1,\ldots,\wh_n)>\bm{0}$ be the unique minimizer of (\ref{eq:Tyl-what-objective}).  
There are $c,C,\epsilon_0>0$, that depend on the distribution of $Y$ and on $\gamma$, such that for all $\epsilon<\epsilon_0$,
\begin{enumerate}
    \item Assume~\AssumLC. Then $\Pr\left( \max_{1\le i \le n}\left| \wh_i-1\right| \ge \epsilon \right) \le Cn^2 e^{-c \Cheeger\min\left\{ \sqrt{n}\epsilon,\,  n^{1/4}\right\} } $.
    % \[
    % \Pr\left( \max_{1\le i \le n}\left| \wh_i-1\right| \ge \epsilon \right) \le Cn^2 e^{-c \min\left\{ \left(\Cheeger\sqrt{n}\right\}\epsilon,\, \Cheeger^{2/3}n^{1/6}\right)} \,.
    % \]
    \item Assume~\AssumIndep~or~\AssumCCPSB. Then $\Pr\left( \max_{1\le i \le n}\left| \wh_i-1\right| \ge \epsilon \right) \le Cn^2 e^{-c \min\left\{ n\epsilon^2,\, n^{1/2}\right\}}$.
    % \[
    % \Pr\left( \max_{1\le i \le n}\left| \wh_i-1\right| \ge \epsilon \right) \le Cn^2 e^{-c \min\left\{ n\epsilon^2,\, n^{1/4}\right\}} \,.
    % \]
\end{enumerate}
\end{lemma}
\begin{proof}
        Choose $\beta=\beta_0/n$ per Lemma~\ref{lem:Tyl-inv-bound}, such that $\|\left( \nabla g_\beta(\oneVec)\right)^{-1}\|_{\infty,\infty}\le B$ holds w.h.p. By Lemma~\ref{lem:Tyl-smooth}, for some $L>0$, w.h.p.
        $\| \nabla g_\beta(\w)- \nabla g_\beta(\oneVec)\|_{\infty,\infty}\le \frac{L}{B} \|\w-\oneVec\|_\infty$ holds uniformly inside the $\ell_\infty$ ball $\|\w-\oneVec\|_\infty\le \frac12$. By Lemma~\ref{lem:Tyl-almost-zero}, $\|g_\beta(\oneVec)\|_\infty < \frac{\epsilon}{2B}$ holds w.h.p.
        Under the intersection of these events, $h_\beta(\w)=\left(\nabla g_\beta(\oneVec)\right)^{-1}  g_\beta(\w)$ satisfies the conditions of Lemma~\ref{lem:Tyl-perturb}, with constants $R=\frac12$ and $L$.  Assuming $\epsilon\le \epsilon_0:=\min\{\frac12,L^{-1}\}$, by Lemma~\ref{lem:Tyl-perturb} $h_\beta(\cdot)$ has a zero $\w^*$, equivalently a stationary point of $G_\beta(\cdot)$, with $\|\w^*-\oneVec\|_\infty\le \epsilon$. By Lemma~\ref{lem:Tyl-obs}, $\w^*=\bwh$. 

\end{proof}

\begin{proof}
    [Of Theorem~\ref{thm:main_tyler}]
    Recall that the weights $\wTyl_i$ are related to $\wh_i$ via (\ref{eq:what-to-wTyl}).
Denote $\trCov=p^{-1}\tr(\trueCov)$.
 Under the high-probability event $\max_{1\le i \le n}|\wh_i-1|\le \epsilon$, we have
 \begin{equation}\label{eq:tyl-proof-last-bound}
    \frac{1-\epsilon}{(1+\epsilon) \cdot p^{-1}\tr(S)/\trCov)} \le  \trCov\wTyl_i = \frac{ \wh_i}{\trCov^{-1} p^{-1}\tr \left( \frac1n \sum_{i=1}^n \wh_i \x_i\x_i^\T  \right)} \le \frac{1+\epsilon}{(1-\epsilon) \cdot p^{-1}\tr(S)/\trCov}\,. 
 \end{equation}
 We next show that the denominator of (\ref{eq:tyl-proof-last-bound}) concentrates tightly around $1$, namely, that w.h.p. $|p^{-1}\tr(S)-\trCov|\le \epsilon \trCov$. To this end,  let $u_1,\ldots,u_p$ be an orthonormal basis of eigenvectors of $\trueCov$, so that $\trueCov u_i = \lambda_i u_i$. Since $S=\trueCov^{1/2}T\trueCov^{1/2}$, 
    \[
    |p^{-1}\tr(S)-\trCov| = \left|p^{-1}\sum_{i=1}^p \lambda_i (u_i^\T T u_i -1)\right| \le \trCov \max_{1\le i \le p} |u_i^\T T u_i-1| \,. 
    \]
    By Lemma~\ref{lem:sc-entrywise} (and a union bound over $1\le i \le p$), under \AssumLC, $\Pr\left( \left| p^{-1} \tr(S)-\trCov\right| \ge \epsilon \, \trCov \right) \le Cpe^{-c(\Cheeger \sqrt{n})\epsilon}$, whereas under \AssumIndep~or \AssumCCPSB, $\Pr\left( \left| p^{-1} \tr(S)-\trCov\right| \ge \epsilon \, \trCov \right) \le Cpe^{-cn\epsilon^2}$. Combining with (\ref{eq:tyl-proof-last-bound}) yields that w.h.p. $\max_{1\le i\le p}|\trCov\wTyl_i-1| \le C\epsilon$, and the theorem follows.
    
\end{proof}

\subsection{Proof of Theorem \ref{thm:main_maronna_reg} (MRE)}
\label{sect:proof-maronna-reg}

By \cite[Theorem 1]{ollila2014regularized}, MRE exists uniquely w.p. $1$.
% always uniquely exists, that is, (\ref{eq:def_maronna_reg})-(\ref{eq:def_maronna_reg_weights}), admit a unique solution. 
We proceed similarly to the proof of Theorem~\ref{thm:main_maronna}.
Define 
% $\bar{h}:\R^n_+ \to \R^n_+$ by
\begin{equation}
    \bar{h}:\R^n_+ \to \R^n_+,\;:\;\quad \bar{h}_j(\bm{d}) = \frac{1+\alpha}{p} \x_j^\T \left(  \frac1n \sum_{i=1}^n u(d_i)\x_i\x_i^\top + \alpha \Id \right)^{-1}\x_j\,,\quad 1\le j \le n \,. 
\end{equation}
By the definition of MRE, Eq. (\ref{eq:def_maronna_reg}), its weights are $\wMRE_i=u(\hat{d}_i)$ where $\bar{h}(\bm{\hat{d}})=\bm{\hat{d}}$ is a fixed point. 
% Similarly to Section~\ref{sec:proof-thm:main_maronna}, 
Accordingly,
we study the fixed points of $\bar{h}$. 
Let $\bm{\hat{d}}>\bm{0}$ be a fixed point, and $j_{\min}=\argmin_{1\le j \le n}\hat{d}_j$, $j_{\max}=\argmax_{1\le j \le n}\hat{d}_j$. Since $u(\cdot)$ is non-increasing, $\bar{h}(\cdot)$ is non-decreasing, and so $\bar{h}(\hat{d}_{j_{\min}}\bm{1}) \le \bar{h}(\bm{\hat{d}}) \le \bar{h}(\hat{d}_{j_{\max}}\bm{1})$. 
% In particular, 
Considering coordinates $j\in \{j_{\min},j_{\max}\}$, and bearing in mind that 
$d_j=\bar{h}_j(\hat{\bm{d}})$,
% recalling that $\bm{\hat{d}}$ is a fixed point,
\begin{equation}\label{eq:proof-mre-1}
    \hat{d}_{j_{\min}} \ge (1+\alpha)p^{-1}\x_{j_{\min}}^\T (u(\hat{d}_{j_{\min}})S + \alpha\Id)^{-1} \x_{j_{\min}}^\T,\quad    \hat{d}_{j_{\max}} \le (1+\alpha)p^{-1}\x_{j_{\max}}^\T (u(\hat{d}_{j_{\max}})S + \alpha\Id)^{-1} \x_{j_{\max}}^\T \,.
\end{equation}
Define the following $n$ functions $\hat{F}_j : \R_+ \to \R_+$, $1\le j \le n$, by
\begin{equation}\label{eq:proof-mre-2}
    \hat{F}_j({d}) = (1+\alpha)d^{-1} p^{-1} \x_j^\T (u(d)S+\alpha\Id)^{-1}\x_j = (1+\alpha)p^{-1}\x_j^\T (\phi(d)S+\alpha d \Id)^{-1}\x_j\,.
\end{equation}
By assumption, $\phi(d)=d u(d)$ is non-decreasing hence $\hat{F}_j$ is decreasing. 
Dividing the left and right inequalities in (\ref{eq:proof-mre-1}) by $\hat{d}_{j_{\min}}$ and $\hat{d}_{j_{\max}}$ respectively, gives $\hat{F}_{j_{\min}}(\hat{d}_{j_{\min}})\le 1$, $\hat{F}_{j_{\max}}(\hat{d}_{j_{\max}})\ge 1$. Since $\hat{F}_j$ is decreasing, we deduce that $\hat{d}_{j_{\min}} \ge \hat{F}_{j_{\min}}^{-1}(1)$, $\hat{d}_{j_{\max}} \le \hat{F}_{j_{\max}}^{-1}(1)$ provided that $1$ is indeed in the range of $\hat{F}_{j_{\min}}(\cdot),\hat{F}_{j_{\max}}(\cdot)$. Since $\hat{d}_{j_{\min}},\hat{d}_{j_{\max}}$ are the smallest and largest coordinates of $\bm{\hat{d}}$ respectively, we deduce that for all $1\le j \le n$,
\begin{equation}\label{eq:proof-mre-3}
    \min_{1\le i\le n}\hat{F}_i^{-1}(1) \le \hat{d}_j \le \max_{1\le i\le n}\hat{F}_i^{-1}(1)\,,
\end{equation}
provided that $1$ is in the range of all $\hat{F}_i$-s. We shall soon see that this is indeed the case, and moreover, that the (data-dependent) quantities $\hat{F}_i^{-1}(1)$ all concentrate around a particular deterministic quantity. 

We now analyze $\hat{F}_i$, defined in (\ref{eq:proof-mre-2}). Decomposing $S=S_{-i} + n^{-1}\x_i\x_i^\T$, by the Sherman-Morrison formula, 
% similarly to (\ref{eq:relation_sherman_morrison}), write
% (\ref{eq:relation_sherman_morrison}), 
\begin{equation}
    \hat{F}_i(d) = (1+\alpha) \frac{\hat{Q}_i(d)}{1+\gamma \phi(d) \hat{Q}_i(d)},\quad\textrm{where}\quad \hat{Q}_i(d) = p^{-1} \x_i^\T ( \phi(d)S_{-i}+\alpha d \Id)^{-1} \x_i \,.
\end{equation}
Next we consider a deterministic analog of $\hat{F}_i$, where $\hat{Q}_i$ is replaced by its expectation $Q$. 
% We move to replace the data-dependent function $\hat{F}_i$ by a deterministic one. 
Define
\begin{equation}\label{eq:proof-mre-4}
    Q(d) = \E Q_i(d) = p^{-1}\E \tr\trueCov(\phi(d)S_{-i}+\alpha d \Id)^{-1} \,,\quad F(d) = (1+\alpha)\frac{Q(d)}{1+\gamma\phi(d) Q(d)} \,.
\end{equation}
\begin{lemma}\label{lem:MRE-1}
    Let $d_0>0$ be given. There are $c,C,\epsilon_0>0$, that depend on the distribution of $Y$, $\gamma$, $s_{\max}$, $\alpha$ and $d_0$, such that the following holds. For all $d\ge d_0$ and $\epsilon\le \epsilon_0$,
    \begin{enumerate}
        \item Assume~\AssumLC. Then $\Pr(\max_{1\le i \le n}|\hat{Q}_i(d)-Q(d)|\ge \epsilon) \le Cne^{-c\Cheeger \sqrt{n}\epsilon}$.
        \item Assume~\AssumIndep~or~\AssumCCPSB. Then $\Pr(\max_{1\le i \le n}|\hat{Q}_i(d)-Q(d)|\ge \epsilon) \le Cne^{-cn\epsilon^2}$. 
    \end{enumerate}
\end{lemma}
We prove Lemma~\ref{lem:MRE-1} in 
% Appendix, Section~
\ref{sec:proof-lem:MRE-1}.
By Lemma~\ref{lem:MRE-1}, the functions $\hat{F}_i$ concentrate pointwise around the deterministic function $F$. To proceed, we  show that $1\in \mathrm{Range}(F)$ and  study the local behavior of $F$ around this point. 
Before stating our next result, we remark that up to this point, the analysis in this section applies both to MRE and TRE, the latter corresponding to $u(x)=x^{-1}$. The proof of the next lemma, however, relies on the boundedness of the function $u$, which is always assumed for MRE, but does not hold for TRE.
\begin{lemma}\label{lem:MRE-2}
    There 
    % exists a unique $d^*=d^*_{n,p}$ such that 
    is a unique root
    $F(d^*)=1$. Moreover, there exist constants $0<\underline{d}<\overline{d}$, and $\eta>0$, depending on the distributions of $Y$, $\gamma$, $s_{\max}$, $\trLB$ and $\alpha$, so that: 1) $d^*\in (\underline{d},\overline{d})$; 2) For every $d_1,d_2\in (\underline{d},\overline{d})$, $|F(d_1)-F(d_2)|\ge \eta|d_1-d_2|$.
\end{lemma}
We prove Lemma~\ref{lem:MRE-2} in 
% Appendix, Section~
\ref{sec:proof-lem:MRE-2}. 
We are ready to conclude the proof of Theorem~\ref{thm:main_maronna_reg}. 
Fix a small enough $\epsilon>0$ so that $d_1=d^*-\epsilon$, $d_2=d^*+\epsilon$ satisfy $[d_1,d_2]\subseteq (\underline{d},\overline{d})$. 
Let $\eta>0$ be the constant from Lemma~\ref{lem:MRE-2}. By Lemma~\ref{lem:MRE-1}, w.h.p. $|\hat{F}_i(d_\ell)-F(d_\ell)|\le \eta\epsilon/2$ for all $1\le i \le n$ and $\ell=1,2$. Under this event, in particular, $\hat{F}_i(d_2)\le F(d_2)+\eta\epsilon/2 \le F(d^*)-\eta\epsilon+ \eta\epsilon/2 = 1-\eta\epsilon/2$. Similarly, $\hat{F}_i(d_1)\ge 1+\eta\epsilon/2$. Since the functions $\hat{F}_i$ are decreasing and continuous, it follows that $\hat{F}_i^{-1}(1) \in (d_1,d_2)=(d^*-\epsilon,d^*+\epsilon)$ for all $i$. Thus, by (\ref{eq:proof-mre-3}), $\hat{d}_j\in (d^*-\epsilon,d^*+\epsilon)$ for all $1\le j\le n$. Finally, recalling that the weights of MRE are $\wMRE_j=u(\hat{d}_j)$, we conclude that $|\wMRE-u(d^*)|\le L\epsilon$ where $L$ is the Lipschitz constant of $u$.

\qed

\subsection{Proof of Theorem~\ref{thm:main_tyler_reg} (TRE)}

Since the samples $\x_i$ are assumed to have a density, they are in general position w.p. $1$.
% Regarding the existence and uniqueness of TRE, recall that since the random vector $X$ has a density, then w.p. $1$ the points $\x_1,\ldots,\x_n$ are in general position: any $d$-dimensional subspace, $d\le p-1$, contains at most $d$ sample points. Thus, by the result of 
Consequently, since by assumption $\alpha > \max\{0,\gamma-1\}$, \cite[Theorem 3]{ollila2014regularized} implies that TRE exists uniquely w.p. $1$.

Recall that TRE has the same form as MRE, with a crucial difference that the function $u(x)=1/x$ is not bounded. 
% There is a crucial difference, however: the function $u$ is not bounded. Accordingly, the assumption $\alpha > \max\{0,\gamma-1\}$ plays an important part in the analysis. Moreover, our bounds involve an additional constraint on the population covariance, namely $p^{-1}\tr\trueCov^{-1}\le \trInvUB$. 
We follow the proof of Theorem~\ref{thm:main_maronna_reg} from Section~\ref{sect:proof-maronna-reg} above. The argument carries over, verbatim, with the exception of Lemma~\ref{lem:MRE-2}. Thus,  Theorem~\ref{thm:main_tyler_reg}
% readily follows 
follows from
% would be complete once we prove the following similar result, 
Lemma~\ref{lem:TRE-1}, stated below and proven   in \ref{sec:proof-lem:TRE-1}.
\qed

\begin{lemma}
    \label{lem:TRE-1}
    Let $F,Q$ be as in (\ref{eq:proof-mre-4}) with $u(x)=1/x$, and suppose that $\alpha>\max\{0,p/n-1\}$.
    There 
    % exists a unique $d^*=d^*_{n,p}$ such that 
    is a unique root
    $F(d^*)=1$. Moreover, there exist constants $0<\underline{d}<\overline{d}$, and $\eta>0$, depending on the distributions of $Y$, $\gamma$, $s_{\max}$, $\trLB$, $\trInvUB$ and $\alpha$, so that: 1) $d^*\in (\underline{d},\overline{d})$; 2) For every $d_1,d_2\in (\underline{d},\overline{d})$, $|F(d_1)-F(d_2)|\ge \eta|d_1-d_2|$.
    % There exists a unique $d^*=d^*_{n,p}$ such that $F(d^*)=1$. Moreover, there exist constants $0<\underline{d}<\overline{d}$, and $\eta>0$, depending on the distribution of $Y$, $\gamma$, $s_{\max}$, $\trLB$, $\trInvUB$ and $\alpha$, such that: 1) $d^*\in (\underline{d},\overline{d})$; 2) For every $d_1,d_2\in (\underline{d},\overline{d})$, $|F(d_1)-F(d_2)|\ge \eta|d_1-d_2|$.
\end{lemma}

\section{Conclusion and Further Discussion}
\label{sec:discussion}

This paper presented a non-asymptotic analysis of Tyler's and Maronna's M-estimators, as well as their regularized variants, under a substantially broader class of distributions than those considered in previous works. Specifically, we assumed a data distribution of the form $X=\trueCov^{1/2}Y$, where $Y$ is isotropic and satisfies one of several abstract concentration properties. Some of these distributions allow for  
the coordinates of $Y$ to be statistically dependent. 

\paragraph{Results for non-centered distributions}
In our analysis, we  
assumed that $X$ has zero mean.
This is often not the case in real-world applications. A more reasonable model is $X=\bm{\mu}+\trueCov^{1/2}Y$, where  $\mathbb{E}[X]=\bm{\mu}$ is in general not zero. Note that the standard method of coping with a non-zero mean, namely subtracting the sample mean, creates a statistical inter-dependency between the modified samples, so that the results of Section~\ref{sect:main-results} do not immediately apply. 

To overcome this difficulty, \cite{dumbgen1998tyler} suggested the following ``symmetrization'' procedure. Given a data set of $2n$ samples $\x_1,\ldots,\x_{2n}$, 
construct a symmetrized set of $n$ samples: 
\[
\x^{\mathrm{sym}}_i = 2^{-1/2}(\x_i-\x_{i+n})=2^{-1/2}\trueCov^{1/2}(\y_i-\y_{i+n}), \quad 1\le i \le n.
\]
 Clearly, $\E[\x^{\mathrm{sym}}_i]=\bm{0}$ and $\mathrm{Cov}(\x^{\mathrm{sym}}_i)=\trueCov$. 
 To apply our main results, the isotropic random vector ${Y}^{\mathrm{sym}}=2^{-1/2}(Y-Y')$,
 where $Y'$ is an independent copy of $Y$, has to satisfy the same properties as $Y$. Indeed: 
 \begin{enumerate}
     \item under \AssumIndep, ${Y}^{\mathrm{sym}}$ has independent entries, with sub-Gaussian constants $\lesssim K$. In addition, by Lemma~\ref{lem:sbp-indep}, the density of each entry is bounded by $\lesssim C_0$; 
     
     \item under \AssumLC, ${Y}^{\mathrm{sym}}$ is a log-concave random vector, since the log-concave family is closed under convolution of the densities (e.g. \cite{saumard2014log});
     
     \item under \AssumCCPSB:
    By separately conditioning on $Y,Y'$, one may verify that ${Y}^{\mathrm{sym}}$ satisfies both the CCP and and the SBP, possibly with a larger constant.
 \end{enumerate}
 
 In \ref{sec:appendix-zero-mean} we discuss some further details regarding the symmetrization procedure under the model  (\ref{eq:Elliptical}).

\paragraph{Projection Pursuit in high dimension}
Both \cite{bickel2018projection} and recently \cite{Montanari2022OverparametrizedLD} (who extended the results of \cite{bickel2018projection}) 
studied
some fundamental limitations 
on the ability to detect structure by projection pursuit in the high-dimensional setting, assuming multivariate Gaussian observations. 
Specifically, asymptotically as $p/n \to \gamma \in (1,\infty)$, they proved that with high probability, for any i.i.d. observations $X_1,\ldots,X_n \sim \mathcal{N}(0,\Id)$ and a given distribution $G$ with mean zero and variance bounded by $\gamma -1$, one can find a sequence of (data-dependent) projections $u_{p,G} \in \mathbb{S}^{p-1}$ such that the sequence of empirical distributions
% \[
%         F_{n,u} := \sum_{i=1}^{n}\mathbbm{1}_{X^{\top}_i \cdot u_{p,G}}(t)
% \]
$F_{n,u} := \sum_{i=1}^{n}\mathbbm{1}_{X^{\top}_i \cdot u_{p,G}}(t)$
converges to $G$ in 
% distribution, i.e.  
Kolmogorov-Smirnov distance:
$
        \lim_{n \to\infty} \|F_{n,u} - G\|_{\infty} = 0
$.
Informally speaking, this result implies that in the high dimensional regime, one can find structure in the data that does not exist in its underlying distribution, as all its marginals are $\mathcal{N}(0,1)$. This is in contrast to the results of \cite{Diaconis_Freedman}, whereby in the classical regime where $p/n \to 0$, all the empirical marginals converge to their population counterparts $\mathcal{N}(0,1)$, namely  
$
        \lim_{n,p \to\infty}\sup_{u \in \mathbb{S}^{p-1}} \|F_{n,u} - \mathcal{N}(0,1)\|_{\infty} = 0.
$

The mathematical analysis 
in the present paper can be used to generalize the results of \cite{bickel2018projection}. Specifically, most of its results 
on projection pursuit continue to hold for any $X =(X_1,\ldots,X_p)$ where each entry is a  zero mean and variance one independent sub-Gaussian with a uniform constant $K$ (these entries need not be identically distributed). 

\Revision{
\subsection*{Acknowledgements}

We are grateful to Mark Rudelson for help regarding the convex concentration property; and to the anonymous reviewers, whose comments helped improve this manuscript considerably.
}

\appendix

\section{Auxiliary Technical Lemmas}

Let $\y_1,\ldots,\y_n$ be i.i.d. realizations of an isotropic random vector $Y\in \R^p$, with sample covariance $T=\frac1n \sum_{i=1}^n \y_i\y_i^\T$. Recall that $\gamma=\frac{p}{n}$. 

\subsection{Eigenvalue bounds for sample covariance matrices}

The next two lemmas present well-known bounds on the largest and smallest eigenvalues of $T$:

\begin{lemma}\label{lem:main-smax}
    There are $c_1,c_2>0$, that depend on the distribution of $Y$ and on $\gamma$, such that:
    \begin{enumerate}
    \item Assume \AssumLC. Then $\Pr\left( \|T\| \ge c_1 \right) \le e^{-c_2 \sqrt{n}}$.
    % \[
    % \Pr\left( \|T\| \ge c_1 \right) \le e^{-c_2 \sqrt{n}} \,.  
    % \]
    \item Assume \AssumIndep~ or \AssumCCPSB \footnote{In fact, this bound does not require the small-ball property.}. Then $\Pr\left( \|T\| \ge c_1 \right) \le e^{-c_2 n}$.
% \[
% \Pr\left( \|T\| \ge c_1 \right) \le e^{-c_2 n} \,.
% \]
    \end{enumerate}
\end{lemma}
Under \AssumIndep~ and \AssumCCPSB, Lemma~\ref{lem:main-smax} follows from \cite[Theorem 5.39]{Vershynin}. Under \AssumLC, it follows from \cite[Theorem 1]{adamczak2011sharp}. 

\begin{lemma}\label{lem:main-smin}
    Suppose that $\gamma < 1$. There are $c_1,c_2>0$, that depend on the distribution of $Y$ and on $\gamma$, such that:
    \begin{enumerate}
        \item Assume \AssumLC. Then $\Pr\left( \lambda_{\min}(T) \le c_1 \right) \le e^{-c_2 \sqrt{n}}$.
        % \[
        % \Pr\left( \lambda_{\min}(T) \le c_1 \right) \le e^{-c_2 \sqrt{n}} \,.
        % \]
        \item Assume \AssumIndep~ or \AssumCCPSB. Then $\Pr\left( \lambda_{\min}(T) \le c_1 \right) \le e^{-c_2 n}$.
        % \[
        % \Pr\left( \lambda_{\min}(T) \le c_1 \right) \le e^{-c_2 n} \,.
        % \]
    \end{enumerate}
\end{lemma}
For $Y$ with i.i.d. sub-Gaussian entries, Lemma~\ref{lem:main-smin} follows from the work of Rudelson and Vershynin \cite{rudelson2009smallest}, see also 
% the book 
\cite[Theorem 5.38]{Vershynin}. Their non-asymptotic bound remarkably captures the exact ``true'' location of $\lambda_{\min}(T)$; to wit, $c_1$ may be taken up to the Mar\v{c}enko-Pastur lower edge $(1-\sqrt{\gamma})^2$. To our knowledge, for $Y\in \R^p$ with dependent (but uncorrelated) entries, similarly strong results are not currently available. Moreover, existing results which bound the two-sided deviation $\|T-\Id\|$ typically fail (barring the i.i.d. case) to produce a positive bound on $\lambda_{\min}(T)$ in the entire range $\gamma\in (0,1)$. As observed by \cite{koltchinskii2015bounding}, if $Y$ satisfies the SBP then this difficulty can be overcome, albeit with non-sharp $c_1$; this will suffice for our purposes. 
For completeness, we give a proof of Lemma~\ref{lem:main-smin} in 
% Section~
\ref{sec:proof:lem:main-smin}.

\subsection{Concentration of quadratic forms}

Note that for any fixed matrix $A\in \R^{p\times p}$, $\E [Y^\T A Y] = \tr(A)$. We cite a tail bound for the deviation $|Y^\T A Y - \tr(A)|$:  

\begin{lemma}
    \label{lem:main-quadratic}
    There is $c>0$ that depends on the distribution of $Y$, and universal $C>0$, such that for any fixed matrix $A\in \R^{p\times p}$ and $\varepsilon\in (0,1)$:
    \begin{enumerate}
        \item Assume \AssumIndep~ or \AssumCCPSB. Then $\Pr\left( \left| p^{-1}Y^\T A Y - p^{-1}\tr(A)\right| \ge \varepsilon\|A\| \right) \le C e^{-cp\varepsilon^2}$.
        % \[
        % \Pr\left( \left| \frac{1}{p}Y^\T A Y - \frac{1}{p}\tr(A)\right| \ge \varepsilon\|A\| \right) \le C e^{-cp\varepsilon^2} \,.
        % \]
        \item Assume \AssumLC. Then $\Pr\left( \left| p^{-1}Y^\T A Y - p^{-1}\tr(A)\right| \ge \varepsilon\|A\| \right) \le C e^{-c(\Cheeger\sqrt{p})\varepsilon}$.
        % \[
        % \Pr\left( \left| \frac{1}{p}Y^\T A Y - \frac{1}{p}\tr(A)\right| \ge \varepsilon\|A\| \right) \le C e^{-c(\Cheeger\sqrt{p})\varepsilon} \,,
        % \]
        % where $\Cheeger$ is satisfies (\ref{eq:cheeger-bound-main-paper}).        
    \end{enumerate}
\end{lemma}
Under \AssumIndep, Lemma~\ref{lem:main-quadratic} follows from \cite{rudelson2013hanson}. Under \AssumCCPSB, it follows from \cite{adamczak2015note}. Both of these are extensions of the Hanson-Wright bound \cite{hanson1971bound}. A proof under \AssumLC,  appears in \ref{sec:proof:lem:main-quadratic}.
We remark that for large $\varepsilon$, a tighter tail bound can be derived in the log-concave case (without $\Cheeger$), see for example \cite{lee2017eldan}. 
However, to the best of our knowledge, for small 
$\varepsilon$ no sharper bound is currently known.

% Lastly, we cite a bound on the norm $\|Y\|$:
% \begin{lemma}\label{lem:deviations-norm}
%     There $c>0$, that depends on the distributions of $Y$, such that:
%     \begin{enumerate}
%         \item Assume \AssumIndep~or \AssumCCPSB. Then
%         \[
%         \Pr\left( \|Y\| \ge 2\sqrt{p} \right) \le 2e^{-cp} \,.
%         \]
%         \item Assume \AssumLC. Then
%         \[
%         \Pr\left( \|Y\| \ge 2\sqrt{p} \right) \le 2e^{-c\sqrt{p}} \,.
%         \]
%     \end{enumerate}
% \end{lemma}
% Under \AssumIndep~or \AssumCCPSB, Lemma~\ref{lem:deviations-norm} follows from Lemma~\ref{lem:main-quadratic} (setting $A=\Id$). Under \AssumLC, it follows, for example, from \cite[Theorem 16]{lee2017eldan}.

\subsection{Entrywise concentration for the sample covariance}

\begin{lemma}\label{lem:sc-entrywise}
    There are $C,c>0$, that depend on the distribution of $Y$, such that for all unit vectors $u,v$ and $\varepsilon\in(0,1)$,
\begin{enumerate}
    \item Assume \AssumIndep~ or \AssumCCPSB. Then $\Pr\left( \left| u^\T Tv - u^\T v \right|\ge \varepsilon\right) \le Ce^{-cn\varepsilon^2}$.
    % \[
    % \Pr\left( \left| u^\T Tv - u^\T v \right|\ge \varepsilon\right) \le Ce^{-cn\varepsilon^2} \,.     
    % \]
    \item Assume \AssumLC. Then $\Pr\left( \left| u^\T Tv - u^\T v \right|\ge \varepsilon\right) \le Ce^{-c(\Psi_n\sqrt{n})\varepsilon}$.
    % \[
    % \Pr\left( \left| u^\T Tv - u^\T v \right|\ge \varepsilon\right) \le Ce^{-c(\Psi_n\sqrt{n})\varepsilon} \,.     
    % \]
\end{enumerate}
\end{lemma}
\begin{proof}
    Since $u^\T T v =
    \frac14[
    (u+v)^\T T (u+v) - (u-v)^\T  T (u-v)]$, it suffices to prove the lemma for $u=v$. Consider the random vector $W\in \R^n$ with entries $W_i=\y_i^\T u$. Observe that $W$ is centered, isotropic, and: 1) under  \AssumIndep~ or \AssumCCPSB, $W$ has i.i.d. sub-Gaussian entries; 2) under \AssumLC, $W$ is log-concave. Since $u^\T T u = \frac{1}{n}W^\T W$, the
    % claimed 
    result follows from Lemma~\ref{lem:main-quadratic} with $A=I$. 
\end{proof}

\subsection{Additional lemmas}

% \begin{lemma}[Matrix inversion formulas]\label{lem:sherman-morrison-woodbury}
%     Suppose $A$ is invertible. The following holds:
%     \begin{enumerate}[label=(\Roman*)]
%         \item {\bf Woodbury formula.} Suppose $C$ is invertible. Then
%         \[
%         \left(A+UCV\right)^{-1}  = A^{-1} - A^{-1}U\left(C^{-1} + VA^{-1}U\right)^{-1}VA^{-1} \,.
%         \]
%         \item {\bf Sherman-Morrison formula.} 
%         \[
%         \left(A+uv^\T\right)^{-1} = A^{-1} - \frac{A^{-1}uv^\T A^{-1}}{1+v^\T A^{-1}u} \,.
%         \]
%     \end{enumerate}
%         Of course, (II) is a special case of (I), where $UCV$ has rank $1$. 
% \end{lemma}

The following is well-known, see for example \cite[Lemma 4]{couillet2014robust}:
\begin{lemma}
        \label{lem:low_rank_perturb}
        Let $A,B,C \succeq 0$ be non-negative matrices and $z>0$ a positive number. Then
        \begin{equation}
        \abs{ \Tr C\left( zI+A \right)^{-1} - \Tr C\left( zI+B \right)^{-1}   } \le \mathrm{rank}(A-B)\frac{\norm{C}}{z} \,.
        \end{equation} 
\end{lemma}

The following is a standard concentration inequality for ``resolvent-like'' expressions:
\begin{lemma}
        \label{lem:resolvent-conc}
        Let $C \succeq 0$ and $A\succ 0$ be fixed matrices, and let $X_1,\ldots,X_n$ be independent, non-negative random matrices, such that for all $i$, $\mathrm{rank}(X_i)=1$ with probability $1$. Denote $S_n = \sum_{i=1}^{n}X_i$ and $R_n =  \Tr C\left( A + S_n\right)^{-1}$.
        % \[
        % R_n = R_n(X_1,\ldots,X_n) = \Tr C\left( A + S_n\right)^{-1} \,.
        % \] 
        There is universal $c>0$ such that for all $t\ge 0$, 
        \begin{equation}
                \Pr \left(  \abs{R_n-\E(R_n)}>t \right) \le 2\exp\left(-c\frac{t^2}{n\norm{CA^{-1}}^2}\right) \,.
        \end{equation}
\end{lemma}
\begin{proof}
    Consider the filtration $\mathcal{F}_k=\sigma(X_1,\ldots,X_k)$ and the martingale $M_k=\E[R_n|\mathcal{F}_k]$. We have $M_n=R_n$, $M_0=\E R_n$, and by Lemma~\ref{lem:low_rank_perturb}, $|M_{k+1}-M_k| \le 2\|CA^{-1}\|$. 
    The lemma follows from the Azuma-Hoeffding inequality for martingales with bounded increments.
\end{proof}

% \newpage
\section{Proof of Lemma~\ref{lem:main-technical}}
\label{sec:proof-lem:main-technical}

Proving Lemma~\ref{lem:main-technical} by a direct analysis of the quadratic form $\x_i^\top S^{-1}\x_i$ is difficult, since the sample covariance $S$ depends on $\x_i$.
% and thus these two random variables are not independent. 
To disentangle this dependency, as in \cite{couillet2014robust}, we apply the Sherman-Morrison formula,
\begin{equation}
\begin{aligned}S^{-1}
= \left(\frac{1}{n}\sum_{j\ne i}^{n}
\x_j\x_j^{\top}+\frac{1}{n}\x_{i}\x_{i}^{\top}\right)^{-1}
= S_{-i}^{-1}-\frac{\frac{1}{n}S_{-i}^{-1}\x_{i}\x_{i}^{\top}S_{-i}^{-1}}{1+\frac1n \x_{i}^{\top}S_{-i}^{-1}\x_{i}}\,.
\end{aligned}
        \nonumber
\end{equation}
Importantly, $\x_i$ and $S_{-i}$ are  statistically independent. 
Furthermore, 
\begin{align}
\label{eq:relation_sherman_morrison}
\frac{1}{p}\x_i^{\top}S^{-1}\x_i = 
\frac{1}{p}\x_i^\top\left[S_{-i}^{-1}-\frac{\frac{1}{n}S_{-i}^{-1}\x_{i}\x_{i}^{\top}S_{-i}^{-1}}{1+\frac1n \x_{i}^{\top}S_{-i}^{-1}\x_{i}} \right]\x_i
= \frac{\frac{1}{p}\x^{\top}_iS_{-i}^{-1}\x_i}{1+\gamma\cdot \frac1p \x_{i}^{\top}S_{-i}^{-1}\x_{i}}\,,
\end{align}
and so,
\begin{align}\label{eq:lem:main-technical-to-main-technical-minus}
        \abs{ \frac{1}{p}\x_i^{\top}S^{-1}\x_i - 1}
        = \frac{ \abs{(1-\gamma)\cdot \frac{1}{p}\x^{\top}_iS_{-i}^{-1}\x_i - 1}}{1+\gamma\cdot \frac1p \x_{i}^{\top}S_{-i}^{-1}\x_{i}}  
        \le (1-\gamma)\abs{\frac{1}{p}\x^{\top}_iS_{-i}^{-1}\x_i - \frac{1}{1-\gamma}} \,.
\end{align}

The following lemma, proven below, shows that the r.h.s. of (\ref{eq:lem:main-technical-to-main-technical-minus}) is small w.h.p.
\begin{lemma}\label{lem:main-technical-minus-i}   
            Assume  $\gamma<1$. There are 
        $c,C,\epsilon_0>0$, that depend on the distribution of $Y$ and on $\gamma$, so that for all $\epsilon<\epsilon_0$:
        \begin{enumerate}
                
                \item 
                Assume \AssumLC. Then 
                $\Pr(\max_{1\le i \le n}|p^{-1}\x_i^\T S_{-i}^{-1}\x_i - \frac{1}{1-\gamma}|\ge \epsilon) \le Cn^2e^{-c(\Cheeger\sqrt{n}) \epsilon}$.

                \item     
                Assume \AssumIndep~ or \AssumCCPSB. Then 
                $\Pr(\max_{1\le i \le n}|p^{-1}\x_i^\T S_{-i}^{-1}\x_i - \frac{1}{1-\gamma}|\ge \epsilon) \le C n^2 e^{-cn\epsilon^2}$.
        \end{enumerate}
\end{lemma}
% We prove Lemma~\ref{lem:main-technical-minus-i} in 
% \ref{appendix:proof-unregularized-main-lemma} below.

\begin{proof}[ Proof of Lemma \ref{lem:main-technical}]
Immediate from Lemma~\ref{lem:main-technical-minus-i} combined with Eq. (\ref{eq:lem:main-technical-to-main-technical-minus}).
 \end{proof}

\subsection{Proof of Lemma~\ref{lem:main-technical-minus-i}}
\label{appendix:proof-unregularized-main-lemma}

As previously mentioned, $\x_i^\top S^{-1}_{-i}\x_i = \y_i^\T T^{-1}_{-i}\y_i$
% Observe that $\x_i^\top S^{-1}_{-i}\x_i$, 
does not depend on the population covariance $\trueCov$. 
% Indeed, since $\x_i=\trueCov^{1/2}\y_i$ and $S_{-i}=\trueCov^{1/2}T_{-i}\trueCov^{1/2}$,
% then $\x_i^\top S_{-i}^{-1}\x_i = \y_i S_p^{1/2}\left( S_p^{1/2} T_{-i} S_p^{1/2} \right)^{-1} S_p^{1/2}\y_i =  \y_i^\top T_{-i}^{-1}
% \y_i$.
Hence we bound the deviations 
of $\y_i^\top T_{-i}^{-1}
\y_i$ from $1/(1-\gamma)$.
Lemma~\ref{lem:main-technical-minus-i} then follows by a union bound over $1\le i \le n$. 
The analysis proceeds as follows. First,  
we show that $p^{-1} \y_i^\T T_{-i}^{-1}\y_i$ concentrates around $\E[p^{-1}\y_i^\T T_{-i}^{-1}\y_i|T_{-i}] = p^{-1}\tr(T_{-i}^{-1})$.
Next, we prove that $p^{-1}\tr(T_{-i}^{-1})$
concentrates around $1/(1-\gamma)$. 
Proving this directly is difficult, since
the smallest eigenvalue of $T_{-i}$ can take very small value, though with overwhelming small probability. To circumvent this, we consider a regularized variant $\hat{m}(\epsilon)=p^{-1}\tr(T_{-i}+\epsilon\Id)^{-1}$.
        % ,  to be studied instead. 
        % Since, w.h.p., $\lambda_{\min}(T_{-i})$ is lower bounded by a constant,
        We then show that $|\hat{m}(\epsilon)-p^{-1}\tr(T_{-i}^{-1})|\lesssim \epsilon$,
        % The last step is to show 
        and finally
        that $\hat{m}(\epsilon)$ concentrates around $1/(1-\gamma)$.

Note that  $\hat m(\epsilon)$ is the Stieltjes transform of the empirical spectral distribution (ESD) of $T_{-i}$, evaluated at $-\epsilon$. Since $T_{-i}$ is a sample covariance matrix of i.i.d. isotropic samples, its ESD converges to a Mar\v{c}enko-Pastur law with shape parameter $\gamma$. This reveals the reason for the value $(1-\gamma)^{-1}$: it is the value of the Stieltjes transform of the Mar\v{c}enko-Pastur law, evaluated at $0$, cf. \cite{bai2010spectral,couillet2014robust}. 

In light of the above roadmap, write $p^{-1}\y_iT_{-i}^{-1}\y_i - (1-\gamma)^{-1} = \Delta_1 + \Delta_2 + \Delta_3$,
% \[
% p^{-1}\y_iT_{-i}^{-1}\y_i - (1-\gamma)^{-1} = \Delta_1 + \Delta_2 + \Delta_3\,,    
% \]
where
\begin{equation}
    \begin{split}
        \Delta_1 = p^{-1}\y_i^\T T_{-i}\y_i &- p^{-1}\tr(T_{-i}^{-1}),\quad \Delta_2 = p^{-1}\tr(T_{-i}^{-1}) - \mh(\epsilon),\quad \Delta_3=\mh(\epsilon)-(1-\gamma)^{-1},\\
        &\mh(\epsilon) = p^{-1}\tr(T_{-i} + \epsilon\Id)^{-1} \,.
    \end{split}
\end{equation}
It suffices to show that w.h.p., $|\Delta_\ell|\lesssim \epsilon$ for $\ell=1,2,3$. We start with a high-probability bound on $\Delta_1,\Delta_2$:
\begin{lemma}\label{lem:main-technical-helper1}
    Assume the conditions of Lemma~\ref{lem:main-technical-minus-i}.
    There are $c_1,c_2,C_1,\epsilon_0>0$, that may depend on the distribution of $Y$ and on $\gamma$, such that for all $\epsilon\in (0,\epsilon_0)$,
    \begin{enumerate}
        \item Assume \AssumLC. Then $\Pr(|\Delta_1|\ge \epsilon) \le c_1e^{-c_2(\Cheeger\sqrt{p})\epsilon}$ and $\Pr(|\Delta_2|\ge C_1\epsilon) \le c_1e^{-c_2(\Cheeger\sqrt{p})\epsilon}$.
        % \[
        % \Pr(|\Delta_1|\ge \epsilon) \le c_1e^{-c_2(\Cheeger\sqrt{p})\epsilon},\quad \Pr(|\Delta_2|\ge C_1\epsilon) \le c_1e^{-c_2(\Cheeger\sqrt{p})\epsilon} \,.
        % \]
        \item Assume \AssumIndep~or \AssumCCPSB. Then $\Pr(|\Delta_1|\ge \epsilon) \le c_1e^{-c_2p\epsilon^2}$ and $\Pr(|\Delta_2|\ge C_1\epsilon) \le c_1e^{-c_2p\epsilon^2}$.
        % \[
        % \Pr(|\Delta_1|\ge \epsilon) \le c_1e^{-c_2p\epsilon^2},\quad \Pr(|\Delta_2|\ge C_1\epsilon) \le c_1e^{-c_2p\epsilon^2} \,.
        % \]
    \end{enumerate}
\end{lemma}
\begin{proof}
    Let $C>0$ be such that the event $\Ec=\{\lambda_{\min}(T_{-i})\ge C\}$ holds w.h.p., per Lemma~\ref{lem:main-smin}. Of course, $\Pr(|\Delta_\ell|\ge \epsilon) \le \Pr(|\Delta_\ell|\ge \epsilon\,|\,\Ec) + \Pr(\Ec^c)$. Starting with $\ell=1$,
    since $\y_i$ and $T_{-i}^{-1}$ are independent, 
    $\Pr(|\Delta_1|\ge \epsilon\,|\,\Ec)$ may be bounded using Lemma~\ref{lem:main-quadratic}, a concentration inequality for the
    quadratic form $\y_i \mapsto p^{-1}\y_iT_{-i}^{-1}\y_i$, applied conditionally on $T_{-i}^{-1}$. Importantly, note that 
    under $\Ec$, we have $\|T_{-i}^{-1}\|=1/\lambda_{\min}(T_{-i})\le 1/C$. As for $\ell=2$, under $\Ec$,
    \[
    |\Delta_2| = p^{-1}\tr\left[ T_{-i}^{-1}-(T_{-i}+\epsilon\Id)^{-1}  \right] = \epsilon \cdot p^{-1}   \tr\left[ (T_{-i}+\epsilon\Id)^{-1}T_{-i}^{-1}  \right] \le \frac{\epsilon}{(\lambda_{\min}(T_{-i}))^2} \le C^{-2}\epsilon \,,
    \]
    and so the claimed result follows.
\end{proof}

Finally, Lemma~\ref{lem:main-technical-helper2}, proven in \ref{sec:proof-lem:main-technical-helper2},  provides a high-probability bound for $|\Delta_3|$:

\begin{lemma}\label{lem:main-technical-helper2}
    Assume the conditions of Lemma~\ref{lem:main-technical-minus-i}.
    There are $c_1,c_2,C_1,\epsilon_0>0$, that may depend on the distribution of $Y$ and on $\gamma$, such that for all $\epsilon\in (0,\epsilon_0)$,
    \begin{enumerate}
        \item Assume \AssumLC. Then $\Pr(|\Delta_3|\ge C_1\epsilon) \le c_1ne^{-c_2(\Cheeger\sqrt{p})\epsilon}$.
        % \[
        % \Pr(|\Delta_3|\ge C_1\epsilon) \le c_1ne^{-c_2(\Cheeger\sqrt{p})\epsilon} \,.
        % \]
        \item Assume \AssumIndep~or \AssumCCPSB. Then $\Pr(|\Delta_3|\ge C_1\epsilon) \le c_1ne^{-c_2p\epsilon^2}$.
        % \[
        % \Pr(|\Delta_3|\ge C_1\epsilon) \le c_1ne^{-c_2p\epsilon^2} \,.
        % \]
    \end{enumerate}
\end{lemma}
% We prove Lemma~\ref{lem:main-technical-helper2} in Section~\ref{sec:proof-lem:main-technical-helper2} below.
\begin{proof}
    (Of Lemma~\ref{lem:main-technical-minus-i}). Recall that $|p^{-1}\y_iT_{-i}^{-1}\y_i - (1-\gamma)^{-1}| \le |\Delta_1| + |\Delta_2| + |\Delta_3|$, and so the result follows from Lemma~\ref{lem:main-technical-helper1}, Lemma~\ref{lem:main-technical-helper2}, and a union bound over $1\le i \le n$.
\end{proof}

\subsection{Proof of Lemma~\ref{lem:main-technical-helper2}}
\label{sec:proof-lem:main-technical-helper2}

To simplify notation, assume w.l.o.g. that $i=n$. Set $\nb=n-1$, and let $\K\in \R^{\nb\times p}$ be the matrix whose rows are $\w_1^\T,\ldots,\w_{\nb}^\T$, with $\w_j=n^{-1/2}\y_j$. Note that $T_{-n}= \K^\T \K$ and so $\mh(\epsilon)=p^{-1}\Tr(\K^\T\K + \epsilon\Id)^{-1}$. For $1\le j \le \nb$, let $\K_j\in \R^{(\nb-1)\times p}$ be obtained by removing the $j$-th row from $\K$.  

The proof relies on several algebraic ``tricks'' which are classical in random matrix theory, see \cite{bai2010spectral}.
Recall that for any matrix $A\in \R^{d_1\times d_2}$, the spectra of $AA^\top$ and $A^\top A$ are identical up to  $\abs{d_1-d_2}$ zeros. 
% Consequently, for any function $f(\cdot)$ which is analytic on a domain containing the eigenvalues of $A^\T A$ and $0$, $\tr(f(A^\T A))=\tr(f(A A^\T)) + (d_2-d_1)f(0)$. Applying this with $A=\K$ and $f(\cdot)=(\cdot + \epsilon)^{-1}$, 
Thus,
\begin{equation}\label{eq:main-technical-helper2-1}
    \mh(\epsilon)=p^{-1}\Tr(\K^\T\K  +\epsilon\Id )^{-1} = p^{-1} \Tr(\K\K^\T  +\epsilon\Idn{\nb} )^{-1} + \epsilon^{-1} (1-p^{-1}\nb) \,.
\end{equation}
Note that $\K\K^\T$ is the Gram matrix of the vectors $\w_i$. We write the $j$-th diagonal entry of $(\K\K^\T  +\epsilon\Idn{\nb} )^{-1}$ as:
% and in particular its $j$-th column is $\K\w_j$.  Using 
% By the Schur complement formula, 
\begin{eqnarray}
    (\K\K^\T  +\epsilon\Idn{\nb} )^{-1}_{jj} 
        &\overset{(i)}{=}& \left( \epsilon + \|\w_j\|^2 - \w_j^\T \K_j^\T \left(\K_j\K_j^\T + \epsilon\Idn{(\nb-1)} \right)^{-1}\K_j\w_j \right)^{-1} \nonumber\\
        &\overset{(ii)}{=}& \left( \epsilon + \w_j^\T \left[ \Id - \left(\K_j^\T \K_j + \epsilon\Id \right)^{-1}\K_j^\T \K_j \right]\w_j \right)^{-1}  \nonumber \\
        &\overset{(iii)}{=}& \left( \epsilon + \epsilon \w_j^\T \left(\K_j^\T \K_j + \epsilon\Id \right)^{-1}\w_j \right)^{-1} \,,
    \label{eq:main-technical-helper2-2}
\end{eqnarray}
where: (i) follows from the block matrix inversion formula; (ii) follows since for any $f(\cdot)$ and matrix $A$, $A f(A^\T A) A^\T = f(AA^\T)AA^\T$, where for a symmetric matrix $P$, $f(P)$ is the matrix obtained by applying $f(\cdot)$ on the eigenvalues of $P$ (the spectral calculus for symmetric matrices); this may be verified readily by considering the SVD of $A$; and (iii) Follows by straightforward algebraic manipulation.  

% \begin{equation}\label{eq:main-technical-helper2-2}
%     \begin{split}
%         (\K\K^\T  +\epsilon\Idn{\nb} )^{-1}_{jj} 
%         &= \left( \epsilon + \|\w_j\|^2 - \w_j^\T \K_j^\T \left(\K_j\K_j^\T + \epsilon\Idn{\nb-1} \right)^{-1}\K_j\w_j \right)^{-1} 
%         = \left( \epsilon + \w_j^\T \left[ \Id - \left(\K_j^\T \K_j + \epsilon\Id \right)^{-1}\K_j^\T \K_j \right]\w_j \right)^{-1} \\
%         &= \left( \epsilon + \epsilon \w_j^\T \left(\K_j^\T \K_j + \epsilon\Id \right)^{-1}\w_j \right)^{-1} \,.
%     \end{split}
% \end{equation}
Consider
the quadratic form $\w_j^\T (\K_j^\T \K_j + \epsilon\Id)^{-1}\w_j = \gamma p^{-1}\y_j^\T (\K_j^\T \K_j + \epsilon\Id)^{-1}\y_j$. 
We claim that $p^{-1}\y_j^\T (\K_j^\T \K_j + \epsilon\Id)^{-1}\y_j$ is very close (w.h.p.) to $\mh(\epsilon)$, the quantity of interest. To wit, define the residual
% Since $\K_j$ and $\y_j$ are independent, we have  $p^{-1}\y_j^\T (\K_j^\T \K_j + \epsilon\Id)^{-1}\y_j \approx p^{-1}\Tr(\K_j^\T \K_j + \epsilon\Id)^{-1}$, similarly to the proof of Lemma~\ref{lem:main-technical-helper1}.  Note moreover, $\K\K^\T-\K_j^\T \K_j=\w_j\w_j^\T$ is rank 1; thus, we expect the difference between the traces, $p^{-1}\Tr(\K_j^\T \K_j + \epsilon\Id)^{-1}-p^{-1}\Tr(\K^\T \K + \epsilon\Id)^{-1}$ to be small. Recalling that $\mh(\epsilon)=p^{-1}\Tr(\K^\T \K + \epsilon\Id)^{-1}$ is exactly the quantity we want to study, we are led to deduce that $p^{-1}\y_j^\T (\K_j^\T \K_j + \epsilon\Id)^{-1}\y_j$ should be close to $\mh(\epsilon)$. Let us quantify this statement. To that end, define the residuals
\begin{equation}
    \eta_j = p^{-1}\y_j^\T (\K_j^\T \K_j + \epsilon\Id)^{-1}\y_j - \mh(\epsilon) \,,
\end{equation}
so that (\ref{eq:main-technical-helper2-2}) reads
\begin{equation}\label{eq:main-technical-helper2-3}
    (\K\K^\T  +\epsilon\Idn{\nb} )^{-1}_{jj} = \epsilon^{-1}\left(1 + \gamma \mh(\epsilon) + \gamma\eta_j \right)^{-1} \,.
\end{equation}
\begin{lemma}\label{lem:main-technical-helper3}
    Assume the conditions of Lemma~\ref{lem:main-technical-minus-i}.
    There are $c_1,c_2,C_1,\epsilon_0>0$, that may depend on the distribution of $Y$ and on $\gamma$, such that for all $\epsilon\in (0,\epsilon_0)$,
    \begin{enumerate}
        \item Assume \AssumLC. Then $\Pr(\max_{1\le j\le \nb}|\eta_j|\ge C_1\epsilon + (p\epsilon)^{-1}) \le c_1ne^{-c_2(\Cheeger\sqrt{p})\epsilon}$.
        % \[
        % \Pr(\max_{1\le j\le \nb}|\eta_j|\ge C_1\epsilon + (p\epsilon)^{-1}) \le c_1ne^{-c_2(\Cheeger\sqrt{p})\epsilon} \,.
        % \]
        \item Assume \AssumIndep~or \AssumCCPSB. Then $\Pr(\max_{1\le j\le \nb}|\eta_j|\ge C_1\epsilon + (p\epsilon)^{-1}) \le c_1ne^{-c_2p\epsilon^2}$.
        % \[
        % \Pr(\max_{1\le j\le \nb}|\eta_j|\ge C_1\epsilon + (p\epsilon)^{-1}) \le c_1ne^{-c_2p\epsilon^2} \,.
        % \]
    \end{enumerate}
\end{lemma}
\begin{proof}
    Decompose $\eta_j=\eta_{j,1}+\eta_{j,2}$ where
    \begin{align*}
        &\eta_{j,1} =  p^{-1}\y_j^\T (\K_j^\T \K_j + \epsilon\Id)^{-1}\y_j - p^{-1}\Tr(\K_j^\T \K_j + \epsilon\Id)^{-1},
        \quad
        \eta_{j,2} = p^{-1}\Tr(\K_j^\T \K_j + \epsilon\Id)^{-1} - p^{-1}\Tr(\K^\T \K + \epsilon\Id)^{-1} \,.
    \end{align*}
    Since $\K_j$ and $\y_j$ are independent,
    we can bound $|\eta_{j,1}|$ similarly to Lemma~\ref{lem:main-technical-helper1}; we omit the details. As for $\eta_{j,2}$, note that $\K^\T \K-\K_j^\T \K_j=\w_j\w_j^\T$ is rank $1$. Thus, by Lemma~\ref{lem:low_rank_perturb}, $|\eta_{j,2}|\le (p\epsilon)^{-1}$ w.p. $1$.
\end{proof}
Now, considering (\ref{eq:main-technical-helper2-3}), $\left|\left(1 + \gamma \mh(\epsilon) + \gamma\eta_j \right)^{-1} - \left(1 + \gamma \mh(\epsilon) \right)^{-1}\right| \le 2\gamma|\eta_j|$ holds whenever $\gamma|\eta_j|\le \frac12$.
% \[
%     \left|\left(1 + \gamma \mh(\epsilon) + \gamma\eta_j \right)^{-1} - \left(1 + \gamma \mh(\epsilon) \right)^{-1}\right| \le 2\gamma|\eta_j|\,,
% \]
Accordingly, $\left|\frac1{\nb} \sum_{j=1}^{\nb} \left(1 + \gamma \mh(\epsilon) + \gamma\eta_j \right)^{-1}-\left(1 + \gamma \mh(\epsilon) \right)^{-1}\right| \le \max_{1\le j \le \nb}2\gamma|\eta_j|$ holds whenever ${\max_{1\le j\le \nb}\gamma|\eta_j|\le \frac12}$.
Using Eqs.
(\ref{eq:main-technical-helper2-1}), (\ref{eq:main-technical-helper2-3}),  also $p^{-1}\nb=p^{-1}n-p^{-1}=\gamma^{-1}-p^{-1}$, write
\begin{align*}
    \mh(\epsilon) 
    &= p^{-1}\sum_{j=1}^{\nb} \epsilon^{-1}\left(1 + \gamma \mh(\epsilon) + \gamma\eta_j \right)^{-1} + \epsilon^{-1} (1-p^{-1}\nb) 
    = \epsilon^{-1}(p^{-1}\nb)\frac{1}{\nb}\sum_{j=1}^{\nb}\left(1 + \gamma \mh(\epsilon) + \gamma\eta_j \right)^{-1} + \epsilon^{-1} (1-p^{-1}\nb) \\
    &= \epsilon^{-1}\left[ \gamma^{-1}\left(1 + \gamma \mh(\epsilon) \right)^{-1} + 1-\gamma^{-1} + \xi \right]\,,
\end{align*}
where $|\xi|\le 2p^{-1} + \max_{1\le j \le \nb}{2\gamma|\eta_j|}$ whenever ${\max_{1\le j\le \nb}\gamma|\eta_j|\le 1/2}$. Rearranging terms, we deduce that $\mh(\epsilon)$ satisfies the quadratic equation
\begin{equation}\label{eq:main-technical-helper2-4}
    \gamma\epsilon\mh(\epsilon) + (\epsilon+1-\gamma-\gamma\xi)\mh(\epsilon) - (1+\xi) = 0 \,.
\end{equation}
Now, assume w.l.o.g. that $\epsilon\ge p^{-1/2}$; we can do this since the bound of Lemma~\ref{lem:main-technical-helper2} is vacuous for $\epsilon <p^{-1/2}$, provided that $c_1,c_2$ are chosen appropriately. Under the high-probability event of Lemma~\ref{lem:main-technical-helper3}, $|\xi|\le  C_1\epsilon$. Thus, under this event, $\mh(\epsilon)$ satisfies a quadratic equation, whose coefficients are $O(\epsilon)$-close to the coefficients of the linear equation $(1-\gamma)m-1=0$; note that the unique root of the linear equation is $m=(1+\gamma)^{-1}$. 
Let $m_1\le m_2$ be the two roots of (\ref{eq:main-technical-helper2-4}). 
One may readily verify the following: there are $C_2,\epsilon_0>0$, that depend on $C_1$, such that for all $\epsilon\le \epsilon_0$, and $\xi$ satisfying $|\xi|\le C_1\epsilon$, we have $|m_1-(1-\gamma)^{-1}|\le C_2\epsilon$ whereas $m_2> 1/(C_2\epsilon)$. 
It remains to argue that, w.h.p., necessarily $\mh(\epsilon)=m_1$.
Indeed, recall that $\mh(\epsilon)\le 1/\lambda_{\min}(T_{-i})$. By Lemma~\ref{lem:main-smin}, we can find some $C_3>0$ such that $\mh(\epsilon)\le C_3$ holds w.h.p. Consequently, for all $\epsilon<\min\{\epsilon_0,1/(C_2 C_3)\}$, the high-probability event $\{|\xi|\le C_1\epsilon\}\cap \{\mh(\epsilon)\le C_3\}$ implies that necessarily $\mh(\epsilon)=m_1$, and so $|\mh(\epsilon)-(1-\gamma)^{-1}|\le C_2\epsilon$. Thus, the proof of Lemma~\ref{lem:main-technical-helper2} is concluded.
\qed

% \newpage
\section{Proof of Lemmas for Theorem~\ref{thm:main_tyler} (TE)}

\subsection{Proof of Lemma~\ref{lem:Tyl-inv-bound}}
\label{sec:proof:lem:Tyl-inv-bound}

Starting from the definition of $G_\beta$, Eq. (\ref{eq:Tyl-def-G}), one may readily calculate,
\begin{equation}\label{eq:Tyl-grad-G}
\begin{split}
    (g_\beta(\w))_\ell=\left(\nabla G_\beta(\w)\right)_\ell 
    &= \nabla F(\w)_\ell + \beta\left(\sum_{i=1}^n w_i-n\right) 
    = -\frac{1}{w_\ell} + \frac{n}{p}\x_\ell^\T \left(\sum_{i=1}^n w_i \x_i\x_i^\T \right)^{-1}\x_\ell + \beta\left(\sum_{i=1}^n w_i-n\right) \,.
\end{split}
\end{equation}
% Starting from (\ref{eq:Tyl-grad-G}), a rudimentary calculation gives the following expression for $\left(\nabla g_\beta(\w)\right)_{j,\ell} 
%     = \left(\nabla^2 G_\beta (\w)\right)_{j,\ell}$:
Taking the second derivative,
\begin{equation}\label{eq:Tyl-hessian}
\begin{split}
    \left(\nabla g_\beta(\w)\right)_{j,\ell} 
    % = \left(\nabla^2 G_\beta (\w)\right)_{j,\ell} 
    % &= \frac{1}{w_\ell^2}\indic{j=\ell} - \frac{n}{p}\x_j^\T \left( \sum_{i=1}^n w_i\x_i\x_i^\T \right)^{-1} \x_\ell \x_\ell^\T \left( \sum_{i=1}^n w_i\x_i\x_i^\T \right)^{-1} \x_j + \beta \\
    &= \frac{1}{w_\ell^2}\indic{j=\ell} - \gamma \left(\frac1p \x_j^\T \left( \frac1n \sum_{i=1}^n w_i\x_i\x_i^\T \right)^{-1} \x_\ell\right)^2 + \beta \,.
\end{split}
\end{equation}\label{eq:Tyl-hessian-at1}
% Define the matrix $A_{j,\ell} = \left(p^{-1}\x_j^\T S^{-1}\x_\ell \right)^2=\left(p^{-1}\y_j^\T T^{-1}\y_\ell \right)^2$, so that by (\ref{eq:Tyl-hessian}) 
In particular, setting $\w=\oneVec$,
\begin{equation}\label{eq:Tyl-Zhang-1}
    \nabla g_\beta(\oneVec) = \Id - ( A - \beta\oneVec\oneVec^\T),\quad \textrm{where }\; A_{j,\ell} = \gamma \left(p^{-1}\y_j^\T T^{-1}\y_\ell \right)^2\,.
\end{equation}

Following 
% the proof of 
\cite[Lemma 3.3]{zhang2016marvcenko}, 
% recall that 
$\left(\nabla g_\beta(\oneVec)\right)^{-1}=\left( \Id - (A - \beta\oneVec\oneVec^\T)\right)^{-1} = \sum_{k=0}^\infty \left( A - \beta\oneVec\oneVec^\T\right)^k$ provided that the sum converges. Since $\|\left( A - \beta\oneVec\oneVec^\T\right)^k\|_{\infty,\infty} \le \| A - \beta\oneVec\oneVec^\T\|_{\infty,\infty}^k$, the sum clearly converges when $\| A - \beta\oneVec\oneVec^\T\|_{\infty,\infty}<1$, and then
\[
\left\| \left(\nabla g_\beta(\oneVec)\right)^{-1} \right\|_{\infty,\infty} \le \sum_{k=0}^\infty \| A - \beta\oneVec\oneVec^\T\|_{\infty,\infty}^k \le \frac{1}{1 - \| A - \beta\oneVec\oneVec^\T\|_{\infty,\infty}} \,.
\]
% whenever $\|\gamma A - \beta\oneVec\oneVec^\T\|_{\infty,\infty} < 1$. 
Thus, to prove Lemma~\ref{lem:Tyl-inv-bound}, it suffices to show that w.h.p., $\|A - \beta\oneVec\oneVec^\T\|_{\infty,\infty} \le 1-c$ for some constant $c>0$.

By definition,
\begin{equation}\label{eq:A-minus-beta}
    \| A - \beta\oneVec\oneVec^\T\|_{\infty,\infty} = \max_{1\le j \le n} \sum_{\ell=1}^n | A_{j,\ell}-\beta| \,.
\end{equation}
It is instructive to consider (\ref{eq:A-minus-beta}) with $\beta=0$. Observe that $\sum_{\ell=1}^n | A_{j,\ell}| = \frac{1}{np}\sum_{\ell=1}^n \y_j^\T T^{-1}\y_\ell\y_\ell^\T T^{-1}\y_j = \frac{1}{p}\y_j^\T T^{-1}\y_j$, and recall that by Lemma~\ref{lem:main-technical}, this quantity concentrates tightly around $1$.
Accordingly, when $\beta=0$, the norm (\ref{eq:A-minus-beta}) concentrates around $1$. Our goal, then, is to find some $\beta$ so to consistently bias (\ref{eq:A-minus-beta}) away from $1$. To this end, we proceed along 
the argument of Zhang {\it et. al.} \cite[Lemma 3.3]{zhang2016marvcenko}. 
% First, observe that 
% \[
%     \sum_{\ell=1}^n \gamma A_{j,\ell} = \frac{1}{np}\sum_{\ell=1}^n \y_j^\T T^{-1}\y_\ell\y_\ell^\T T^{-1}\y_j = \frac{1}{p}\y_j^\T T^{-1}\y_j \,,
% \]
% and recall that by Lemma~\ref{lem:main-technical}, the right-hand-side concentrates tightly around $1$. Thus, setting $\beta=0$ will not allow us to deduce the required bound, as $\left.\|\gamma A-\beta\oneVec\oneVec^\T \|_{\infty,\infty}\right|_{\beta=0}$ concentrates around $1$, whereas we desire a bound which is strictly smaller (by a constant). Instead, we need to carefully choose $\beta$, which has been arbitrary up to this point, so to consistently create a negative bias in (\ref{eq:A-minus-beta}). 
Parameterize $\beta=\beta_0/n$, for constant $\beta_0$, so that $\beta$ has the same scale as the off-diagonal entries $A_{j,\ell}$ in (\ref{eq:Tyl-Zhang-1}). 
Let $N_j[\beta_0]$ be the number of entries $A_{j,\ell}$, in the $j$-th row ($1\le \ell \le n$), such that $A_{j,\ell}\ge \beta_0/n$. 
For $ A_{j,\ell}\ge \beta_0/n$  we have
$|  A_{j,\ell}-\beta_0/n| =  A_{j,\ell}-\beta_0/n $ whereas if 
$ A_{j,\ell}< \beta_0/n$, we may bound
$|  A_{j,\ell}-\beta_0/n|\le  A_{j,\ell}+\beta_0/n$.
Hence, 
\begin{equation}\label{eq:Tyl-Zhang-2}
\begin{split}
    \| A - \beta\oneVec\oneVec^\T\|_{\infty,\infty} 
    &\le  \max_{1\le j \le n} \left\{ \sum_{\ell=1}^n  A_{j,\ell} - \frac{\beta_0}{n} N_j[\beta_0] + (n-N_j[\beta_0])\frac{\beta_0}{n} \right\} 
    \le \max_{1\le j \le n}p^{-1}\y_j^\T T^{-1}\y_j -  2{\beta_0} \min_{1\le j \le n}\left( \frac{N_j[\beta_0]}{n}-\frac{1}{2} \right) \,.
\end{split}
\end{equation}
Thus, to conclude the proof of the lemma, we need 
to find some constant $\beta_0$ so that w.h.p., ${\min_{1\le j \le n} {N_j[\beta_0]}/{n} \ge \frac12+c}$ (say, $c=0.1$); that is, such that at least $(\frac12+c)n$ of the entries of $A$ in every row are consistently larger than $\beta_0$. 
Observe that for any row $1\le j \le n$, the off-diagonal entries $A_{j,\ell}$, $\ell\ne j$, are identically distributed. A source of difficulty is that they are not independent, and this dependence is manifested in two ways: 1) They all depend on $\y_j$; 2) The sample covariance $T$ depends on all $\y_\ell$-s. The first dependence is easy to overcome (by conditioning on $\y_j$), but the second one is more involved. To deal with the latter, we shall lower bound $A_{j,\ell}$ by a different set of random variables, which are easier to analyze.

\paragraph{The mathematical error in \cite{zhang2016marvcenko}}
In their attempt to
lower bound $A_{j,\ell}$, in the proof
of their Lemma 3.4, \cite{zhang2016marvcenko}  used the following inequality (top of page 122 in their paper), 
\begin{equation}\label{eq:Zhang-Error}
    A_{j,\ell} = \frac{1}{np}(\y_j^\T T^{-1}\y_\ell)^2 \ge \frac1{np} \cdot  \frac{(\y_j^\T \y_\ell)^2}{ \lambda_{\max}(T)^2 }\,.
\end{equation}
They next analyzed the
simpler expressions for the numerator and denominator above.  
Unfortunately,  (\ref{eq:Zhang-Error}) is \emph{false}, as $|u^\T B^{-1} v| \ge |u^\T v|/\lambda_{\max}(B)$ is not generally true for positive matrices $B$.

\paragraph{A corrected argument.}
% We propose a corrected argument to circumvent the dependence between the entries $A_{j,\ell}$, $\ell\ne j$.
Let $k=o(n)$ be an integer, to be chosen later. Partition $[n]=\{1,\ldots,n\}$ into $M=\lceil n/k\rceil$ subsets $\Ic_1,\ldots,\Ic_M$ of size $k/2\le |\Ic_i|\le k$ each. 
The idea is to approximate $\{\y_j^\T T^{-1}\y_\ell\}_{\ell \ne j}$ by a different set of random variables, 
so that variables within the same class $\ell\in \Ic_i$ are independent of one another (conditioned on $\y_j$). For an index $\ell$, denote by $\Ic(\ell)$ the unique class such that $\ell\in \Ic(\ell)$.

For a set $\Ic\subseteq [n]$, denote by $\Y_{\Ic}\in \R^{|\Ic|\times p}$ the matrix whose rows are $\{\y_\ell\,:\,\ell \in \Ic\}$.
Given indices $j \ne \ell$, we decompose $T=T_{-\ell} + n^{-1}\y_\ell\y_\ell^\T = T_{-\ell,-j} + n^{-1}\y_\ell\y_\ell^\T + n^{-1}\y_j\y_j^\T$.
% $
% T = T_{-j,-\Ic(\ell)} + n^{-1}\y_j\y_j^\T + n^{-1}\y_\ell\y_\ell^\T + n^{-1}\Y_{\Ic(\ell)\setminus\{j,\ell\}}^\T \Y_{\Ic(\ell)\setminus\{j,\ell\}}
% $.
By the Sherman-Morrison formula, 
% (see also Eq. (\ref{eq:relation_sherman_morrison}))
\begin{equation}
    \label{eq:Tyl-inv-appendix-1}
    \frac{1}{p}\y_j^\T T^{-1}\y_\ell = \frac{\frac{1}{p}\y_j^\T T_{-\ell}^{-1} \y_\ell}{1 + \gamma\cdot \frac{1}{p}\y_\ell^\T T_{-\ell}^{-1}\y_\ell}
= \frac{\frac{1}{p}\y_j^\T T_{-\ell,-j}^{-1} \y_\ell}{\left(1 + \gamma\cdot \frac{1}{p}\y_\ell^\T T_{-\ell}^{-1}\y_\ell\right)\left(1 + \gamma\cdot \frac{1}{p}\y_j^\T T_{-\ell,-j}^{-1}\y_j\right)}\,.
\end{equation}
We next simplify $T_{-\ell,-j}^{-1}$. Decompose $T_{-\ell,-j} = T_{-j,-\Ic(\ell)} +  n^{-1}\Y_{\Ic(\ell)\setminus\{j,\ell\}}^\T \Y_{\Ic(\ell)\setminus\{j,\ell\}}$. Recall that by the Woodbury formula, 
for invertible matrices $A,C$, one has $\left(A+UCV\right)^{-1}  = A^{-1} - A^{-1}U\left(C^{-1} + VA^{-1}U\right)^{-1}VA^{-1}$. Applying this with $A=T_{-\ell,-j}^{-1}$, $C=I$, $U=V^\T=n^{-1/2}\Y_{\Ic(\ell)\setminus\{j,\ell\}}^\T$, gives
% Applying the Woodbury formula, Lemma~\ref{lem:sherman-morrison-woodbury}, to $T_{-\ell,-j} = T_{-j,-\Ic(\ell)} + n^{-1}\Y_{\Ic(\ell)\setminus\{j,\ell\}}^\T \Y_{\Ic(\ell)\setminus\{j,\ell\}}$,
% % gives,
\begin{equation}
    \begin{split}
        \label{eq:Tyl-inv-appendix-2}
         T_{-\ell,-j}^{-1} =  T_{-j,-\Ic(\ell)}^{-1} - n^{-1}T_{-j,-\Ic(\ell)}^{-1} \Y_{\Ic(\ell)\setminus\{j,\ell\}}^\T  \left(I + n^{-1}\Y_{\Ic(\ell)\setminus\{j,\ell\}} T_{-j,-\Ic(\ell)}^{-1} \Y_{\Ic(\ell)\setminus\{j,\ell\}}^\T  \right)^{-1} \Y_{\Ic(\ell)\setminus\{j,\ell\}} T_{-j,-\Ic(\ell)}^{-1}\,.
    \end{split}
\end{equation}
% For $i\in [M]$ and $j\in [n]$, denote
Let $\underline{\Omega}$ be the following upper bound on the denominator of (\ref{eq:Tyl-inv-appendix-1}):
\begin{equation}\label{eq:Tyl-Omegas-1}
\begin{split}
        \underline{\Omega} &= \max\left\{ \left(1 + \gamma\cdot \frac{1}{p}\y_\ell^\T T_{-\ell}^{-1}\y_\ell\right)\left(1 + \gamma\cdot \frac{1}{p}\y_j^\T T_{-\ell,-j}^{-1}\y_j\right)\;:\;j\in [n],\ell\in[n]\setminus\{j\} \right\} \,.
\end{split}
\end{equation}
Similarly, define
\begin{equation}\label{eq:Tyl-Omegas-2}
    \overline{\Omega} = \max \left\{ (np)^{-1}\left\| \Y_{\Ic(\ell)\setminus\{j,\ell\}}T_{-j,-\Ic(\ell)}^{-1}\y_\ell \right\|\left\| \Y_{\Ic(\ell)\setminus\{j,\ell\}}T_{-j,-\Ic(\ell)}^{-1}\y_j \right\| \;:\;j\in [n],\ell\in [n]\setminus\{j\}\right\}\,.
\end{equation}
Observe that per (\ref{eq:Tyl-inv-appendix-2}), $|\frac{1}{p}\y_j^\T T_{-\ell,-j}^{-1} \y_\ell|\ge |\frac{1}{p}\y_j^\T T_{-j,-\Ic(\ell)}^{-1} \y_\ell| - \overline{\Omega}$.
Finally, denote
\begin{equation}\label{eq:Tyl-xi-nu}
    \xi_{j,\ell} = p^{-1/2} \left\langle \frac{T_{-j,-\Ic(\ell)}^{-1}\y_j}{\left\|T_{-j,-\Ic(\ell)}^{-1}\y_j\right\|}, \y_\ell \right\rangle\,,\quad\textrm{ and } \nu = \min_{j,\ell} \left\|T_{-j,-\Ic(\ell)}^{-1}\y_j\right\| \,.
\end{equation}
Importantly, observe that the random variables $\xi_{j,\ell}$ within the same class $\ell\in \Ic_i$ are statistically independent of one another, conditioned on $\y_j$. 
Combining (\ref{eq:Tyl-inv-appendix-1})-(\ref{eq:Tyl-inv-appendix-3}) yields the following lower bound,
\begin{equation}\label{eq:Tyl-inv-appendix-3}
    \left|\frac{1}{p}\y_j^\T T^{-1}\y_\ell\right|  \ge p^{-1/2}|\xi_{j,\ell}|\frac{\nu}{\underline{\Omega}} - \frac{\overline{\Omega}}{\underline{\Omega}} \,.
\end{equation}
Next we derive high-probability bounds on $\overline{\Omega},\underline{\Omega},\nu$.
\begin{lemma}\label{lem:Tyl-inv-appendix-A}
    For a number $C_1$, let $\Ec_{\mathrm{Lem.}\ref*{lem:Tyl-inv-appendix-A}}$ be the event that: 1) $\underline{\Omega}\le C_1$; 2) $\nu\ge 1/C_1$; 3) $\overline{\Omega}\le C_1\frac{k}{n}$. 
    
    Assume $k\le \frac{1-\gamma}{1+\gamma}n-1$. There are $c,C,C_1>0$, that depend on the distribution of $Y$ and on $\gamma$, so that 
        \begin{enumerate}
        \item Assume \AssumLC. Then $\Pr(\Ec_{\mathrm{Lem.}\ref*{lem:Tyl-inv-appendix-A}}^c)\le Cn^2e^{-c\Psi_p\sqrt{k}}$.
        \item Assume \AssumIndep~or \AssumCCPSB. Then $\Pr(\Ec_{\mathrm{Lem.}\ref*{lem:Tyl-inv-appendix-A}}^c)\le Cn^2e^{-ck}$.
    \end{enumerate}
\end{lemma}
\begin{proof}
    We start by bounding $\underline{\Omega}$ and $\nu$. Considering their definitions, in (\ref{eq:Tyl-Omegas-1}) and (\ref{eq:Tyl-xi-nu}), it suffices to show that the following are all high-probability events (for some constant $C>0$): 
    (I) $\max_{j,\ell}\lambda_{\max}(T_{-j,\Ic(\ell)}) \le C$; (II) $\min_{j,\ell}\lambda_{\max}(T_{-j,\Ic(\ell)}) \ge 1/C$; (III) $\max_{\ell}|p^{-1}\|\y_\ell\|^2-1|\le \frac12$. Observe that $T_{-j,\Ic(\ell)}$ is, 
    % excludes at most $k+1\le \frac{1-\gamma}{1+\gamma}n$ samples, and is therefore, 
    up to the normalization, the sample covariance of at least $n-k-1\ge \frac{2\gamma}{1+\gamma}n=\frac{1}{1/2+\gamma/2}p$ i.i.d. measurements. Thus, 
    conditions (I) and (II) can be verified using Lemmas~\ref{lem:main-smax} and~\ref{lem:main-smin} respectively.
    Condition (III) can be verified using Lemma~\ref{lem:main-quadratic}. We omit the details.
    
    We next consider $\overline{\Omega}$, defined in (\ref{eq:Tyl-Omegas-2}). Let us bound, $\omega_{j,\ell}:=(np)^{-1}\left\| \Y_{\Ic(\ell)\setminus\{j,\ell\}}T_{-j,-\Ic(\ell)}^{-1}\y_\ell \right\|^2$, so that a high-probability bound on $\overline{\Omega}$ may be attained by union bound over $\ell,j$. Denote $u_{j,\ell}=T_{-j,-\Ic(\ell)}^{-1}\y_\ell$, $\hat{u}_{j,\ell}={u}_{j,\ell}/\|{u}_{j,\ell}\|$ so $\omega_{j,\ell}=\frac{k}{n} \cdot p^{-1}\|{u}_{j,\ell}\|^2 \cdot \hat{u}_{j,\ell}^\T \left(\frac1k\Y_{\Ic(\ell)\setminus\{j,\ell\}}^\T \Y_{\Ic(\ell)\setminus\{j,\ell\}}\right) \hat{u}_{j,\ell} $. The terms $p^{-1}\|u_{j,\ell}\|^2$ may be treated upper-bounded similarly to the previous paragraph. As for the term $\hat{u}_{j,\ell}^\T \left(\frac1k\Y_{\Ic(\ell)\setminus\{j,\ell\}}^\T \Y_{\Ic(\ell)\setminus\{j,\ell\}}\right) \hat{u}_{j,\ell}$, observe that $\frac1k\Y_{\Ic(\ell)\setminus\{j,\ell\}}^\T \Y_{\Ic(\ell)\setminus\{j,\ell\}}$ is a sample covariance matrix consisting of (at most) $k$ samples, and $\hat{u}_{j,\ell}$ is a unit vector which is statistically independent of $\Y_{\Ic(\ell)\setminus\{j,\ell\}}$. By Lemma~\ref{lem:sc-entrywise}, this quadratic form is bounded by a constant w.h.p.
    
\end{proof}
Next, we show that w.h.p., there are many large $|\xi_{j,\ell}|$-s. 
For a number $\alpha$, let $\tilde{N}_j[\alpha]$ be the number of variables $\xi_{j,\ell}$ ($1\le \ell \le n$) in row $j$, such that $|\xi_{j,\ell}|\ge \alpha$.
% \begin{equation}
%     N_j[\alpha] = \sum_{\ell=1}^n \indic{|\xi_{j,\ell}|\ge \alpha} \,.
% \end{equation}
\begin{lemma}\label{lem:Tyl-inv-appendix-3}
    There are $c,C,\alpha_*>0$, that depend on the distribution of $Y$, so that ${\Pr(\min_{1\le j\le n} \tilde{N}_j[\alpha_*]\le 0.6n)\le Cn^2e^{-ck}}$.
    % under either \AssumLC, \AssumIndep~or \AssumCCPSB, we have 
    % \[
    %     \Pr(\min_{1\le j\le n} N_j[\alpha_*]\le 0.6n)\le Cn^2e^{-ck} \,.
    % \]
    \begin{proof}
        For $1\le i \le M$, let $\tilde{N}_j^i[\alpha] = \sum_{\ell\in \Ic_i} \indic{|\xi_{j,\ell}|\ge \alpha}$, so that $\tilde{N}_j[\alpha]=\sum_{i=1}^M \tilde{N}_j^i[\alpha]$. 
        % Note, moreover, that we may assume w.l.o.g. that the partition $\{\Ic_i\}$ was chosen such that $|\Ic_i|\ge k/2$ for all $i$. 
        Let $\mathcal{F}_{j,i}$ be the $\sigma$-algebra generated by $\{\y_j\}\cup \{\y_\ell\}_{\ell\notin \Ic_i }$. Conditioned on $\mathcal{F}_{j,i}$, $\{\xi_{j,\ell}\}_{\ell\in\Ic_i\setminus\{j\}}$ are i.i.d. 
        Since $Y$ satisfies the SBP, Definition~\ref{def:small-ball}, there is some $\alpha_*$ such that $\Pr(|\xi_{j,\ell}|\ge \alpha_*|\mathcal{F}_{j,i})\ge 0.8$. By Hoeffding's inequality, $\Pr(\tilde{N}_j^i[\alpha_*]\le 0.7|\Ic_i\setminus\{j\}|\;|\mathcal{F}_{j,i}) \le 2e^{-c_1|\Ic_i\setminus\{j\}|} \le 2e^{-c_2k}$. Taking a union bound over $i\in [M]$, w.p. $\ge 1-2ne^{-c_2k}$ it holds that $\tilde{N}_j^i[\alpha_*]\ge 0.7|\Ic_i\setminus\{j\}|$ simultaneously for all $i$, and in particular $\sum_{i=1}^M \tilde{N}_j^i[\alpha_*] \ge 0.7(n-1)$. To finish the proof of the Lemma, take a union bound over all $1\le j \le n$.
    \end{proof}
\end{lemma}

We are ready to conclude the proof of Lemma~\ref{lem:Tyl-inv-bound}. Set $k=c_K n^{1/2}$, for a small constant $c_K>0$, to be chosen momentarily. By Lemma~\ref{lem:Tyl-inv-appendix-A} and Eq. (\ref{eq:Tyl-inv-appendix-3}), there are $C_1,C_2,\alpha_*$ such that w.h.p.,
\[
p^{-1}|\y_j T^{-1}\y_\ell| \ge n^{-1/2}(C_1|\xi_{j,\ell}|-C_2c_K)    \quad \textrm{for all}\; j,l\in [n],\,\ell \ne j 
\]
and 
\[
    \tilde{N}_j[\alpha_*]\ge 0.6n \quad \textrm{ for all } j\in[n]    \,.
\]
Accordingly, choose $c_K$ so that $C_2c_K\le 0.5C_1\alpha_*$.
Recall, by Eq. (\ref{eq:Tyl-Zhang-1}), that $A_{j,\ell}=\gamma \left(p^{-1}|\y_j T^{-1}\y_\ell|\right)^2$. Taking $\beta_0=\gamma(0.5C_1\alpha_*)^2$, 
observe that the high-probability event above implies that $\min_{1\le j \le n}N_j[\beta_0]\ge 0.6n$, that is,
each row $j$ of $A$ contains at least $0.6n$ entries $A_{j,\ell}$ satisfying $A_{j,\ell}\ge \beta_0/n$. As explained in the beginning of this section, this establishes the proof of the Lemma.
% . Taking $\beta=(0.5C_1\alpha_*)^2 p^{-1}$, bounding (\ref{eq:A-minus-beta}) under this event,
% \[
%     \|\gamma A - \beta\oneVec\oneVec^\T\|_{\infty,\infty} = \max_{1\le j \le n} \sum_{\ell=1}^n |\gamma A_{j,\ell}-\beta| \le \max_{1\le j \le n} \sum_{\ell=1}^n A_{j,\ell} - 0.2( 0.5C_1\alpha_*)^2 = \max_{1\le j \le n} p^{-1}\y_j^\T T^{-1}\y_j - C_3   \,.
% \]
% By Lemma~\ref{lem:main-technical}, $\max_{1\le j \le n} p^{-1}\y_j^\T T^{-1}\y_j \le 0.5C_3$ is a high-probability event. We conclude that for some $c>0$, $\|\gamma A - \beta\oneVec\oneVec^\T\|_{\infty,\infty}\le 1-c$ holds w.h.p. This implies the desired result of Lemma~\ref{lem:Tyl-inv-bound}, as argued in the beginning of this section.
\qed

\subsection{Proof of Lemma~\ref{lem:Tyl-smooth}}
\label{sec:proof:lem:Tyl-smooth}

Denote $S(\w)=\frac1n \sum_{i=1}^n w_i \x_i\x_i $, $T(\w)=\frac1n \sum_{i=1}^n w_i\y_i\y_i^\T$, so that $\x_j^\T S^{-1}(\w)\x_\ell = \y_j^\T T^{-1}(\w)\y_\ell$.
% Note that $\|T(\w_1)-T(\w_2)\|\le \|T\|\|\w_1-\w_2\|_\infty$. 
By (\ref{eq:Tyl-hessian}), 
% for all $\beta$,
\begin{align*}
    \left| \left(\nabla g_\beta(\w)-\nabla g_\beta(\oneVec)\right)_{j,\ell} \right| 
% &\le \left| \frac{1}{w_\ell^2}-1\right|\indic{j=\ell} + \frac{\gamma}{p^2}\left| (\x_j^\T S^{-1}(\oneVec)\x_\ell)^2-(\x_j^\T S^{-1}(\w)\x_\ell)^2\right| 
= \left| \frac{1}{w_\ell^2}-1\right|\indic{j=\ell} + \frac{1}{np}\left| (\y_j^\T T^{-1}(\oneVec)\x_\ell)^2-(\y_j^\T T^{-1}(\w)\y_\ell)^2\right| \,.
\end{align*}
If $\|\w-\oneVec\|_\infty\le 1/2$ then $\left| \frac{1}{w_\ell^2}-1\right| = \left| \frac{1}{w_\ell^2}+\frac{1}{w_\ell}\right| |1-w_\ell|\le 6\|\w-\oneVec\|_\infty$. And so, for all $1\le j \le n$,
\[
\sum_{\ell=1}^n \left| \left(\nabla g_\beta(\w)-\nabla g_\beta(\oneVec)\right)_{j,\ell} \right| \le 6\|\w-\oneVec\|_\infty + \frac{1}{np} \sum_{\ell=1}^n \left| (\y_j^\T T^{-1}(\oneVec)\y_\ell)^2-(\y_j^\T T^{-1}(\w)\y_\ell)^2\right| \,.
\]
% To prove the Lemma, we need to show that, w.h.p., the right-hand-side is $\le L\|\w-\oneVec\|_\infty$ for some $L$. 
Write 
\begin{align*}
    (\y_j^\T T^{-1}(\oneVec)\y_\ell)^2-(\y_j^\T T^{-1}(\w)\y_\ell)^2 
    &= \left((\y_j^\T T^{-1}(\oneVec)\y_\ell)+(\y_j^\T T^{-1}(\w)\y_\ell)\right) \left( (\y_j^\T T^{-1}(\oneVec)\y_\ell)-(\y_j^\T T^{-1}(\w)\y_\ell)\right) \\
    &= \y_j^\T \left( T^{-1}+T^{-1}(\w) \right) \y_l \cdot \y_j^\T \left( T^{-1}-T^{-1}(\w) \right) \y_l \,.
\end{align*}
By
Cauchy-Schwartz, $\frac{1}{np}\sum_{\ell=1}^n \left| (\y_j^\T T^{-1}(\oneVec)\y_\ell)^2-(\y_j^\T T^{-1}(\w)\y_\ell)^2\right|\le  (I_1 I_2)^{1/2}$, where
\begin{equation*}
    \begin{split}
        &I_1 = \frac{1}{np}\sum_{\ell=1}^n \left( \y_j^\T \left( T^{-1}+T^{-1}(\w) \right) \y_l \right)^2 = p^{-1}\y_j^\T \left( (T^{-1}+T^{-1}(\w))T(T^{-1}+T^{-1}(\w)) \right) \y_j\,,\\
        &I_2 = \frac{1}{np}\sum_{\ell=1}^n \left( \y_j^\T \left( T^{-1}-T^{-1}(\w) \right) \y_l \right)^2 = p^{-1}\y_j^\T \left( (T^{-1}-T^{-1}(\w))T(T^{-1}-T^{-1}(\w)) \right) \y_j\,.
    \end{split}
\end{equation*}
$\|\w-\oneVec\|_\infty\le 1/2$ implies $\|T^{-1}+T^{-1}(\w)\|\le 3\|T^{-1}\|$, and $\|T^{-1}-T^{-1}(\w)\|= \|T^{-1}(T(\w)-T)(T^{-1}(\w))\|\le (3/2)\|T^{-1}\|^2\|T\|\|\w-\oneVec\|_\infty$. Thus, for numerical $c>0$, ${(I_1I_2)^{1/2}\le cp^{-1}\|\y_{j}\|^2\|T^{-1}\|^3\|T\|^2\|\w-\oneVec\|_\infty}$. We get
\begin{align*}
    \|\nabla g_\beta(\w)-\nabla g_\beta(\oneVec)\|_{\infty,\infty}
    &=\max_{1\le j \le n} \sum_{\ell=1}^n \left| \left(\nabla g_\beta(\w)-\nabla g_\beta(\oneVec)\right)_{j,\ell} \right| 
    \le \left(3 + c\|T^{-1}\|^3\|T\|^2\max_{1\le j \le n}p^{-1}\|\y_j\|^2\right)\|\w-\oneVec\|_\infty \,.
\end{align*}
For conclude the proof, recall  that w.h.p.: 1) $\|T\|\le C_1$ (by Lemma~\ref{lem:main-smax}); 2) $\|T^{-1}\|\le C_2$ (by Lemma~\ref{lem:main-smin}); and 3) $\max_{1\le j \le p} p^{-1}\|\y_j\|^2\le C_3$ (by Lemma~\ref{lem:main-quadratic} and a union bound over $1\le j \le n$).
\qed

% \newpage
\section{Proof of Lemmas for Theorems~\ref{thm:main_maronna_reg} (MRE) and \ref{thm:main_tyler_reg} (TRE)}

\subsection{Proof of Lemma~\ref{lem:MRE-1}}
\label{sec:proof-lem:MRE-1}

% Recall the definition of $\hat{Q}_i$ in (\ref{eq:proof-mre-2}) and $Q$ in  (\ref{eq:proof-mre-4}), so that $Q(d)=\E\hat{Q}_i(d)$. 
Write $\hat{Q}_i(d)-Q(d)=\Delta_1+\Delta_2$ where
\begin{align*}
    &\Delta_1 = p^{-1}\x_i^\T(\phi(d)S_{-i}+\alpha d \Id)^{-1}\x_i-p^{-1}\tr \trueCov(\phi(d)S_{-i}+\alpha d \Id)^{-1} \,, \\ 
    &\Delta_2 = p^{-1}\tr \trueCov(\phi(d)S_{-i}+\alpha d \Id)^{-1} - p^{-1}\E\tr \trueCov(\phi(d)S_{-i}+\alpha d \Id)^{-1} \,. 
\end{align*}

To bound $|\Delta_1|$, use Lemma~\ref{lem:main-quadratic}, applied to the quadratic form $p^{-1}\y_i^\T B \y_i$, where $B=\trueCov^{1/2}(\phi(d)S_{-i}+\alpha d \Id)^{-1}\trueCov^{1/2}$. Since $\|\trueCov\|\le s_{\max}$ and $d\ge d_0$ then $\|B\|\le s_{\max}(\alpha d_0)^{-1}$. 
To bound $|\Delta_2|$, use Lemma~\ref{lem:resolvent-conc} with $S_n=S_{-i}$, $C=\trueCov$, 
$A=\alpha d \Id$. 

\qed

\subsection{Proof of Lemma~\ref{lem:MRE-2}}
\label{sec:proof-lem:MRE-2}

We first show that $d^*$ indeed exists. By (\ref{eq:proof-mre-4}), the functions $Q,F$ are continuous and strictly decreasing. 
Since $\phi(d)=du(d)$ and $u$ is bounded, then $\lim_{d\to0}\phi(d)=0$. This, in turn, implies
$\lim_{d\to 0}Q(d)=\infty$ and $\lim_{d\to 0}F(d)=\infty$. Moreover, $\lim_{d\to \infty}Q(d)=0$, and so $ \lim_{d\to \infty}F(d)=0$. 
Consequently, $d^*=F^{-1}(1)$ exists uniquely.
% Since $F$ is continuous (and positive), $\mathrm{Range}(F)=(0,\infty)$; since it is strictly decreasing, $d^*=F^{-1}(1)$ is a unique number.

Next, we bound $d^*$. To this end, we first upper bound $F$. By Eq. (\ref{eq:proof-mre-4}), 
% $Q(d) = p^{-1}\E \tr\trueCov(\phi(d)S_{-i}+\alpha d \Id)^{-1}$, and so 
$Q(d) \le p^{-1}\E \tr\trueCov(\alpha d \Id)^{-1} = \trCov/(\alpha d)$, where $\trCov=p^{-1}\tr\trueCov \le s_{\max}$. Since $Q(d)\ge 0$, clearly $F(d)\le (1+\alpha)Q(d)\le (s_{\max}\frac{1+\alpha}{\alpha})d^{-1}$. Setting $d=d^*$, $F(d^*)=1$, we conclude $d^* \le \overline{d}=\frac{1+\alpha}{\alpha}s_{\max}$.

To lower bound $d^*$, we need a lower bound on $Q$.
Clearly, $S_{-i}\preceq \|S_{-i}\|\Id $, 
so
\begin{equation}\label{eq:Q-lb}
    Q(d) \ge
    p^{-1}\tr \trueCov \E \left(\phi(d)\|S_{-i}\|\Id+\alpha d \Id\right)^{-1}  \ge \trLB\cdot \E\frac{1}{\phi(d)\|S_{-i}\|+\alpha d} \,.
\end{equation}
% Note that $S_{-i}=\trueCov^{1/2}T_{-i}\trueCov^{1/2}\preceq \|T_{-i}\|\trueCov$, where $\preceq$ denotes PSD order, and $\phi(d)=u(d)d\le u(0)d$, since $u$ is decreasing. Consequently, 
% \begin{equation}
%     Q(d) 
%     \ge p^{-1}\E \tr \trueCov(\phi(d)\|T_{-i}\|\trueCov+\alpha d \Id)^{-1} = d^{-1} p^{-1}\E \tr\trueCov(u(0)d\|T_{-i}\|\trueCov + \alpha\Id)^{-1}   \,.
% \end{equation}
By Lemma~\ref{lem:main-smax}, there is some $C_0>0$ such that $\|S_{-i}\|\le s_{\max} C_0$ holds w.p. $\ge \frac12$. Moreover, $\phi(d)=u(d)d\le u(0)d$, since $u$ is decreasing. Combining this with (\ref{eq:Q-lb}) yields $Q(d)\ge \frac12 \frac{\trLB}{u(d)dC_0 s_{\max}+\alpha d} = C_1/d$.
% \[
%     Q(d) \ge \frac12 d^{-1}p^{-1}\tr\trueCov(C_0 u(0)\trueCov+\alpha\Id)^{-1} \ge \frac{\frac12\trCov}{Cs_{\max} + \alpha}d^{-1} \ge \frac{\frac12\trLB}{Cs_{\max} + \alpha}d^{-1} \,.
% \]   
Plugging  this bound and the previously derived upper bound $Q(d)\le C_2/d$ into (\ref{eq:proof-mre-4}) yields $F(d)\ge  (1+\alpha)\frac{C_1/d}{1+\gamma \phi(d)(C_2/d)} \ge (1+\alpha) \frac{C_1/d}{1+\gamma u(0)C_2}$, 
from which a lower bound on $d^*$ follows.

Finally, it remains to show that for some $\eta$, $|F(d_1)-F(d_2)|\ge \eta|d_1-d_2|$ inside the interval $[\underline{d},\overline{d}]$. Let $d_1\le d_2$. 
% By the definition of $F$, Eq. (\ref{eq:proof-mre-4}),
Then,
\begin{align*}
    (1+\alpha)^{-1}(F(d_1)-F(d_2)) 
    &= \frac{Q(d_1)}{1+\gamma \phi(d_1)Q(d_1)} - \frac{Q(d_2)}{1+\gamma \phi(d_2)Q(d_2)} 
    = \frac{Q(d_1)-Q(d_2)+\gamma (\phi(d_2)-\phi(d_1))Q(d_1)Q(d_2)}{(1+\gamma \phi(d_1)Q(d_1))(1+\gamma \phi(d_2)Q(d_2))} \\
    &\overset{(\star)}{\ge} \frac{Q(d_1)-Q(d_2)}{(1+\gamma \phi(d_1)Q(d_1))(1+\gamma \phi(d_2)Q(d_2))}\,,
\end{align*}
where in $(\star)$ we used $\phi(d_2)-\phi(d_1)\ge 0$, since $\phi$ is non-decreasing. The denominator is upper bounded by a constant inside the interval,
% (follows from our previous bounds on $Q$), 
and the numerator is non-negative. Thus, it suffices to show that $Q(d_1)-Q(d_2)\ge \eta_0(d_1-d_2)$ for some $\eta_0>0$. Since $\phi$ is non-decreasing,
$
Q(d_1) = p^{-1}\E\tr\trueCov(\phi(d_1)S_{-i}+\alpha d_1\Id)^{-1} \ge p^{-1}\E \tr\trueCov(\phi(d_2)S_{-i}+\alpha d_1\Id)^{-1}
$,
and so
\[
Q(d_1)-Q(d_2) \ge p^{-1}\E\tr\trueCov (\phi(d_2)S_{-i}+\alpha d_2\Id)^{-1})( (d_1-d_2)\Id ) (\phi(d_2)S_{-i}+\alpha d_1\Id)^{-1} \,.
\]
Consider the matrix $A=(\phi(d_2)S_{-i}+\alpha d_2\Id)^{-1})(\phi(d_2)S_{-i}+\alpha d_1\Id)^{-1}$. It is positive, being the product of {commuting} positive matrices. Consequently, $Q(d_1)-Q(d_2)\ge (d_1-d_2)p^{-1}\E \Tr \trueCov A \ge (d_1-d_2)\trCov \E\lambda_{\min}(A)$. Lastly, $\E\lambda_{\min}(A)\ge \eta_0$ for some $\eta_0$ since, by Lemma~\ref{lem:main-smax}, for some $C$, $\|S_{-i}\|\le Cs_{\max}$ holds w.h.p.

\qed

\subsection{Proof of Lemma~\ref{lem:TRE-1}}
\label{sec:proof-lem:TRE-1}

We focus on Item 1), namely showing the existence and boundedness of $d^*$. The proof of Item 2) is identical to the corresponding part in Lemma~\ref{lem:MRE-2} (MRE).
Using $\phi(d)=1$ in (\ref{eq:proof-mre-4}), $F(d^*)=1$ is equivalent to 
\begin{equation}\label{eq:TRE-proof-1}
    Q(d^*) = \frac{1}{1+\alpha-\gamma} \,,\quad \textrm{ where }\quad Q(d)=p^{-1}\E\Tr \trueCov(S_{-i}+\alpha d \Id)^{-1} \,.
\end{equation}
Since $\alpha>\max\{0,\gamma-1\}$ then $\frac{1}{1+\alpha-\gamma}>0$. 
The function $Q$ is positive, continuous and strictly decreasing, with $\lim_{d\to\infty} Q(d)=0$. 
We have $Q(d)\le s_{\max}/(\alpha d)$, and therefor if a solution $Q(d^*)=(1+\alpha-\gamma)^{-1}$ exists, then necessarily $d^* \le \overline{d} \le \frac{s_{\max}}{\alpha}(1+\alpha-\gamma)$. As for establishing existence, by continuity it suffices to show that $\lim_{d\to0} Q(d)>(1+\alpha-\gamma)^{-1}$; in other words, we need to study the behavior of $Q(d)$ near $d=0$. To this end, we consider separately the regimes $\gamma<1$ and $\gamma\ge 1$, noting that $S_{-i}$ is only invertible (w.p. 1) in the regime $p/n=\gamma<1$. 

\paragraph{The case $\gamma<1$} Note that $(1+x)^{-1}\ge 1-x$ for all $x\ge 0$. Thus, for any non-negative matrix $P$, $(I+P)^{-1}\succeq I-P$. Write $\trueCov^{1/2}(S_{-i}+\alpha d \Id)^{-1}\trueCov^{1/2} 
    = (T_{-i}+\alpha d \trueCov^{-1})^{-1} 
    = T_{-i}^{-1/2}(\Id + \alpha d T_{-i}^{-1/2}\trueCov^{-1}T_{-i}^{-1/2})T_{-i}^{-1/2}$, and so $\trueCov^{1/2}(S_{-i}+\alpha d \Id)^{-1}\trueCov^{1/2} \succeq T_{-i}^{-1} - \alpha d T_{-i}^{-1}\trueCov^{-1} T_{-i}^{-1} 
    \succeq T_{-i}^{-1} - \alpha d \trueCov^{-1}/\lambda_{\min}(T_{-i})^2$. 
% \begin{align*}
%     \trueCov^{1/2}(S_{-i}+\alpha d \Id)^{-1}\trueCov^{1/2} 
%     &= (T_{-i}+\alpha d \trueCov^{-1})^{-1} 
%     = T_{-i}^{-1/2}(\Id + \alpha d T_{-i}^{-1/2}\trueCov^{-1}T_{-i}^{-1/2})T_{-i}^{-1/2} \\
%     &\succeq T_{-i}^{-1} - \alpha d T_{-i}^{-1}\trueCov^{-1} T_{-i}^{-1} 
%     \succeq T_{-i}^{-1} - \alpha d \trueCov^{-1}/\lambda_{\min}(T_{-i})^2 \,.
% \end{align*}
For an event $\Ec$, denote for brevity $\E^{\Ec}[\cdot] = \E[\cdot \indic{\Ec}]$. 
By Lemmas~\ref{lem:main-technical-minus-i} and \ref{lem:main-smin}, assuming large enough $n$, there is $C_*=C_*(\gamma,Y,\alpha)$ such that for the event $\Ec=\{\lambda_{\min}(T_{-i})\ge C_*\}$, we have $\E^{\Ec}[p^{-1}\tr T_{-i}^{-1}] \ge \frac12 (\frac{1}{1-\gamma} + \frac{1}{1-\gamma+\alpha})$. Thus,
\[
Q(d)\ge p^{-1}\E^{\Ec}\tr\trueCov^{1/2}(S_{-i}+\alpha d \Id)^{-1}\trueCov^{1/2} 
\ge p^{-1}\E^{\Ec}T_{-i}^{-1} - \alpha d p^{-1}\E^{\Ec}\tr \trueCov^{-1}/\lambda_{\min}(T_{-i})^2
\ge \frac12 \left(\frac{1}{1-\gamma} + \frac{1}{1-\gamma+\alpha}\right) - \frac{\alpha }{C_*^2} \trInvUB \cdot d \,.
\]
% where we also used $p^{-1}\tr\trueCov^{-1}\le \trInvUB$.
% Since the quantity under the expectation is non-negative, $Q(d)\ge p^{-1}\E^{\Ec}\tr\trueCov^{1/2}(S_{-i}+\alpha d \Id)^{-1}\trueCov^{1/2} $, and therefore
% \[
% Q(d) \ge  \frac12 \left(\frac{1}{1-\gamma} + \frac{1}{1-\gamma+\alpha}\right) - \frac{\alpha d}{C_*^2} p^{-1}\tr\trueCov^{-1} \,.
% \]
This implies that $\lim_{d\to 0}Q(d)>(1-\gamma+\alpha)^{-1}$;  moreover,  
setting $d=d^*$, yields an explicit lower bound on $d^*$.
% and the assumed bound $p^{-1}\tr\trueCov^{-1}\le \trInvUB$, we get 
% \[
% d^*\ge \underline{d} =     \frac{C_*^2}{2\alpha\trInvUB} \left(\frac{1}{1-\gamma} - \frac{1}{1-\gamma+\alpha}\right) \,.
% \]

\begin{remark}\label{remark:goes-alpha}
    When $\alpha$ is sufficiently large, we can obtain a lower bound $\underline{d}$ which does not depend on $p^{-1}\tr\trueCov^{-1}$, similarly to the analysis of \cite{goes2017robust}. 
    They use (\ref{eq:Q-lb}), recalling that by Lemma~\ref{lem:main-smax},  there is $\overline{C}=\overline{C}(\gamma,Y)$ such that $\Pr(\|S_{-i}\|\le \overline{C} s_{\max})=1-o(1)$. Thus, $Q(d)\ge (1-o(1))\trCov (\overline{C}s_{\max} + \alpha d)^{-1}$, which
    % , assuming large $n$, 
    yields a positive lower bound on $d^*$ whenever $\frac{1}{1+\alpha-\gamma} < \frac{\trLB}{\overline{C}s_{\max}} $, that is, $\alpha > \gamma - 1 + \overline{C}s_{\max}/\trLB$. The resulting lower bound may be arbitrarily better (larger) than the previously derived lower bound (which depends on $\trInvUB$), since $p^{-1}\tr\trueCov^{-1}$ may be very large when $\trueCov$ has only one eigenvalue close to $0$.
    % Their argument goes as follows: By Lemma~\ref{lem:main-smax},  there is $\overline{C}=\overline{C}(\gamma,Y)$ such that $\Pr(\|S_{-i}\|\le \overline{C} s_{\max})=1-o(1)$. Thus, $Q(d)\ge (1-o(1))\trCov (\overline{C}s_{\max} + \alpha d)^{-1}$, which, assuming large $n$, yields a positive lower bound on $d^*$ whenever $\frac{1}{1+\alpha-\gamma} < \frac{\trLB}{\overline{C}s_{\max}} $, that is, $\alpha > \gamma - 1 + \overline{C}s_{\max}/\trLB$. Observe that the resulting lower bound may be arbitrarily better than the previously derived $\underline{d}$, since $p^{-1}\tr\trueCov^{-1}$ may be made arbitrarily large by only driving  a single eigenvalue of $\trueCov$ arbitrarily close to $0$.
\end{remark}

\paragraph{The case $\gamma\ge 1$} 
% Starting with an upper bound, $Q(d)\le s_{\max}(\alpha d)^{-1}$ holds in this regime as well, yielding an identical upper bound $d^*\le \overline{d}$. 
% As before, we need a lower bound on $Q$; we do this 
We analyze this case
essentially by reduction 
% Next, we derive a lower bound on $d^*$ by essentially reduction 
to the case $\gamma<1$. Making explicit the dependence of $Q$ on $n$, denote $Q_n(d)=p^{-1}\E\tr \trueCov(S^{(n)}_{-i}+\alpha d \Id)^{-1}$, where $S^{(n)}=n^{-1}\sum_{j=1}^n \x_j\x_j^\T$ is the sample covariance of $n$ i.i.d. measurements. 
For an integer $m\ge 0$, let $\x_{n+1},\ldots,\x_{n+m}$ be $m$ new i.i.d. samples from $X$. W.p. 1,
\begin{align*}
    \tr \trueCov(S^{(n)}_{-i}+\alpha d \Id)^{-1} 
    &\ge p^{-1}\Tr \trueCov\left( \frac{1}{n}\sum_{j=1}^{n+m-1}\x_j\x_j^\T + \alpha d \Id \right)^{-1} 
    = \frac{n}{n+m} p^{-1}\Tr \trueCov\left( S^{(n+m)}_{-i} + \alpha \frac{n}{n+m} d \Id \right)^{-1}\,.
\end{align*}
Consequently, $Q_n(d)\ge \frac{n}{n+m}Q_{n+m}\left( \frac{n}{n+m}d \right)$. For $n+m-p=\Omega(p)$, Lemma~\ref{lem:main-technical-minus-i} implies, assuming $n$ is large, that $Q_{n+m}(0)\ge 0.99 \frac{1}{1-\frac{p}{n+m}}$, hence $\frac{n}{n+m}Q_{n+m}(0)\ge 0.99\frac{n}{n+m-p}=0.99\frac{1}{1+m/n-\gamma}$. Clearly, we may set $m=c_0 n$ for some $c_0=c_0(\gamma,\alpha)$ such that $0.99\frac{1}{1+c_0-\gamma}$ is larger than $\frac{1}{1-\gamma+\alpha}$ by a constant. We can then lower bound $Q_{n+m}(d)$ following the same argument as in the case $\gamma<1$, noting that $S^{(n+m)}$ is invertible w.p. $1$; we omit the details.
\qed

% \newpage
\section{Proof of Additional Technical Lemmas and Symmetrization Properties}

\subsection{Proof of Lemma~\ref{lem:main-smin}}
\label{sec:proof:lem:main-smin}

Recall that under either assumption on $Y$, \AssumIndep, \AssumCCPSB~or \AssumLC, it satisfies the SBP with some constant $C_0$ (for \AssumIndep~and \AssumLC, see Lemmas~\ref{lem:sbp-indep} and \ref{lem:logconcave-smallball} respectively). Let $\Y \in \R^{n\times p}$ be the matrix whowse rows are $\y_1^\T,\ldots,\y_n^\T$. Since $T=n^{-1}\Y^\T \Y$, our goal is to show that w.h.p.,
$\sigma_{\min}(\Y) = \min_{v\in \Sp} \|\Y v\| = \Omega(\sqrt{n})$.
% holds w.h.p. 

The following proof is based on \cite[Corollary 4.6]{rudelson2013lecture}.
We first show that for any \emph{fixed} $v\in \Sp$, $\|\Y v\|$ is large w.h.p.; the desired result will then follow by a standard net argument. Let $t>0$ and $s\in (0,1)$; observe that whenever $\|\Y v\|^2\le tn$, then there are at most $sn$ entries of $\Y v$ for which $|(\Y v)_i|^2>t/s$; equivalently, there are (at least) $(1-s)n$ entries for which $|(\Y v)_i|^2 \le t/s$. 
Note that all $n$ rows of $\Y v$ are i.i.d., with the same law as $v^\T Y$. Taking a union bound over all possible subsets $S$ of $[n]$ with size $(1-s)n$, corresponding to ``small'' coordinates in $\Y v$, 
\begin{equation}
    \label{eq:smin-proof-1}
    \Pr\left( \|\Y v\|^2 \le tn \right) \le \binom{n}{(1-s)n}\left( \Pr(|v^\T Y|^2\le t/s) \right)^{(1-s)n} \le \left( C_0 \sqrt{\frac{t}{s}} \frac{e}{(1-s)} \right)^{(1-s)n} \,.
\end{equation}
Above, we used $\binom{n}{k}\le (en/k)^k$ and the small-ball property for $Y$, Definition~\ref{def:small-ball}. 

Let $\epsilon_0\in (0,1)$, to be chosen later, and let $\Nc$ be an $\epsilon_0$-net of $\Sp$ of minimal size. By a standard packing argument \cite[Lemma 5.2]{Vershynin}, $|\Nc|\le (1+\frac{2}{\epsilon_0})^p \le (3/\epsilon_0)^p = (3/\epsilon_0)^{\gamma n}$. Now,
\begin{equation}
% \label{eq:smin-proof-2}
    \sigma_{\min}(\Y) = \min_{v\in \Sp} \|\Y v\| \ge \min_{v\in \Sp} \min_{v^*\in \Nc} \left\{ \|\Y v^*\| - \|\Y (v-v^*)\|  \right\} \ge \min_{v^*\in \Nc} \|\Y v^*\| - \sigma_{\max}(\Y)\epsilon_0 \,.
\end{equation}
By Lemma~\ref{lem:main-smax}, there is $C_1$ such that w.p. $\ge 1-e^{-c\sqrt{n}}$, $\sigma_{\max}(\Y)\le C_1\sqrt{n}$. Thus, it suffices to show that for some fixed $\epsilon_0$, w.h.p.,  $\min_{v^*\in \Nc} \|\Y v^*\| > 2C_1 \epsilon_0\sqrt{n}$. Using \eqref{eq:smin-proof-1} with $t=(2C_1\epsilon_0)^2$,
% and taking a union bound over $\Nc$,
\begin{equation}
\label{eq:smin-proof-3}
    \begin{split}
        \Pr\left( \min_{v^*\in \Nc}\|\Y v^*\| \le 2C_1\epsilon_0 \sqrt{n} \right) 
        &\le |\Nc| \left(  \frac{2e C_0C_1}{\sqrt{s}(1-s)} \epsilon_0 \right)^{(1-s)n} 
        % &\le \left\{ \left( 2C_0C_1 \frac{e}{\sqrt{s}(1-s)} \epsilon_0 \right)^{1-s} \left( 1 + \frac{2}{\epsilon_0}\right)^{\gamma} \right\}^n 
        \le \left( C_3 (\sqrt{s}(1-s))^{s-1} \epsilon_0^{1-s-\gamma} \right)^n \,.
    \end{split}
\end{equation}
Recall that $\gamma<1$, and fix any $s\in (0,1-\gamma)$. As $\epsilon_0\to 0$, the RHS of (\ref{eq:smin-proof-3}) tends to zero. Thus, for all small enough (but constant) $\epsilon_0$, the RHS of (\ref{eq:smin-proof-3}) is $\le e^{-C_4 n}$. As discussed above, this concludes the proof of the Lemma.
\qed
% We conclude that for some (in fact, all) sufficiently small $\epsilon_0$, depending on $C_0,C_1,s,\gamma$, the bound in \eqref{eq:smin-proof-2} is $\le e^{-C_3 n}$ for some appropriate $C_3>0$, and the claim follows.

\subsection{Proof of Lemma~\ref{lem:main-quadratic}}
\label{sec:proof:lem:main-quadratic}

We prove Lemma~\ref{lem:main-quadratic} assuming $Y$ is an isotropic log-concave random vector.
Denote the ball $\Bc = \left\{ y\in \R^p\,:\,\|y\|\le 2\sqrt{p} \right\}$,
and let $Y_{\Bc}$ be a random vector distributed according to the law of $Y$, conditioned on $Y\in \Bc$. Clearly,
\begin{align*}
    \Pr\left( \left| p^{-1}Y^\T A Y - p^{-1}\tr(A) \right| \ge \varepsilon\|A\| \right) 
    &\le \Pr(Y\notin \Bc) + \Pr\left( \left| p^{-1}Y_\Bc^\T A Y_\Bc - p^{-1}\tr(A) \right| \ge \varepsilon\|A\| \right) \\
    &\le \Pr(Y\notin \Bc) + \Pr\left( \left| p^{-1}Y_\Bc^\T A Y_\Bc - \E\left[ p^{-1}Y_\Bc^\T A Y_\Bc \right] \right| \ge (\varepsilon-\varepsilon_0)\|A\| \right) \,,
\end{align*}
where $\epsilon_0 = \|A\|^{-1} p^{-1}\left| \E Y_\Bc^\T A Y_\Bc - \E Y^\T A Y \right|$.
% \[
% \varepsilon_0 := \frac{\left| \E\left[ \frac{1}{p}Y_\Bc^\T A Y_\Bc \right] - \frac1p\tr(A) \right|}{\|A\|} = \frac{\left| \E\left[ \frac{1}{p}Y_\Bc^\T A Y_\Bc \right] - \E\left[ \frac{1}{p}Y^\T A Y \right] \right|}{\|A\|} \,.
% \]
Since $\E\|Y\|\le \sqrt{p}$, 
Lemma~\ref{lem:log-concave-lip-concentration} implies that $\Pr(Y\notin \Bc)\le e^{-c_1\Cheeger\sqrt{p}}$. Observe that $Y_\Bc$ is a log-concave random vector, being the restriction of $Y$ onto a \emph{convex} set. Moreover, for any $u\in \Sp$,
\[
\mathrm{Var}(u^\T Y_{\Bc}) \le \E(u^\T Y_{\Bc})^2 \le \frac{\E(u^\T Y)^2}{\Pr(Y\in \Bc)} = 1 + O(e^{-c\Cheeger\sqrt{p}})=O(1) \,, 
\]
hence $\|\mathrm{Cov}(Y_\Bc)\| = O(1) $. Since the function $y\mapsto p^{-1}y^\T A y$ is $L=O(\|A\|p^{-1/2})$-Lipschitz on $\Bc$, Lemma~\ref{lem:log-concave-lip-concentration} implies
% that\footnote{Note that the bound holds even when $\varepsilon<\varepsilon_0$; it is vacuous in that case.}
\begin{align*}
    \Pr\left( \left| p^{-1}Y_\Bc^\T A Y_\Bc - \E\left[ p^{-1}Y_\Bc^\T A Y_\Bc \right] \right| \ge (\varepsilon-\varepsilon_0)\|A\| \right) 
    \le e^{-c \Cheeger \frac{(\varepsilon-\varepsilon_0)\|A\|}{L}} \le e^{-c_2\Cheeger \sqrt{p}(\varepsilon-\varepsilon_0)} \,,
\end{align*}
so that 
\[
\Pr\left( \left| p^{-1}Y^\T A Y - p^{-1}\tr(A) \right| \ge \varepsilon\|A\| \right)  \le e^{-c_1\Cheeger \sqrt{p}} + e^{-c_2\Cheeger \sqrt{p}(\varepsilon-\varepsilon_0)} \le C_3e^{-c_3\Cheeger \sqrt{p}(\varepsilon-\varepsilon_0)} \,.
\]
% for some $c_3,C_3>0$. 
It remains to show that $\varepsilon_0 = O( (\Cheeger\sqrt{p})^{-1})$. Decompose $\E\left[ Y^\T A Y \right] = \E\left[ Y_\Bc^\T A Y_\Bc \right]\Pr(Y\in \Bc) + \E\left[ Y^\T A Y \cdot \indic{Y\notin \Bc} \right]$, so
% \[
% \E\left[ Y^\T A Y \right] = \E\left[ Y_\Bc^\T A Y_\Bc \right]\Pr(Y\in \Bc) + \E\left[ Y^\T A Y \cdot \indic{Y\notin \Bc} \right] \,,
% \]
% we have
\begin{align*}
    \varepsilon_0 
    \le  \|A\|^{-1}{ p^{-1}\left| \E\left[ Y_\Bc^\T A Y_\Bc \right] \right|\Pr(Y\notin \Bc)} + \|A\|^{-1}p^{-1}{\left| \E\left[ Y^\T A Y  \indic{Y\notin \Bc} \right] \right|}
    % \\
    % &\le \E[\|Y_\Bc\|^2] \Pr(Y\notin \Bc) + \E[\|Y\|^2 \cdot \indic{Y\notin \Bc}] 
    \le O(e^{-c_1\Cheeger\sqrt{p}}) + p^{-1}\E\left[\|Y\|^2  \indic{\|Y\|^2>4p} \right] \,.
\end{align*}
It remains to bound the second term above. Use
\begin{align*}
    \E\left[\|Y\|^2 \indic{\|Y\|^2>4p} \right] 
    &= \int_{0}^\infty \Pr\left( \|Y\|^2  \indic{\|Y\|^2>4p} \ge t \right)dt  
     \\
    &= 4p \Pr(\|Y\|^2\ge 4p) + \int_{4p}^{\infty} \Pr(\|Y\|^2\ge t)dt 
    = O(pe^{-c_1\Cheeger\sqrt{p}}) + \int_{4p}^{\infty} \Pr(\|Y\|\ge \sqrt{t})dt \,.
\end{align*}
By Lemma~\ref{lem:log-concave-lip-concentration}, 
% and since $\E[\|\Y\|]\le \sqrt{p}$, 
$\Pr(\|Y\|\ge \sqrt{t})\le e^{-c_4\Cheeger(\sqrt{t}-\sqrt{p})}$. Moreover, when $\sqrt{t}\ge 2\sqrt{p}$, we have $\sqrt{t}-\sqrt{p}\ge \frac12\sqrt{t}$. Thus,
\[
\int_{4p}^{\infty} \Pr(\|Y\|\ge \sqrt{t})dt \le \int_{4p}^{\infty} e^{-(c_4/2)\Cheeger\sqrt{t}} dt \le e^{-(c_4/4)\Cheeger\sqrt{4p}}\int_{4p}^{\infty} e^{-(c_4/4)\Cheeger\sqrt{t}} dt = O(e^{-c_5\Cheeger\sqrt{p}})\,,
\]
and we are done. 

\subsection{Relaxing the zero mean assumption}
\label{sec:appendix-zero-mean}

As described in Section~\ref{sec:discussion}, we considered the symmetrization procedure of \cite{dumbgen1998tyler} to relax the zero mean assumption. We note that under the elliptical model, with $Y$ uniform on the sphere, this procedure is especially appealing, as  
the scaled difference
$(zY-z'Y')/R$ with $R=\sqrt{z^2+z'^2}$
is also uniformly distributed on the sphere. 
%$\tilde{X}^{\mathrm{sym}}=\tilde{X}-\tilde{X}'=\trueCov(z Y - z' Y')$ that preserves the elliptical structure. Specifically, assuming that $z$ has a density, taking $R=\sqrt{z^2+{z'}^2}$ one may verify that  $(z Y - z' Y)/R$ is distributed uniformly on the sphere. Unfortunately, under the proposed generalization of the elliptical model, that is, when $Y$ is not a uniform on the sphere, the symmetrization procedure does not a priori preserve the favorable elliptical structure of the measurements. 

Here we show that our main results continue to hold under a data distribution of the form $X=\trueCov^{1/2}Y^{\circ}$, where $Y^{\circ}=\zeta Y + \zeta' Y'$ and $\zeta=z/R$, $\zeta'=-z'/R$.  
By construction, the random vector $Y^{\circ}$ is isotropic; however, since $z,z'$ are arbitrary, $Y^{\circ}$ in general does not inherit the favorable distributional properties of $Y$. Fortunately, our analysis does not require these properties in their full detail. In fact, to carry out the proofs, it suffices to verify that $Y^{\circ}$ satisfies the following:

\begin{itemize}
    \item {\it Small-ball property:} $Y^{\circ}$ satisfies the SBP. To see this, observe that w.p. 1, either $|\zeta|\ge 1/\sqrt{2}$ or $|\zeta'|\ge 1/\sqrt{2}$ (because $\zeta^2+{\zeta'}^2=1$). Condition on $\zeta,\zeta'$ and  assume w.l.o.g. that $|\zeta|\ge 1/\sqrt{2}$. Then $|Y^{\circ}-a|\le t$ implies that $|Y-(a-\zeta'Y')/\zeta|\le t/|\zeta| \le \sqrt{2}t$. Since $(a-\zeta'Y')/\zeta$ is independent of $Y$, $\Pr(|Y^{\circ}-a|\le t|\zeta,\zeta') \le \Pr(|Y-(\zeta'Y'+a)/\zeta|\le  \sqrt{2}t|\zeta,\zeta')\le \sqrt{2}C_0 t$, where $C_0$ is the small-ball constant of $Y$.
    
    \item {\it Eigenvalue bounds for the sample covariance:} Let $S^{\circ}=n^{-1}\sum_{i=1}^n \y_i^{\circ}{\y_i^\circ}^\T$ be the sample covariance matrix of $n$ $Y^{\circ}$-distributed measurements. Also denote $S=n^{-1}\sum_{i=1}^n \y_i\y_i^\T$, $S'=n^{-1}\sum_{i=1}^n \y'_i{\y'}_i^\T$. 
    
    First, we need a high-probability bound on $\lambda_{\max}(S^\circ)$. Observe that for any $u$, by Cauchy-Schwartz, $(u^\T Y^{\circ})^2 = (\zeta u^\T Y + \zeta' u^\T Y)^2 \le (\zeta^2+{\zeta'}^2)( (u^\T Y)^2 + (u^\T Y')^2) = (u^\T Y)^2 + (u^\T Y')^2$. Consequently, $\lambda_{\max}(S^\circ) \le \lambda_{\max}(S) + \lambda_{\max}(S')$, which may be bounded w.h.p. using Lemma~\ref{lem:main-smax}.
    Next, when $\gamma<1$ we need a high-probability lower bound on $\lambda_{\min}(S^\circ)$. To this end, one can follow the proof of Lemma~\ref{lem:main-smin}.  To carry it out, we needed two components: the SBP, and a high-probability upper bound on $\lambda_{\max}(S^\circ)$; as explained, both hold.
    
    \item {\it Concentration for quadratic forms:}
    While complicated functions of $Y^\circ$ should not be expected to concentrate, since $\zeta,\zeta'$ are  arbitrary, concentration of quadratic forms is maintained due to their bilinear nature. 
    We need to prove an analog of Lemma~\ref{lem:main-quadratic}. Note that for \emph{fixed} $\zeta,\zeta'$, the random vector $\zeta Y + \zeta' Y$ inherits the favorable concentration properties of $Y$. Since the conditional expectation of a quadratic form does not depend on $\zeta,\zeta'$, $\E[{Y^\circ}^\T A Y^{\circ}|\zeta,\zeta']=\tr(A)$, we may simply apply Lemma~\ref{lem:main-quadratic} pointwise conditioned on $\zeta,\zeta'$.
    
    % For a matrix $A$, we need to show that $p^{-1}{Y^\circ}^\T A Y^{\circ}$ concentrates around its expectation $p^{-1}\tr(A)$. Indeed, we may decompose $p^{-1}{Y^\circ}^\T A Y^{\circ} = p^{-1}\zeta^2 Y^\T A Y + p^{-1}{\zeta'}^2 {Y'}^\T A {Y'} + p^{-1}\zeta \zeta' Y^\T A Y' + p^{-1}\zeta'\zeta {Y'}^\T A Y$. By Lemma~\ref{lem:main-quadratic}, $p^{-1}Y^\T A Y, p^{-1}{Y'}^\T A Y'$ concentrate tightly around $p^{-1}\tr(A)$, whereas $p^{-1}{Y'}^\T A {Y'}, p^{-1} Y^\T A Y'$ concentrate around $0$ (since $Y,Y'$ are independent). Thus, $p^{-1}{Y^\circ}^\T A$ concentrates around $(\zeta^2+{\zeta'}^2)p^{-1}\tr(A)=p^{-1}\tr(A)$. Alternatively, one could arrive at this result by conditioning on $\zeta,\zeta'$, noting that the conditional expectation (taken with respect to $Y,Y'$) does not depend on $\zeta,\zeta'$. 
    
    \item {\it Entrywise concentration for the sample covariance:} We need an analog of Lemma~\ref{lem:sc-entrywise}. We may carry out the proof of Lemma~\ref{lem:sc-entrywise}, essentially verbatim, conditioned on $\{\zeta_{i},\zeta_{i}'\}_{1\le i \le n}$ and noting that $\E[ u^\T S^{\circ} v | \{\zeta_{i},\zeta_{i}'\}_{1\le i \le n}] = u^\T v$ does not depend on $\{\zeta_{i},\zeta_{i}'\}_{1\le i \le n}$. 
\end{itemize}
\bibliographystyle{myjmva} % Style BST file (imsart-number.bst or imsart-nameyear.bst)
\bibliography{refs}       % Bibliography file (usually '*.bib')

\end{document}